\documentclass[12pt]{article}
\usepackage{amssymb,amsfonts,amsmath,latexsym,amsthm}
\usepackage[active]{srcltx}

\numberwithin{equation}{section}

\newcommand{\n}{^{(1)}}
\newcommand{\m}{^{(2)}}

\newtheorem{pred}{Theorem}
\newtheorem{predl}{Lemma}

\newtheorem{predll}{Definition}

\begin{document}

\sloppy

\renewcommand{\proofname}{Proof}
\renewcommand{\abstractname}{Abstract}

\thispagestyle{empty}
\begin{center}
{\textbf{MIXED VARIATIONAL APPROACH TO FINDING GUARANTEED ESTIMATES FROM SOLUTIONS AND RIGHT-HAND SIDES OF
THE SECOND-ORDER LINEAR ELLIPTIC EQUATIONS UNDER INCOMPLETE DATA}}
\end{center}
\bigskip\bigskip\bigskip\bigskip\bigskip\bigskip\bigskip\bigskip\bigskip\bigskip\bigskip

\begin{center}
%{\bf 
	Yuri Podlipenko$^1$, \\[3pt] Yury Shestopalov$^2$ \\
	%}
\end{center}

\begin{center}
$^1$ %{\large {\it {
			National Taras Shevchenko University of Kiev, Kiev, Ukraine
			%}}},
			\\ yourip@mail.ru\\
$^2$
%{\large {\it {
%G$\ddot a$\!vle University, G$\ddot a$\!vle, Sweden
G\"avle University, G\"avle, Sweden
%}}},
\\ yuyshv@hig.se
\end{center}

\renewcommand{\contentsname}%
{\begin{center}Contents\end{center}}

\newpage

%\addtocontents{toc}{\large}

\tableofcontents

\newpage

%\large

\begin{abstract}
We 
%first 
investigate the problem of guaranteed estimation
of values of linear continuous functionals
defined on solutions
%and right-hand sides
to mixed variational equations generated by
%Dirichlet boundary value
linear elliptic problems
%for the second order linear elliptic equations
from indirect noisy observations of these solutions.

We assume that right-hand sides of the equations, as well as the second moments of noises in observations
are not known; the only available information is that they belong to given bounded sets in the appropriate functional spaces.
%The right-hand sides of
%these
%variational equations as well as the second moments of noises in observations are assumed to be unknown and subjected
%to some quadratic restrictions.

We are looking for linear with respect to observations optimal estimates
of solutions
%and right-hand sides
of aforementioned equations called minimax or guaranteed estimates. We develop constructive methods for finding these estimates and
estimation errors which are expressed in terms of solutions to special mixed variational
equations and prove that Galerkin approximations of the obtained variational
equations converge to their exact solutions.

We study also  the problem of guaranteed estimation
of right-hand sides of mixed variational equations.
\end{abstract}

%\large

\section*{}

\addcontentsline{toc}{section}{Introduction}

\begin{center}
{\textbf{Introduction}}
\end{center}

%{\bf 1 Introduction}

Estimation theory for systems with lumped and distributed parameters under
uncertainty conditions was developed intensively during the last 30 years. That
was motivated by the fact that the realistic setting of boundary value problems
describing physical processes often contains perturbations of unknown (or partially
unknown) nature. In such cases the minimax estimation method proved to be
useful, making it possible to obtain optimal estimates both for the unknown
solutions (or right-hand sides of equations appearing in the boundary value
problems) and for linear functionals from them, that is estimates looked for in
the class of linear estimates with respect to observations
%\footnote{
	Here we understand observations of unknown solutions as the functions that are linear transformations of same solutions distorted by additive random noises.
	%}
	%distorted by additive random noises.},
for which the maximal
mean square error taken over all the realizations of perturbations from certain
given sets takes its minimal value.

The above estimation method was investigated in the works by N. N. Krasovsky,
A. B. Kurzhansky, A. G. Nakonechny, and others (see \cite{BIBLkras}, \cite{BIBLkurzh}, \cite{BIBLnak1}--\cite{nak2}, \cite{Pod}). This approach makes it possible to find optimal
estimates of parameters of boundary value problems reckoning on the "worst"
realizations of perturbations.

A. G. Nakonechny used traditional variational formulations of boundary value
problems (their solvability is based on the Lax-Milgram lemma), to obtain systems
of variational equations whose solutions generate the minimax mean square
estimates.

At the same time many physical processes of the real world are described by
mixed variational problems. Among such processes, there are flows of viscous fluids,
propagation of electromagnetic and acoustical waves. In addition, many classical
boundary value problems admit mixed variational formulations. The mixed method
consists of simultaneous finding, from systems of variational equations, both
solutions and certain expression generated by solutions taken as new auxiliary
unknowns. As a rule, these unknowns are related to derivatives of the solutions
and have important physical meaning (such as flux, bending moment etc), and
their calculation or estimation often has even greater practical significance.

The theory of mixed variational methods of solving boundary value problems
and their numerical implementation, the mixed finite element methods, was
developed by Babu\v{s}ka, Brezzi, Fortin, Raviard, Glowinski and others (see \cite{Brezzi}
%, \cite{BIBLNedelec},  \cite{BIBLRaviart},
-\cite{Gatica}). In particular,
Brezzi and Fortin proved solvability theorems for a wide class of mixed variational
problems and their discrete analogs.

In this paper we show that mixed variational formulations of boundary value
problems can be used also for a guaranteed estimation of linear functionals from unknown
solutions and their gradients, as well as functionals from unknown right-hand sides
of second order linear elliptic equations. It is proved that guaranteed estimates of
these functionals and estimation errors are expressed explicitly from the solutions
of special systems of mixed variational equations, for which the unique solvability
is established. We develop, on the basis of the Galerkin method, numerical methods
of finding these solutions and prove the convergence of the approximate solutions
to exact ones.

The estimation methods proposed here yield, for example, in stationary and
non-stationary heat conduction problems, estimates of heat flux from temperature
observations, or conversely, estimates of temperature from heat flux observations,
as well as estimates of the unknown distribution of density of sources from heat flux
observations. The theory of guaranteed estimation developed in this work provides
an essential generalization of well-known results in this direction by the authors
mentioned above.

Note that the available estimation methods do not provide solution of such
estimation problems, so that the methods developed here are essentially new.

%\subsection
%{\bf 2 Preliminaries and auxiliary results}
\section{Preliminaries and auxiliary results}

Let us introduce the notations and definitions that will be used in this work.

We denote matrices
%by bold capital letters
and vectors by bold
%lower case
letters;
$x=(x_1,\dots ,x_n)$ denotes a spatial variable in an open domain $D \subset \mathbb R^n$  with
Lipschitzian boundary $\Gamma;$ $dx=dx_1\cdots dx_n$ is Lebesgue measure in $\mathbb R^n;$ $H^1(D)$ and $H_0^1(D)$ are standard Sobolev spaces of the first order
in the domain $D$ with corresponding norm.

If $X$ is a Hilbert space over $\mathbb R$
with inner product $(\cdot,\cdot)_{X}$ and norm
$\|\cdot\|_{X},$ then $J_{X}\in \mathcal L(X,X')$ denotes the Riesz operator acting from $X$ to its adjoint
$X'$ and determined by the equality (we note that this operator  exists according to the Riesz theorem)
$
(v,u)_{X}=$ $<v,J_{X} u>_{X\times X'}$ $\forall u,v\in
X,$ where $<x, f>_{X\times X'}$ $:=f(x)$ for $x\in X,$ $f\in
X'.$

Below random variable $\xi$ with values in a separable Hilbert space $X$ is considered as a function $\xi:\Omega\to X$ mapping random events  $E\in\mathcal B$
to Borel sets in $H$ (Borel $\sigma$-algebra in $X$
is generated by open sets in $X$). By $L^2(\Omega,X)$ we denote the Bochner space composed of random
 variables $\xi=\xi(\omega)$ defined on a certain probability space $(\Omega, \mathcal B, P)$ with values in
a separable Hilbert space $X$ such that
\begin{equation}\label{rand}
\|\xi\|_{L^2(\Omega,X)}^2
=\int_{\Omega}\|\xi(\omega)\|_X^2dP(\omega)<\infty.
\end{equation}
In this case there exists the Bochner integral
\begin{equation}\label{rand'}
\mathbb
E\xi:=\int_{\Omega}\xi(\omega)\,dP(\omega)\in X
\end{equation}
called the mathematical expectation or the mean value of random variable
$\xi(\omega)$ which satisfies the condition
\begin{equation}\label{rand1}
(h,\mathbb
E\xi)_X=\int_{\Omega}(h,\xi(\omega))_X\,dP(\omega)\quad\forall
h\in X.
\end{equation}
Being applied to random variable $\xi$ with values in $\mathbb R$ this expression leads to a usual definition of its mathematical expectation because the Bochner integral \eqref{rand'} reduces to a Lebesgue integral with probability measure $dP(\omega).$

In $L^2(\Omega,X)$ one can introduce the inner product
\begin{equation}\label{rand2}
(\xi,\eta)_{L^2(\Omega,X)}:=\int_{\Omega}(\xi(\omega),
\eta(\omega))_X\,dP(\omega)\quad\forall\xi,\eta\in L^2(\Omega,X).
\end{equation}
Applying the sign of mathematical expectation, one can write relationships \eqref{rand}$-$\eqref{rand2} as
\begin{equation}\label{rand3}
\|\xi\|_{L^2(\Omega,X)}^2 =\mathbb E\|\xi(\omega)\|_X^2,
\end{equation}
\begin{equation}\label{rand4}
(h,\mathbb E\xi)_X=\mathbb E(h,\xi(\omega))_X\quad\forall h\in X,
\end{equation}
\begin{equation}\label{rand5}
(\xi,\eta)_{L^2(\Omega,X)}:=\mathbb E(\xi(\omega),
\eta(\omega))_X\quad\forall\xi,\eta\in L^2(\Omega,X).
\end{equation}
$L^2(\Omega,X)$ equipped with norm \eqref{rand3} and inner product \eqref{rand5} is a Hilbert space.

%\subsection
%{\bf 3 Statement of the estimation problem of linear functionals from solutions to mixed variational equations}
\section{Statement of the estimation problem of linear functionals from solutions to mixed variational equations}

Let the state of a system be characterized by the function
$\varphi(x)$ which is defined as a solution of the Dirichlet boundary value problem:
\begin{equation} \label{1}
-\mbox{div\,}(\mathbf A\,\mbox{\bf grad\,}\varphi)+c\varphi=f\quad \mbox
{in}\quad D,
\end{equation}
\begin{equation} \label{3a}
\varphi=0 \quad \mbox {on}\quad \Gamma.
\end{equation}
Introducing the additional unknown $\mathbf j=-\mathbf A\,\mbox{\bf grad\,}\varphi$ in $D,$
%Making the substitution $\mathbf j=-\mathbf A\,\mbox{\bf grad\,}\varphi,$
rewrite this problem as the first-order system
%\begin{equation}\label{val}
%$\mathbf j=-\mathbf A\,\mbox{\bf grad\,}\varphi,$
%\end{equation}
%можна записати у вигляді еквівалентної системи першого порядку:
\begin{equation}\label{eqsystem1}
\mathbf A^{-1}\,\mathbf j=-\mbox{\bf grad\,}\varphi\quad \mbox
{in}\quad D,
\end{equation}
\begin{equation}\label{eqsystem2}
\mbox{div\,}\mathbf j+c\varphi=f\quad \mbox {в}\quad D,\quad
%\end{equation}
%\begin{equation} \label{eqsystem3}
\varphi=0 \quad \mbox {on}\quad \Gamma,
\end{equation}
where $\mathbf A=\mathbf A(x)=(a_{ij}(x))$ is an $n\times n$ matrix with entries $a_{ij}\in L^{\infty}(D)$
for which there exists a positive number
$\mu$ such that
\begin{equation*}\label{matrb}
\mu\sum_{i=1}^n\xi_i^2\leq\sum_{i,j=1}^na_{ij}(x)\xi_i\xi_j \quad\forall x\in
D\quad \forall \xi=(\xi_1,\ldots,\xi_n)^T\in\mathbb R^n,
\end{equation*}
$\mathbf A^{-1}$ is the inverse matrix of $\mathbf A,$
and $c$ is a piecewise continuous function satisfying for $x\in D$ the inequality $c_0\leq c(x) \leq c_1,$
$c_0,c_1=\mbox{const},\,0\leq c_0\leq c_1.$
%У відповідності з \cite{Brezzi},

According to \cite{Brezzi} and \cite{Hoppe}, by a solution of problem \eqref{eqsystem1}, \eqref{eqsystem2} we will mean a pair of functions
%\eqref{eqsystem1}--\eqref{eqsystem2}\, будемо розуміти пару
$(\mathbf j,\varphi)\in H(\mbox{div};D)\times L^2(\Omega)$ such that
%задовольняє співвідношенням
\begin{equation}\label{generalized}
\int_D\!((\mathbf A(x))^{-1}\mathbf j(x),\mathbf q(x))_{\mathbb
R^n}dx-\int_D\varphi(x)\mbox{div}\, \mathbf
q(x)dx=0\quad \forall \mathbf q\in H(\mbox{div};D)
\end{equation}
\begin{equation}\label{generalized1}
\int_Dv(x)\mbox{div}\,\mathbf j(x)dx+\int_Dc(x)\varphi (x) v(x)dx=\int_Df(x)v(x)\,dx\,\,\forall v\in
L^2(D).
\end{equation}
Note that from equations \eqref{generalized} and \eqref{generalized1} it follows that
$\varphi\in H_0^1(D),$ i. e. the boundary condition $\varphi|_{\Gamma}=0$
is implicitly contained in these equations.

Problem \eqref{generalized}, \eqref{generalized1} is commonly referred to as the mixed formulation of \eqref{eqsystem1}, \eqref{eqsystem2}.

From physical point of view problem
%\eqref{1}--\eqref{3a}, або еквівалента до неїзадача
\eqref{eqsystem1}, \eqref{eqsystem2} simulates a stationary process of the propagation of heat
%mоделює усталений процесс розповсюдження
in the domain $D,$ and the functions $\varphi(x),$ $\mathbf j(x),$ and $f(x)$ have the sense of temperature, heat flux, and volume density of heat sources, respectively,
 at the point $x.$

Introduce the bilinear forms $a,b,c$ and the functional $l$ given by
\begin{equation} \label{a}
a(\mathbf q_1,\mathbf q_2):=\int_D ((\mathbf A(x))^{-1}\mathbf q_1,\mathbf q_2)_{\mathbb R^n}dx \quad \forall \mathbf q_1,\mathbf q_2\in H(\rm div;D),
\end{equation}
\begin{equation} \label{b}
b(\mathbf q,v):= -\int_D v\,\mbox{\rm div}\,\mathbf q\,dx \quad \forall \mathbf q\in H(\rm div;D), v\in L^2(D),
\end{equation}
\begin{equation} \label{c}
c(v,v):=\int_D c(x)v_1(x)v_2(x)\,dx\quad \forall v_1,v_2\in L^2(D),
\end{equation}
\begin{equation} \label{l}
l(v):=-\int_D fv\,dx \quad\forall v\in L^2(D).
\end{equation}
Then the problem under study may be stated as follows.

Find $(\mathbf j,\varphi)\in H(\mbox{div};D)\times L^2(\Omega)$ such that
\begin{equation} \label{eqn2}
a(\mathbf j,\mathbf q)+b(\mathbf q,\varphi)=0\quad \forall \mathbf q \in H(\mbox{div};D),
\end{equation}
\begin{equation} \label{eqn2'}
b(\mathbf j,v)- c(\varphi,v)=l(v)\quad\forall v\in L^2(D).
\end{equation}

Denote by $B:H(\rm div;D)\to L^2(D)$ the operator associated with the bilinear form $b$.
%defined by $$<Bv,q>_{L^2(D)'\times L^2(D)}=b(v,q)\quad\forall v\in L^2(D), \quad\forall q\in L^2(D).$$
It is easy to see that $a,b$ and $c$ are continuous bilinear forms with $a$ being
coercive on $\rm Ker\,B$, $c$ being symmetric, positive semidefinite and $b$ satisfying the standard
inf-sup condition (Brezzi condition). Since $\mbox{\rm Im}\,B=L^2(D)$, we have $\mbox{\rm Ker}\, B^t
=\{\emptyset\},$ where $B^t
%=\mbox{-grad}:
:L^2(D)\to H(\mbox{div};D)'$ is the transpose operator of $B$ defined by 
$$
<v,B^tq>_{H(\mbox{\rm div},D)\times H(\mbox{\rm\small div},D)'}=b(v,q)\quad\forall v\in H(\mbox{\rm div},D), \quad\forall q\in L^2(D).
$$
Consequently, it follows from Theorem 1.2 of \cite{Brezzi} that problem \eqref{eqn2}, \eqref{eqn2'} has a unique solution and the following {\textit {a priori}} estimate
is valid
\begin{equation*}\label{aprest}
\|\mathbf j\|_{H(\mbox{\small div};D)}+\|\varphi\|_{L^2(D)}\leq C\|f\|_{L^2(D)}  \quad (C=\mbox{const}).
\end{equation*}

Further we assume  that the function $f(x)$ in equations
\eqref{eqsystem2} and
\eqref{generalized1} is unknown and belongs to the set
\begin{equation}\label{h11h}
G_0:=\Bigl\{\tilde f\in  L^2(D) :
\Bigl( Q(\tilde f-f_0),\tilde
f-f_0\Bigr)_{L^2(D)} \leq \epsilon_1\Bigr\},
\end{equation}
where $f_0\in L^2(D)$ is a given function, $\epsilon_1>0$ is a given constant,\label{s5}
%\begin{equation}\label{3r3}
%\int_{D}q_0(x)\,dx=0,
%\end{equation}
and $Q:L^2(D)\to L^2(D)$  is a bounded
positive selfadjoint operator for which there exists the inverse bounded operator $Q^{-1}.$
It is known that the operator $Q^{-1}$ is  positive and selfadjoint.

% Як відомо (\cite{temam}, cт. 34), для кожної пари функцій
% $(\mathbf f,q)\in L^2(D)^n\times \widetilde H^1(D)$
% існують функції $\mathbf v\in H^2(D)^n$  і $p\in
%H^1(D),$ що задовольняють\label{ggj} (\ref{1}) --
%(\ref{3})\footnote{Відмітимо, що коли $(\mathbf v_1,p_1)$ і
%$(\mathbf v_2,p_2)$ -- два розв'язки задачі \eqref{1}--\eqref{3},
%то $\mathbf v_1=\mathbf v_2$ і $p_1-p_2=c=\mbox{const}.$}
% майже всюди в $D.$ Крім того, існуе
%константа $c_0=c_0(\nu,D),$ така що
%\begin{equation}\label{cont}
%\|\mathbf v\|_{H^2(D)^n}+\|p\|_{H^1(D)/\mathbb R} =\|\mathbf
%v\|_{H^2(D)^n}+\inf_{a\in\mathbb R}\|p+a\|_{H^1(D)}\leq
%c_0\left(\|\mathbf f\|_{L^2(D)^n}+\|q\|_{H^1(D)}\right).
%\end{equation}
% причому $\mathbf v$ єдина, а $p$
%єдина з точністю до адитивної сталої.
% \thispagestyle{empty}
% \makeatletter
%\renewcommand{\@oddfoot}{}
In this paper we focus on the following estimation problem:

From observations of random variables
\begin{equation}\label{observ}
y_1=C_1\mathbf j+\eta_1, \quad
y_2=C_2\varphi+\eta_2,
\end{equation}
with values in separable Hilbert spaces $H_1$ and $H_2$ over
$\mathbb R,$ respectively, it is necessary to estimate the value of the linear functional
\begin{equation}\label{linf}
l(\mathbf j, \varphi):=  \int_{D}(\mathbf l_1(x),\mathbf
j(x))_{\mathbb R^n}\, dx+\int_{D}l_2(x)\varphi(x)\, dx
\end{equation}
in the class of the estimates linear with respect to observations, which have the
form
\begin{equation}\label{clas}
\widehat{l(\mathbf j, \varphi)}:=(y_1,u_1)_{H_1}
+(y_2,u_2)_{H_2}+c,
\end{equation}
where $(\mathbf j,\varphi)$ is a solution of problem
\eqref{generalized}, \eqref{generalized1}, $\mathbf l_1$ and $l_2$
are given functions from $L^2(D)^n$ and $L^2(D),$ $u_1\in H_1,$ $u_2\in
H_2,$ $c\in \mathbb R,$ $C_1\in\mathcal L(L^2(D)^n,H_1)$, and
$C_2\in\mathcal L(L^2(D),H_2)$ are linear bounded operators,
\begin{equation}\label{etaG_1}
\eta:=(\eta_1,\eta_2)\in G_1;
\end{equation}
by $G_1$ we denote the set of pairs $\{(\tilde\eta_1,\tilde\eta_2)\}$ of uncorrelated 
random variables
 $\tilde \eta_1\in L^2(\Omega,H_1)$ and $\tilde
\eta_2\in L^2(\Omega,H_2)$ with zero expectations satisfying the condition
\begin{equation}\label{restr2}
\mathbb E(\tilde Q_1\tilde \eta_1,\tilde \eta_1)_{H_1} \leq \epsilon_2,\quad \mathbb
E(\tilde Q_2\tilde \eta_2,\tilde \eta_2)_{H_2} \leq \epsilon_3,
\end{equation}
 where
$\tilde Q_1$ and $\tilde Q_2$ are bounded positive-definite selfadjoint operators in $H_1$ and $H_2,$ respectively, for which there exist the inverse bounded operators $\tilde Q_1^{-1}$ and $\tilde Q_2^{-1}.$ \label{s6}
We note that random variables $\xi_1\in H_1$ and $\xi_2\in
	H_2$ are called uncorrelated if
	\begin{equation}\label{nekor}
		\mathbb E(\xi_1,u_1)_{H_1}(\xi_2,u_2)_{H_2}=0\quad\forall u_1\in H_1, u_2 \in H_2
	\end{equation}
	(see, for example, \cite{Vah}, p. 146).

It is known that  operators $\tilde Q_1^{-1}$ and $\tilde Q_2^{-1}$ are  positive definite and selfadjoint, that is there exists a positive number $\alpha$ such that
\begin{equation}\label{posdefme}
\!\!(\tilde Q_1^{-1}u_1,u_1)_{H_1}\geq \alpha \|u_1\|_{H_1}^2\,\,\forall u_1\in H_1, \,\,
(\tilde Q_2^{-1}u_2,u_2)_{H_2}\geq \alpha \|u_2\|_{H_2}^2\,\,\forall u_2\in H_2.
\end{equation}
Set $u:=(u_1,u_2)\in H:=H_1\times H_2$.

\begin{predll} An estimate
\begin{equation*}\label{defme}
\widehat{\widehat {l(\mathbf j,\varphi)}}=(y_1,\hat u_1)_{H_1}+(y_2,\hat
u_2)_{H_2}+\hat c
\end{equation*}
is called a guaranteed (or minimax) estimate of $l(\mathbf j,\varphi),$ if elements $\hat u_1\in H_1,$ $\hat u_2\in H_2$ and a number $\hat c$
are determined from the condition
\begin{equation*}\label{m16}
\inf_{u\in H,\,c\in \mathbb R}\sigma(u,c)=\sigma(\hat u,\hat c),
\end{equation*}
where 
$$
\sigma(u,c):=\sup_{\tilde f \in
G_0, (\tilde \eta_1,\tilde \eta_2)\in G_1}\mathbb E|l(\mathbf
{\tilde j},\tilde \varphi)-\widehat {l(\mathbf {\tilde j},\tilde
\varphi)}|^2,
$$
and
$(\mathbf {\tilde j},\tilde
\varphi)$ is a solution of problem
\eqref{eqsystem1}, \eqref{eqsystem2} at $f(x)=\tilde f(x)$,
$\widehat{l(\tilde {\mathbf j}, \tilde \varphi)}:=(\tilde y_1,u_1)_{H_1}
+(\tilde y_2,u_2)_{H_2}+c, \tilde y_1=C_1\tilde {\mathbf j}+\tilde \eta_1, \quad
\tilde y_2=C_2\tilde \varphi+\tilde \eta_2.$

The quantity
\begin{equation}\label{mm16}
\sigma:=[\sigma(\hat u,\hat c)]^{1/2}
% \varrho:=\{\mathbb
%E|l(\mathbf {j},\varphi)-\widehat{\widehat {l(\mathbf
%{j},\varphi)}}|^2\}^{1/2}
\end{equation} 
is called the error of the guaranteed estimation of $l(\mathbf {j},\varphi).$
\end{predll}

Thus, the guaranteed estimate is an estimate minimizing the maximal
mean-square estimation error calculated for the ``worst''
implementation of perturbations.

Further, without loss of generality, we may set
$\epsilon_k=1,$ $k=1,2,3,$ in \eqref{h11h} and \eqref{restr2}.

%\subsection
%{\bf 4 Reduction of the estimation problem to the optimal control problem of a system governed by mixed variational equations
\section{Reduction of the estimation problem to the optimal control problem of a system governed by mixed variational equations}
%Зведення задачі мінімаксного
%оцінювання
%функціоналів, визначених на розв'язках еліптичних рівнянь та на їх градієнтах,
%до задачі оптимального керування системою, що описується
%змішаними варіаційними рівняннями}
%{\bf Представлення для мінімаксних оцінок і похибок
%оцінювання.

Introduce a pair of functions
$
(\mathbf z_1(\cdot;u),z_2(\cdot;u))\in
H(\mbox{div};D)\times L^2(D)
$
%як єдиний pозв’язок наступної
%крайової задачі:
%\begin{equation}\label{13}
%\mathbf (A^{-1})^T\,\mathbf z_1(\cdot;u)+\mbox{\bf
%grad\,}z_2(\cdot;u)\\=\mathbf l_1-C_1^tJ_{H_1}u_1\quad \mbox
%{в}\quad D,
%\end{equation}
%\begin{equation}\label{14}
%\mbox{div\,}\mathbf z_1(\cdot;u)+cz_2(\cdot;u)=-(l_2-C_2^tJ_{H_2}u_2)\quad
%\mbox {в}\quad D,
%\end{equation}
%\begin{equation} \label{15}
%z_2(\cdot;u)=0 \quad \mbox {на}\quad \Gamma,
%\end{equation}
%під яким слід розуміти розв’язок наступної
as a solution of the problem:
\begin{multline}\label{113}
\int_{D}(((\mathbf A(x))^{-1})^T\mathbf z_1(x;u),\mathbf
q(x))_{\mathbb R^n}dx- \int_{D}z_2(x;u)\mbox{div\,}\mathbf
q(x)\,dx\\= \int_{D}(\mathbf l_1(x)-(C_1^tJ_{H_1}u_1)(x),\mathbf q(x))_{\mathbb
R^n}dx\quad \forall \mathbf q\in H(\mbox{div},D),
\end{multline}
\begin{multline}\label{115}
-\int_{D}v(x)\mbox{div\,}\mathbf
z_1(x;u)\,dx-\int_{D}c(x)z_2(\cdot;u)v(x)\,dx\\=\int_{D}(l_2(x)-(C_2^tJ_{H_2}u_2)(x))v(x)\,dx\quad
\forall v\in L^2(D),
\end{multline}
%\noindent
where $u\in H,$
$C_1^t:H_1'\to L^2(D)^n$ and $C_2^t:H_2'\to L^2(D)$ are the transpose operators of
$C_1$ and $C_2$, respectively, defined by
$\int_D(v(x),C_1^tw(x))_{\mathbb R^n}dx=<Cv,w>_{H_1\times H_1'}$
for all $v\in L^2(D)^n,$ $w\in H_1'$
and $\int_Dv(x)C_2^tw(x)\,dx=<Cv,w>_{H_2\times H_2'}$ for all $v\in
L^2(D),$ $w\in H_2'.$
%Iз \eqref{113} і \eqref{115} маємо $z_2\in H_0^1(D).$

From the theory of mixed variational problems it is known that
the pair $(\mathbf z_1(x;u),z_2(x;u))$ is uniquely determined
\footnote{In fact, note that problem \eqref{113}, \eqref{115} can be rewritten in the form
\begin{equation*} \label{eqn2a}
a^*(\mathbf z_1,\mathbf q)+b(\mathbf q,z_2)=l_1(\mathbf q)\quad \forall \mathbf q \in H(\rm div;D),
\end{equation*}
\begin{equation*} \label{eqn2'a}
b(\mathbf z_1;v)- c(z_2;v)=l_2(v)\quad\forall v\in L^2(D),
\end{equation*}
where $a^*(\mathbf z_1,\mathbf q)=a(\mathbf q,\mathbf z_1),$ the bilinear forms $a,$ $b,$ and $c,$ are defined by \eqref{a}, \eqref{b}, and \eqref{c}, respectively, $l_1(\mathbf q)=(\mathbf l_1-C_1^tJ_{H_1}u_1,\mathbf q)_{L^2(D)^n},$
$l_2(v)=(l_2-C_2^tJ_{H_2}u_2,v)_{L^2(D)}.$  Since $a^*(\mathbf q,\mathbf q)=a(\mathbf q,\mathbf q)$
then the bilinear form $a^*(\mathbf z_1,\mathbf q)$ is also coercive on $\rm Ker\,B$ and, hence, by Theorem 1.2 from \cite{Brezzi} problem \eqref{113}, \eqref{115} is uniquely solvable.
Moreover we have:
\begin{equation} \label{eqn2'ap}
\|\mathbf z_1\|_{H(\mbox{\rm \small div};D)}+\|z_2\|_{L^2(D)} \leq C(\|l_1\|_{H(\mbox{\small div};D)'}+\|l_2\|_{L^2(D)})
\leq C((\mathbf l_1-C_1^tJ_{H_1}u_1)_{L^2(D)^n}+(l_2-C_2^tJ_{H_2}u_2)_{L^2(D)}),
\end{equation}
where $C=$const.
}

\begin{predl}
The problem of guaranteed estimation of the functional $l(\mathbf j, \varphi)$
(i.e. the determination of $\hat u=(\hat u_1,\hat u_2)$ and $\hat c$) is equivalent to
the problem of optimal control of the system described by mixed variational problem (\ref{113}), (\ref{115}) with a cost function
\begin{equation}\label{m20}
I(u)=(Q^{-1}z_2(\cdot;u),
z_2(\cdot;u))_{L^2(D)}
+ (\tilde
Q_1^{-1}u_1,u_1)_{H_1}+(\tilde Q_2^{-1}u_2,u_2)_{H_2}\!\to
\inf_{u\in H}.
\end{equation}
\end{predl}
\begin{proof}
From relations \eqref{observ}--\eqref{clas} at $\mathbf j=\tilde{\mathbf j},$ $\varphi=\tilde \varphi,$ $\eta_1=\tilde \eta_1,$ $\eta_2=\tilde \eta_2,$ we have
%маємо при $(u_1,u_2)\in H_1\times H_2$
$$
l(\tilde{\mathbf j},\tilde
\varphi)-\widehat{l(\tilde{\mathbf j},\tilde \varphi)}=
(\mathbf l_1,\tilde{\mathbf j})_{L^2(D)^n}+(l_2,\tilde{\varphi})_{L^2(D)}
$$
$$
-(\tilde y_1,u_1)_{H_1}
-(\tilde y_2,u_2)_{H_2}-c
=(\mathbf l_1,\tilde{\mathbf j})_{L^2(D)^n}+(l_2,\tilde{\varphi})_{L^2(D)}
$$
$$
-(u_1,C_1\tilde{\mathbf j}+\tilde\eta_1)_{H_1}-(u_2,C_2\tilde
\varphi+\tilde\eta_2)_{H_2}-c
=(\mathbf l_1,\tilde{\mathbf j})_{L^2(D)^n}+(l_2,\tilde{\varphi})_{L^2(D)}
$$
$$
-<J_{H_1}u_1,C_1\tilde{\mathbf j}>_{H_1'\times H_1}-<
J_{H_2}u_2,C_2\tilde \varphi>_{H_2'\times H_2}
$$
$$
 -(u_1,\tilde\eta_1)_{H_1}-(u_2,\tilde\eta_2)_{H_2}-c
=(\mathbf l_1,\tilde{\mathbf j})_{L^2(D)^n}+(l_2,\tilde{\varphi})_{L^2(D)}
$$
$$
-
(C_1^tJ_{H_1}u_1,\tilde{\mathbf j})_{L^2(D)^n}
-(C_2^t
J_{H_2}u_2,\tilde \varphi)_{L^2(D)}
-(u_1,\tilde\eta_1)_{H_1}-(u_2,\tilde\eta_2)_{H_2}-c
$$
$$
=(\mathbf l_1-C_1^tJ_{H_1}u_1,\tilde{\mathbf j})_{L^2(D)^n}+(l_2-C_2^t J_{H_2}u_2,\tilde{\varphi})_{L^2(D)}
$$
\begin{equation}\label{gh1}
-(u_1,\tilde\eta_1)_{H_1}-(u_2,\tilde\eta_2)_{H_2}-c.
\end{equation}\label{fffg}
Futher, taking into account system of variational equations
\begin{equation} \label{16at1v}
\!\!\int_D(\mathbf
(\mathbf A(x))^{-1}\mathbf { \tilde j}(x),\mathbf q(x))_{\mathbb R^n}dx
-\int_D\tilde\varphi(x)\,\mbox{div}\,\mathbf q(x)\,dx=0
\,\,\,\forall \mathbf q\in H(\mbox{div};D),
\end{equation}
%\begin{multicols}{2}
\begin{equation} \label{14aat1v}
\int_Dv(x)\,\mbox{div}\,\mathbf {\tilde j}(x)\,dx+\int_{D}c(x)\tilde\varphi(x)v(x)\,dx=\int_D\tilde
f(x)v(x)\,dx \quad \forall v\in L^2(D),
\end{equation}
which follows from \eqref{generalized}--\eqref{generalized1} if we set there $\mathbf f=\tilde {\mathbf f},$
%\begin{multline}\label{13at1v}
%\int_D
%(((\mathbf A(x))^{-1})^T\mathbf z_1(x;u),\boldsymbol{\chi}_1(x))_{\mathbb R^n}dx
%-\int_Dz_2(x;u)\,\mbox{div}\,\boldsymbol{\chi}_1(x)\,dx
%\\=\int_D(\mathbf l_1(x)-(C_1^tJ_{H_1}u_1)(x),\boldsymbol{\chi}_1(x))_{\mathbb R^n}
%\,dx\quad\forall\boldsymbol{\chi}_1\in H(\mbox{div};D),
%\end{multline}
%\begin{multline}\label{14at1v}
%-\int_D\chi_2(x)\,\mbox{div}\,\mathbf z_1(x;u)\,dx-\int_{D}c(x)z_2(\cdot;u)\chi_2(x)\,dx
%\\=\int_D(l_2(x)-(C_2^tJ_{H_2}u_2)(x))\chi_2(x))\,dx\quad\forall\chi_2\in L^2(D),
%\end{multline}
 and \eqref{113}, \eqref{115},
transform the third and the fourth summands in (\ref{gh1}). By setting
$\mathbf q=\mathbf{\tilde j}$ in \eqref{113}
and $v=\tilde \varphi$ in \eqref{115}, we have
%\end{multicols}
\begin{multline}\label{gh2}
\int_D((\mathbf A(x))^{-1}\mathbf { \tilde j}(x),\mathbf z_1(x;u))_{\mathbb R^n}dx
-\int_Dz_2(x;u)\,\mbox{div}\, \mathbf { \tilde j}(x)\,dx
\\=\int_D(\mathbf l_1(x)-(C_1^tJ_{H_1}u_1)(x),\mathbf{\tilde j}(x))_{\mathbb R^n}dx,
\end{multline}
\begin{multline} \label{gh3}
-\int_D\tilde \varphi(x)\mbox{div}\,\mathbf z_1(x;u)\,dx-\int_{D}c(x)z_2(x;u)\tilde \varphi(x)\,dx
\\=\int_D(l_2(x)-(C_2^tJ_{H_2}u_2)(x))\tilde\varphi(x)\,dx.
\end{multline}
%\begin{multicols}{2}
On the other hand, putting $\mathbf q=\mathbf z_1(\cdot;u)$ in \eqref{16at1v} and
$v= z_2(\cdot;u)$ in \eqref{14aat1v}, we find
%\end{multicols}
\begin{equation} \label{gh4}
\int_D((\mathbf A(x))^{-1}\mathbf { \tilde j}(x),\mathbf z_1(x;u))_{\mathbb R^n}dx
-\int_D\tilde\varphi(x)\,\mbox{div}\,\mathbf z_1(x;u)\,dx=0,
\end{equation}
\begin{equation} \label{gh5}
\int_Dz_2(x;u)\,\mbox{div}\,\mathbf {\tilde j}(x)\,dx+\int_{D}c(x)\tilde\varphi(x)z_2(x;u)\,dx=\int_D\tilde
f(x)z_2(x;u)\,dx.
\end{equation}
%\begin{multicols}{2}
From \eqref{gh2}--\eqref{gh5}, we get
%\end{multicols}
$$
(\mathbf l_1-C_1^tJ_{H_1}u_1,\mathbf{ \tilde j})_{L^2(D)^n}
+(l_2-C_2^tJ_{H_2}u_2,\tilde\varphi)_{L^2(D)}
$$
%\end{multicols}
$$
=\int_D((\mathbf A(x))^{-1}\mathbf { \tilde j}(x),\mathbf z_1(x;u))_{\mathbb R^n}dx
-\int_Dz_2(x;u)\,\mbox{div}\, \mathbf { \tilde j}(x)\,dx
$$
$$
-\int_D\tilde \varphi(x)\,\mbox{div}\,\mathbf z_1(x;u)\,dx-\int_{D}c(x)z_2(x;u)\tilde \varphi(x)\,dx
=\int_D((\mathbf A(x))^{-1}\mathbf { \tilde j}(x),\mathbf z_1(x;u)_{\mathbb R^n}dx
$$
$$
-\int_D\tilde \varphi(x)\,\mbox{div}\,\mathbf z_1(x;u)\,dx
-\int_Dz_2(x;u)\,\mbox{div}\, \mathbf { \tilde j}(x)\,dx-\int_{D}c(x)z_2(x;u)\tilde \varphi(x)\,dx
=
$$
\begin{equation}\label{yxxy}
=0-(\tilde
f,z_2(\cdot;u)_{L^2(D)}=-(\tilde
f,z_2(\cdot;u)_{L^2(D)}.
\end{equation}
%\begin{multicols}{2}
Equalities \eqref{yxxy} and \eqref{gh1} imply
$$
l({\mathbf{\tilde j}},\tilde
\varphi)-\widehat{l({\mathbf{\tilde j}},\tilde \varphi)}
=-(\tilde
f,z_2(\cdot;u_1,u_2)_{L^2(D)}
-(u_1,\tilde\eta_1)_{H_1}-(u_2,\tilde\eta_2)_{H_2}-c
$$
$$
=-(\tilde f-f_0,z_2(\cdot;u_1,u_2))_{L^2(D)}-(f_0,z_2(\cdot;u_1,u_2))_{L^2(D)}
$$
\begin{equation}\label{lhjka}
-(u_1,\tilde\eta_1)_{H_1}-(u_2,\tilde\eta_2)_{H_2}-c=:\xi,
\end{equation}
where by $\xi$ we denote the random variable defined by the right-hand side of the latter equality. It is obvious that
%в силу \eqref{rand4}
\begin{equation*}
\mathbb E\xi=-(\tilde f-f_0,z_2(\cdot;u_1,u_2))_{L^2(D)}
\\-(f_0,z_2(\cdot;u_1,u_2))_{L^2(D)}-c,
\end{equation*}
$$
\xi-\mathbb
E\xi=-(u_1,\tilde\eta_1)_{H_1}-(u_2,\tilde\eta_2)_{H_2}.
$$
Taking into consideration the relationship
\begin{equation}\label{dispe}
\mathbf {D}\xi=\mathbb E(\xi-\mathbb E\xi)^2=\mathbb
E\xi^2-(\mathbb E\xi)^2
\end{equation}
that couples dispersion $\mathbf D\xi$ of the random variable $\xi$ and its
%mathematical
expectation  $\mathbb E\xi,$ we obtain from \eqref{lhjka}
$$
\mathbb E\left|l(\mathbf{\tilde  j},\tilde \varphi)-\widehat
{(l(\mathbf{\tilde  j},\tilde \varphi)}\right|^2
 =\left|(\tilde
f_2-f_0,z_2(\cdot;u))_{L^2(D)}\right.
$$
$$
\!\left.+(f_0,z_2(\cdot;u))_{L^2(D)}+c\right|^2\!
+\mathbb E[(u_1,\tilde\eta_1)_{H_1}+(u_2,\tilde\eta_2)_{H_2}]^2,
$$
whence we get \label{fffd}
$$
 \inf_{c \in
\mathbb R}\sup_{\tilde f\in G_0, (\tilde
\eta_1,\tilde \eta)\in G_1} \mathbb E|l(
\mathbf {\tilde j},\tilde\varphi)-\widehat {l(
\mathbf{\tilde  j},\tilde\varphi)}|^2=
$$
\begin{equation*}
=\inf_{c \in \mathbb R}\sup_{\tilde f\in
G_0}\left[(\tilde f-f_0,z_2(\cdot;u))_{L^2(D)}
\right.\\ \left.+(f_0,z_2(\cdot;u))_{L^2(D)}+c\right]^2
\end{equation*}
$$
+ \sup_{ (\tilde \eta_1,\tilde \eta_2)\in G_1}\mathbb
E[(\tilde\eta_1,u_1)_{H_1}+(\tilde\eta_2,u_2)_{H_2}]^2
$$
\begin{equation}\label{exh}
=\sup_{\tilde f\in G_0}\left[
 (\tilde f_2-f_2^{(0)},z_2(\cdot;u))_{L^2(D)}\right]^2
\\+\sup_{ (\tilde \eta_1,\tilde \eta_2)\in G_1}\mathbb
E[(\tilde\eta_1,u_1)_{H_1}+(\tilde\eta_2,u_2)_{H_2}]^2,
\end{equation}
with
$$
c=-(f_0,z_2(\cdot;u))_{L^2(D)}.
$$
In order to calculate the first term on the right-hand side of (\ref{exh})
make use of the Cauchy$-$Bunyakovsky inequality (see \cite{BIBLhat}, p. 186) and
(\ref{h11h}). We have
$$
|(\tilde f-f_0,z_2(\cdot;u))_{L^2(D)}|^2\leq
$$
$$
\leq
(Q^{-1}z_2(\cdot;u),z_2(\cdot;u))_{L^2(D)}
(Q(\tilde f-f_0),\tilde f-f_0)_{L^2(D)}
\leq
(Q^{-1}z_2(\cdot;u),z_2(\cdot;u))_{L^2(D)}.
$$
The direct substitution shows that last inequality
is transformed to an equality on the element
$$
\tilde f=f_0+\frac {Q^{-1}z_2(\cdot;u)}{(Q^{-1}z_2(\cdot;u),z(\cdot;u))_{L^2(D)}^{1/2}}.
$$
Hence,
\begin{equation}\label{ner88}
\sup_{\tilde f\in G_0}\left[
 (\tilde f_2-f_2^{(0)},z_2(\cdot;u))_{L^2(D)}\right]^2
 \\=(Q^{-1}z_2(\cdot;u),z_2(\cdot;u))_{L^2(D)}.
\end{equation}
In order to calculate the second term on the right-hand side of  (\ref{exh}),
note that the Cauchy$-$Bunyakovsky inequality, \eqref{restr2}, \eqref{rand4}, and \eqref{nekor} yields
$$
 \sup_{ (\tilde \eta_1,\tilde \eta_2)\in G_1}\mathbb
E[(\tilde\eta_1,u_1)_{H_1}+(\tilde\eta_2,u_2)_{H_2}]^2
$$
$$
\leq \sup_{ (\tilde \eta_1,\tilde \eta_2)\in G_1}[(\tilde Q_1^{-1}u_1,u_1)_{H_1}\mathbb E
(\tilde Q_1\tilde\eta_1,\tilde\eta_1)_{H_1}+(\tilde Q_2^{-1}u_2,u_2)_{H_2}\mathbb E(\tilde Q_2\tilde\eta_2,\tilde\eta_2)_{H_2}]
$$
\begin{equation}\label{ner56}
\leq (\tilde Q_1^{-1}u_1,u_1)_{H_1}+(\tilde Q_2^{-1}u_2,u_2)_{H_2}
\end{equation}
It is easy to see that \eqref{ner56} takes the form
at
$$
\tilde\eta_1=\nu_1\tilde Q_1^{-1}u_1/(\tilde Q_1^{-1}u_1,u_1)^{1/2},\quad \tilde\eta_2=\nu_2\tilde Q_2^{-1}u_1/(\tilde Q_2^{-1}u_1,u_1)^{1/2},
$$
where $\nu_1$ and $\nu_2$ are uncorrelated random variables with $\mathbb E\nu_1=\mathbb E\nu_2=0,$ $\mathbb E\nu_1^2=\mathbb E\nu_2^2=1.
$
Therefore,
\begin{equation}\label{ner560}
\sup_{ (\tilde \eta_1,\tilde \eta_2)\in G_1}\mathbb
E[(\tilde\eta_1,u_1)_{H_1}+(\tilde\eta_2,u_2)_{H_2}]^2=(\tilde Q_1^{-1}u_1,u_1)_{H_1}+(\tilde Q_2^{-1}u_2,u_2)_{H_2}.
\end{equation}
From (\ref{ner560}), (\ref{ner88}), and (\ref{exh}), we find
$$
 \inf_{c \in
\mathbb R}\sup_{\tilde f\in G_0, (\tilde
\eta_1,\tilde \eta)\in G_1} \mathbb E|l(
\mathbf {\tilde j},\tilde\varphi)-\widehat {l(
\mathbf{\tilde  j},\tilde\varphi)}|^2
=I(u),
$$
at $c=-(
z_2(\cdot;u),f_0)_{L^2(D)},$ where $I(u)$
is determined by (\ref{m20}). This proves the required assertion.
\end{proof}

% % §§§

\section{Representation for guaranteed estimates and errors of estimation via solutions of mixed variational equations}
%{\bf 5 Representation for guaranteed estimates and errors of estimation via solutions of mixed variational equations}

%{\bf Представлення для мінімаксних оцінок і похибок
%оцінювання.}

Solving optimal control problem
(\ref{113})--(\ref{m20}), we come to the following result.
%\begin{pred} \label{t4}
\begin{pred}
%{\bf Teорема 2.1}
There exists a unique guaranteed eatimate of $l(\mathbf {j},\varphi)$
which has the form
\begin{equation}\label{exactest}
\widehat{\widehat {l(\mathbf {j},\varphi)}}=(y_1,\hat u_1)_{H_1}+(y_2,\hat
u_2)_{H_2}+\hat c,
\end{equation}
where
\begin{equation}\label{hat}
\hat c=-\int_{D}\hat
z_2(x)f_0(x)\,dx,\quad\hat u_1=\tilde Q_1C_1\mathbf
{p}_1,\quad\hat u_2= \tilde Q_2C_2 p_2,
\end{equation}
and the functions $\mathbf {p}_1\in H(\mbox{\rm
div},D)$ and $\hat z_2 ,p_2\in L^2(D)$ are determined as a solution of the following uniquely solvable
problem:
\begin{multline} \label{r13at1}
\int_{D}(((\mathbf A(x))^{-1})^T\,\hat{\mathbf z}_1(x),\mathbf
q_1(x))_{\mathbb R^n}dx-\int_{D}\hat z_2(x)\mbox{\rm
div\,}\mathbf q_1(x)\,dx
\\=\int_{D}(\mathbf l_1(x)-C_1^tJ_{H_1}\tilde Q_1C_1\mathbf
{p}_1(x),\mathbf q_1(x))_{\mathbb R^n}\,dx \quad\forall\mathbf
q_1\in H(\mbox{\rm div},D),
\end{multline}
\begin{multline} \label{r14at1}
\!\!-\int_{D}v_1(x)\mbox{\rm div\,}\hat{\mathbf z}_1(x)\,dx-\int_{D}c(x)z_2(x)v_1(x)\,dx\\
\!=\!\int_{D}(l_2(x)-C_2^tJ_{H_2}\tilde Q_2C_2
p_2(x))v_1(x)\,dx\,\,\,\forall v_1\in L^2(D),
\end{multline}
\begin{equation}\label{r16at1'}
\int_{D}((\mathbf A(x))^{-1}\mathbf {p}_1(x),\mathbf q_2(x))_{\mathbb
R^n}dx\!-\!\int_{D}p_2(x)\mbox{\rm div\,}\mathbf q_2(x)\,dx=0
\,\,\,
\forall \mathbf q_2\in H(\mbox{\rm div},D),
\end{equation}
\begin{multline} \label{r15at1}
-\int_{D}v_2(x)\mbox{\rm div\,}\hat{\mathbf p}_1(x)\,dx-\int_Dc(x)p_2(x) v_2(x)dx
\\=\int_{D}v_2(x)Q^{-1}\hat z_2(x)\,dx \quad \forall v_2\in
L^2(D),
\end{multline}
where $\hat{\mathbf z}_1\in H(\mbox{\rm
div},D).$
%Задача \eqref{r13at1} -- \eqref{r15at1} однозначно розвязна.
The error of estimation $\sigma$ is given by an expression
\begin{equation}\label{rll}
\sigma=l(\mathbf {p}_1,p_2)^{1/2}.
\end{equation}
\end{pred}
\begin{proof}
Let us prove that the solution to the optimal control problem
(\ref{113})--(\ref{m20}) can be reduced to the solution
of system (\ref{r13at1})-(\ref{r15at1}).

Note first that functional
$I(u),$ where $u\in H$ can be represented in the form
\begin{equation*}\label{qfunk}
I(u)=\tilde I(u)+L(u) +\int_{D}Q^{-1}\tilde z_2^{(0)}(x)\tilde z_2^{(0)}(x) \,dx,
\end{equation*}
where
\begin{equation*}\label{qfunk1}
 \tilde I(u)=\int_{D}Q^{-1}\tilde z_2(x;u)\tilde z_2(x;u) \,dx
+ (\tilde
Q_1^{-1}u_1,u_1)_{H_1}+(\tilde Q_2^{-1}u_2,u_2)_{H_2},
\end{equation*}
\begin{equation*}\label{lfunk}
L(u)= 2\int_{D} Q^{-1}\tilde z_2(x;u)\tilde z_2^{(0)}(x) \, dx,
\end{equation*}
$\tilde z_2(x;u)$ is the second component of the pair
$(\tilde {\mathbf  z}_1(x;u), \tilde z_2(x;u))$ which the unique solution to problem
%$\tilde z_2(x,t;\mathbf u)$
%-- находится из решения задачи
(\ref{113}), (\ref{115}) at $\mathbf
 l_0^{(1)}(x)=0,$ $l_0^{(2)}(x)=0,$ and
 $\tilde z_2^{(0)}(x)$ is the second component of the pair $(\tilde {\mathbf z}_1^{(0)}(x), \tilde z_2^{(0)}(x))$
which the unique solution to the same problem at $u=0.$
%a из решения этой же задачи при $\mathbf u=0.$

Show that $\tilde I(u)$ is a quadratic form corresponding
to a symmetric continuous bilinear form \label{24p}
\begin{equation}\label{77}
\pi(u,v):= \int_{D}Q^{-1}\tilde z_2(x;u)\tilde z_2(x;v) \,dx+ (\tilde
Q_1^{-1}u_1,v_1)_{H_1}+(\tilde Q_2^{-1}u_2,v_2)_{H_2}
\end{equation}
on $H\times H$
and $
L(u)$ is a linear continuous functional defined on  $H$.

The continuity of form $
\pi(u,v)$ on $H\times H$ means that for all $u,v\in H$ the inequality
\begin{equation}\label{contbil}
|\pi(u,v)|\leq C\|u\|_H \|v\|_H
\end{equation}
must be valid,
where $C=\mbox{const}.$

To prove \eqref{contbil}, we use the estimate
\begin{equation}\label{78}
\int_{D}\tilde z_2^2(x;u)\,dx\leq
c_1\Bigl(\left\|C_1^tJ_{H_1}u_1\right\|^2_{L^2(D)^n}+\left\|C_2^tJ_{H_2}u_2\right\|^2_{L^2(D)}\Bigr),\quad c_1=\mbox{const},
\end{equation}
which follows from the inequality \eqref{eqn2'ap} at $\mathbf l_1=0$ and $l_2=0.$
%а также тем фактом, что
%\begin{equation}\label{79}
%\|\cdot\|_{H(\mbox{\small div};D)'}\leq c_2\|\cdot\|_{L^2(D)^n},\quad c_2=\mbox{const}.
%\end{equation}
For the first term in the right-hand side of \eqref{77}, due to the Cauchy$-$Bunyakovsky inequality and \eqref{78}
we have
$$
\left|\int_{D}Q^{-1}\tilde z_2(x;u)\tilde z_2(x;v) \,dx
\right|\leq
$$
\begin{equation}\label{uu'}
\leq c_2\left(\int_{D}\tilde z_2^2(x;u)\,dx\right)^{1/2}
\left(\int_{D}\tilde z_2^2(x;v)\,dx\right)^{1/2},
\end{equation}
$$
\leq c_2c_3\Bigl(\left\|C_1^tJ_{H_1}u_1\right\|^2_{L^2(D)^n}+\left\|C_2^tJ_{H_2}u_2\right\|^2_{L^2(D)}\Bigr)^{1/2}
$$
$$
\times c_3\Bigl(\left\|C_1^tJ_{H_1}v_1\right\|^2_{L^2(D)^n}+\left\|C_2^tJ_{H_2}v_2\right\|^2_{L^2(D)}\Bigr)^{1/2}
$$
\begin{equation}\label{hhg}
\leq c_4\Bigl(\left\|u_1\right\|^2_{H_1}+\left\|u_2\right\|^2_{H_2}\Bigr)^{1/2}
\Bigl(\left\|v_1\right\|^2_{H_1}+\left\|v_2\right\|^2_{H_2}\Bigr)^{1/2}=c_4\|u\|_H\|v\|_H,
\end{equation}
where $c_2,c_3,c_4=\mbox{const}.$

Analogously,
$$
(\tilde
Q_1^{-1}u_1,v_1)_{H_1}+(\tilde Q_2^{-1}u_2,v_2)_{H_2}\leq
c_5\|u\|_H\|v\|_H,\quad c_5=\mbox{const}.
$$
From this estimate and \eqref{hhg}
it follows the validity of the inequality \eqref{contbil}.

The continuity of linear functional
$L(u)$ on $H$ can be proved similary. \label{25p}

It is obvious that
$$
\tilde I(u)=\pi(u,u)\geq
(Q_1^{-1}u_1,u_1)_{H_1}+(\tilde Q_2^{-1}u_2,u_2)_{H_2}
\geq \alpha\|u\|_H^2 \quad \forall
u\in H,
$$
where $\alpha$ is a constant from \eqref{posdefme}.
In line with Theorem 1.1
proved in
\cite{BIBLlio}, p. 11,  the latter statements imply the existence of the unique element
 $\hat{u}:=(\hat u_1,\hat u_2)\in H$ such that
$$I(\hat{u})=\inf_{u
\in H}I(u).$$
Therefore, for any fixed
 $
w\in H$ and $\tau \in \mathbb R$ the function $s(\tau):=I(\hat{
u}+\tau w)$ reaches its minimum at a unique point $\tau
=0,$ so that,
\begin{equation} \label{z43}
\frac {d}{d\tau} I(\hat {u}+\tau w)\left.
\right|_{\tau=0}=0.
\end{equation}
Since
$$z_2(x; \hat {u}+\tau w)=z_2(x;\hat u)+\tau \tilde z_2(x;w),$$
relation (\ref{z43})
 yields
$$
\frac{1}{2}\frac{d}{dt}I(\hat u+\tau
w)\Bigl.\Bigr|_{\tau=0}
$$
\begin{equation}\label{z43a}
=(Q^{-1}z_2(\cdot;\hat u),\tilde z_2(\cdot;w ))_{L^2(D)}
+(\tilde Q_1^{-1}\hat u_1,w_1)_{H_1}
+(\tilde Q_2^{-1}\hat u_2,w_2)_{H_2}=0.
\end{equation}
Introduce a pair of functions $(\mathbf {p}_1,p_2)\in H(\mbox{\rm
div},D)\times L^2(D)$ as the unique solution of the problem
\begin{multline}\label{r16at1r}
\int_{D}((\mathbf A(x))^{-1}\mathbf {p}_1(x),\mathbf q_2(x))_{\mathbb
R^n}dx
\\- \int_{D}p_2(x)\mbox{\rm div\,}\mathbf q_2(x)\,dx=0
 \quad
\forall \mathbf q_2\in H(\mbox{\rm div},D),
\end{multline}
\begin{multline} \label{r15aat1r}
-\int_{D}v_2(x)\mbox{\rm div\,}{\mathbf p}_1(x)\,dx-\int_Dc(x)p_2(x) v_2(x)dx
\\=\int_{D}v_2(x)Q^{-1} z_2(x;\hat u)\,dx \quad \forall v_2\in
L^2(D).
\end{multline}
%Тоді перший доданок в лівій частині (\ref{z43a}) можна перетворити
%наступним чином:
Setting in (\ref{r16at1r}) $\mathbf q_2=\tilde {\mathbf z}_1(\cdot;w )$ and in (\ref{r15aat1r})
$v_2=\tilde z_2(\cdot;w ),$ we obtain
\begin{equation}\label{r16at1rr}
\int_{D}((\mathbf A(x))^{-1}\mathbf {p}_1(x),\tilde {\mathbf z}_1(x;w ))_{\mathbb
R^n}dx\\-\int_{D}p_2(x)\mbox{\rm div\,}\tilde {\mathbf z}_1(x;w )\,dx=0,
\end{equation}
\begin{multline} \label{r15at1rr}
-\int_{D}\tilde z_2(x;w )\mbox{\rm div\,}{\mathbf p}_1(x)\,dx-\int_Dc(x)p_2(x) \tilde z_2(x;w )dx
\\=\int_{D}v_2(x)Q^{-1} z_2(x;\hat u)\,dx.
\end{multline}
From (\ref{r16at1rr}) and (\ref{r15at1rr}), we find
$$
(Q^{-1}z_2(\cdot;\hat u),\tilde z_2(\cdot;w ))_{L^2(D)}
=-\int_{D}\tilde z_2(x;w)\mbox{\rm div\,}\mathbf p_1(x)\,dx-\int_Dc(x)p_2(x) \tilde z_2(x;w )dx
$$
\begin{equation*}
+\int_{D}((\mathbf A(x))^{-1}\mathbf {p}_1(x),\tilde {\mathbf z}_1(x;w ))_{\mathbb
R^n}dx\\-\int_{D}p_2(x)\mbox{\rm div\,}\tilde {\mathbf z}_1(x;w )\,dx
\end{equation*}
\begin{equation*}
=\int_{D}(((\mathbf A(x))^{-1})^T\tilde {\mathbf z}_1(x;w ),\mathbf {p}_1(x))_{\mathbb
R^n}dx
\\-\int_{D}\tilde z_2(x;w)\mbox{\rm div\,}\mathbf p_1(x)\,dx
\end{equation*}
$$
- \int_{D}p_2(x)\mbox{\rm div\,}\tilde {\mathbf z}_1(x;w )\,dx-\int_Dc(x)p_2(x) \tilde z_2(x;w )dx
$$
$$
=- \int_{D}(C_1^tJ_{H_1}w_1,\mathbf {p}_1(x))_{\mathbb
R^n}dx-\int_{D}(C_2^tJ_{H_2}w_2)p_2\,dx
$$
$$
=-(w_1,C_1\mathbf {p}_1)_{H_1}-(w_2,C_2p_2)_{H_2}.
$$
Last relation and (\ref{z43a}) imply
\begin{equation*}
(w_1,C_1\mathbf {p}_1)_{H_1}+(w_2,C_2p_2)_{H_2}\\=
(\tilde Q_1^{-1}\hat u_1,w_1)_{H_1}+(\tilde Q_2^{-1}\hat u_2,w_2)_{H_2}.
\end{equation*}
Hence,
\begin{equation}\label{z43af}
\hat u_1=\tilde Q_1C_1\mathbf {p}_1,\quad\hat u_2=\tilde Q_2C_2p_2.
\end{equation}

Setting these expressions into \eqref{113}, \eqref{115}
%$u_1$ and $u_2$ на $\hat u_1$ i $\hat u_2$, котрі визначаються виразами \eqref{z43af}, і,
and and denoting $\mathbf z_1(x;\hat u)=:\hat {\mathbf z}_1(x),$
$z_2(x;\hat u)=:\hat z_2(x),$ we establish that functions
$(\hat {\mathbf z}_1,$ $\hat z_2),$ $(\mathbf p_1,$ and $p_2)$ satisfy
\eqref{r13at1} -- \eqref{r15at1}; the unique solvability of the
problem \eqref{r13at1} -- \eqref{r15at1} follows from the existence of the unique minimum point
$\hat u$ of functional $I(u)$.

Now let us establish the validity of formula \eqref{rll}.
From (\ref{m20}) at $u=\hat u$ and \eqref{z43af}, it follows
$$
\sigma^2=I(\hat u)
=(Q^{-1}z_2(\cdot;\hat u),
z_2(\cdot;\hat u))_{L^2(D)}\\+(\tilde
Q_1^{-1}\hat u_1,\hat u_1)_{H_1}+(\tilde Q_2^{-1}\hat u_2,\hat u_2)_{H_2}
$$
\begin{equation}\label{z43aff}
=(Q^{-1}\hat z_2,
\hat z_2)_{L^2(D)}\\+(C_1\mathbf {p}_1,\tilde Q_1C_1\mathbf {p}_1)_{H_1}
+(C_2p_2,\tilde Q_2C_2p_2)_{H_2}.
\end{equation}
Transform the first term in \eqref{z43aff}. Setting in \eqref{r16at1r} and \eqref{r15aat1r}
$\mathbf q_2=\mathbf {\hat z}_1$ and $v_2=\hat z_2,$
we find
\begin{equation*}
\int_{D}((\mathbf A(x))^{-1}\mathbf {p}_1(x),\mathbf {\hat z}_1(x))_{\mathbb
R^n}dx\\- \int_{D}p_2(x)\mbox{\rm div\,}\mathbf {\hat z}_1(x)\,dx=0,
\end{equation*}
$$
-\int_{D}\hat z_2(x)\mbox{\rm div\,}\mathbf p_1(x)\,dx-\int_{D}c(x)p_2(x)\hat z_2(x)\,dx
=\int_{D}\hat z_2(x)Q^{-1}\hat z_2(x)\,dx.
$$
From the latter relations and from equations
\eqref{r13at1} and \eqref{r14at1} with $\mathbf q_1=\mathbf p_1$
and $v_1=p_2,$ we have
%\end{multicols}
$$
(Q^{-1}\hat z_2,
\hat z_2)_{L^2(D)}=-\int_{D}\hat z_2(x)\mbox{\rm div\,}\mathbf p_1(x)\,dx-\int_{D}c(x)p_2(x)\hat z_2(x)\,dx
$$
$$
+\int_{D}((\mathbf A(x))^{-1}\mathbf {p}_1(x),\mathbf {\hat z}_1(x))_{\mathbb
R^n}dx- \int_{D}p_2(x)\mbox{\rm div\,}\mathbf {\hat z}_1(x)\,dx
$$
$$
=\int_{D}(((\mathbf A(x))^{-1})^T\,\mathbf {\hat z}_1(x),\mathbf {p}_1(x))_{\mathbb
R^n}dx-\int_{D}z_2(x)\mbox{\rm div\,}\mathbf p_1(x)\,dx
$$
$$
- \int_{D}p_2(x)\mbox{\rm div\,}\mathbf {\hat z}_1(x)\,dx-\int_{D}c(x)p_2(x)\hat z_2(x)\,dx
$$
$$
=\int_{D}(\mathbf l_1(x)-C_1^tJ_{H_1}\tilde Q_1C_1\mathbf
{p}_1(x),\mathbf p_1(x))_{\mathbb R^n}\,dx+\int_{D}(l_2(x)-C_2^tJ_{H_2}\tilde Q_2C_2
p_2(x))p_2(x)\,dx
$$
$$
=\int_{D}(\mathbf l_1(x),\mathbf
p_1(x))_{\mathbb R^n}\, dx+\int_{D}l_2(x)p_2(x)\, dx
$$
\begin{equation}\label{yry}
-(C_1\mathbf {p}_1,\tilde Q_1C_1\mathbf {p}_1)_{H_1}
-(C_2p_2,\tilde Q_2C_2p_2)_{H_2}.
\end{equation}
%\begin{multicols}{2}
From \eqref{z43aff} and \eqref{yry}, we otain \eqref{rll}. Theorem is proved.
\end{proof}

Note that the pair of functions $(\hat{\mathbf z}(x),\hat z_2(x))=(\mathbf z_1(x;\hat u),z_2(x;\hat u))$
and the element $u=\hat u\in H$
is a solution of optimal control problem
(\ref{113}), (\ref{115}), \eqref{m20}.

In the following theorem we obtain an alternative representation
for the guaranteed estimate of
quantity $l(\mathbf j,\varphi)$ which is expressed via a solution of certain system of mixed
variational equations not
depending on $\mathbf l_1$ and $l_2$.
%To this end, introduce vector-functions
%Such a representation is obtained in the following
\begin{pred}\label{at666}
The guaranteed estimate of $l({\mathbf j},\varphi)$ has the form
\begin{equation}\label{Altgxx}
\widehat{\widehat {l({\mathbf j},\varphi)}}=l(\hat{\mathbf j},\hat \varphi),
\end{equation}
where the pair
$(\hat{\mathbf j},\hat \varphi)\in H(\mbox{\rm
div},D)\times L^2(D)$ is a solution to the following problem:
\begin{multline} \label{13aaag}
\int_{D}(((\mathbf A(x))^{-1})^T\hat{\mathbf p}_1(x),\mathbf
q_1(x))_{\mathbb R^n}dx- \int_{D}\hat p_2(x)\mbox{\rm
div\,}\mathbf q_1(x)\,dx\\ =\int_{D}(C_1^tJ_{H_1}\tilde
Q_1(y_1-C_1\hat{\mathbf
j})(x),\mathbf q_1(x))_{\mathbb R^n}dx
\,\,\forall\mathbf q_1\in H(\mbox{\rm div},D),
\end{multline}
\begin{multline} \label{p14aaag}
-\int_{D}v_1(x)\mbox{\rm div\,}\hat{\mathbf p}_1(x)\,dx-\int_{D}c(x)\hat p_2(x)v_1(x)\,dx
\\=\int_DC_2^tJ_{H_2}\tilde
Q_2(y_2-C_2\hat
\varphi)(x)v_1(x)\,dx\quad\forall v_1\in L^2(D),
\end{multline}
\begin{equation}\label{pr16at1g'}
\int_{D}((\mathbf A(x))^{-1}\hat{\mathbf j}(x),\mathbf
q_2(x))_{\mathbb R^n}\,dx- \int_{D}\hat\varphi(x)\mbox{\rm
div\,}\mathbf q_2(x)\,dx=0\quad
\forall \mathbf q_2\in H(\mbox{\rm div},D),
\end{equation}
\begin{multline} \label{18aaagq}
-\int_{D}v_2(x)\mbox{\rm div\,}\hat{\mathbf j}(x)\,dx-\int_{D}c(x)\hat\varphi(x)v_2(x)\,dx
\\=\int_{D}v_2(x)(Q^{-1}\hat p_2(x)-f_0(x))\,dx \quad
\forall v_2\in L^2(D),
\end{multline}
where equalities \eqref{13aaag}--\eqref{18aaagq} are fulfilled with probability
$1.$ Problem \eqref{13aaag} -- \eqref{18aaagq} is uniquely
solvable.

The random fields
$\hat{\mathbf j}$, $\hat{\mathbf p}_1$ and $\hat\varphi$,  $\hat p_2,$ whose realizations
%$\hat{\mathbf p}(x,t)$ and $\hat{\boldsymbol{\varphi}}(x,t)$
satisfy problem \eqref{13aaag}--\eqref{18aaagq},
belong to the spaces
$L^2(\Omega,H(\mbox{\rm
div},D))$ and $L^2(\Omega,L^2(D)),$ respectively.
\end{pred}
\begin{proof}
Note that unique solvability of problem \eqref{13aaag}--\eqref{18aaagq} at realizations $y_1$ and $y_2$
that belong with probability $1$ to the spaces $H_1$ and $H_2,$ respectively,
can be proved
similarly as to the problem \eqref{r13at1}--\eqref{r15at1}.

Namely, consider optimal control problem of the system described by
\footnote{ Unique solvability of problem \eqref{13aaa1tg}--\eqref{p14aaa1tg} for every fixed $u=(u_1,u_2)$ follows from correctness of stochastic
statement of mixed variational problem (2.2) on page 1427 in \cite{Ernst}.}
\begin{equation}\label{13aaa1tg}
\hat{\mathbf p}_1\in L^2(\Omega,H(\mbox{\rm
div},D))\quad\hat p_2\in L^2(\Omega,L^2(D)),
\end{equation}
\begin{multline} \label{13aaa1atg}
\mathbb E\Bigl[\int_{D}(((\mathbf A(x))^{-1})^T\hat{\mathbf p}_1(x;u),\mathbf
q_1(x))_{\mathbb R^n}dx\Bigr]- \mathbb E\Bigl[\int_{D}\hat p_2(x;u)\mbox{\rm
div\,}\mathbf q_1(x)\,dx\Bigr] \\=\mathbb E\Bigl[\int_{D}(\mathbf d_1(x)-(C_1^tJ_{H_1}u_1)(x)),\mathbf q_1(x))_{\mathbb R^n}dx\Bigr]
\,\,\forall\mathbf q_1\in L^2(\Omega,H(\mbox{\rm div},D)),
\end{multline}
\begin{multline} \label{p14aaa1tg}
-\mathbb E\Bigl[\int_{D}v_1(x)\mbox{\rm div\,}\hat{\mathbf p}_1(x;u)\,dx\Bigr]
-\mathbb E\Bigl[\int_{D}c(x)\hat p_2(x;u)v_1(x)\,dx\Bigr]
\\=\mathbb E\Bigl[\int_D(d_2(x)-(C_2^tJ_{H_2}u_2)(x))v_1(x)\,dx\Bigr]\quad\forall v_1\in L^2(\Omega,L^2(D)),
\end{multline}
with cost function
$$
I(u)=\mathbb E\Bigl[\int_{D} Q^{-1}(\hat p_2(\cdot;u)
-Qf_0)(x)(\hat p_2(\cdot;u)-Qf_0)(x)\, dx\Bigr]
$$
\begin{equation*} \label{g24gsttg}
+ (\tilde
Q_1^{-1}u_1,u_1)_{L^2(\Omega,H_1)}+(\tilde Q_2^{-1}u_2,u_2)_{L^2(\Omega,H_2)}\!\to
\inf_{u=(u_1,u_2)\in L^2(\Omega,H)=L^2(\Omega,H_1\times H_2)},
\end{equation*}
where
$$
\mathbf d_1(x)=C_1^tJ_{H_1}\tilde
Q_1y_1(x),
$$
$$
d_2(x)=C_2^tJ_{H_2}\tilde
Q_2y_2(x).
$$
Functional $I(u)$ is quadratic and coercive on the space $L^2(\Omega,H).$
Therefore, there exists a unique element $\hat
u\in L^2(\Omega,H)$ such that
$$
I(\hat u)= \inf_{u\in L^2(\Omega,H)} I(u).
$$
Next, denoting by $(\hat{\mathbf j},\hat \varphi)\in L^2(\Omega,H(\mbox{\rm
div},D))\times L^2(\Omega,L^2(D))$ a unique solution of the problem:
\begin{multline*}
\mathbb E\Bigl[\int_{D}((\mathbf A(x))^{-1}\hat{\mathbf j}(x),\mathbf
q_2(x))_{\mathbb R^n}\,dx\Bigr]\\- \mathbb E\Bigl[\int_{D}\hat\varphi(x)\mbox{\rm
div\,}\mathbf q_2(x)\,dx\Bigr]=0\quad
\forall \mathbf q_2\in L^2(\Omega,H(\mbox{\rm div},D)),
\end{multline*}
\begin{multline*}
-\mathbb E\Bigl[\int_{D}v_2(x)\mbox{\rm div\,}\hat{\mathbf j}(x)\,dx\Bigr]
-\mathbb E\Bigl[\int_{D}c(x)\hat\varphi(x)v_2(x)\,dx\Bigr]
\\=\mathbb E\Bigl[\int_{D}v_2(x)(Q^{-1}\hat p_2(x;\hat u)-f_0(x))\,dx\Bigr] \quad
\forall v_2\in L^2(\Omega,L_2),
\end{multline*}
and making use of virtually the same reasoning that led to the proof of Theorem 1,
we arrive at the equalities $\hat u_1=\tilde Q_1C_1\hat{\mathbf
j}$ and $\hat u_2=\tilde Q_2C_2\hat
\varphi.$ Denoting $\hat{\mathbf p}_1(x)=\hat{\mathbf p}_1(x;\hat u),$ $\hat p_2(x)=\hat
p_2(x;\hat u),$ we deduce from the latter statement the unique solvability of problem
\begin{multline*}
\mathbb E\Bigl[\int_{D}(((\mathbf A(x))^{-1})^T\,\hat{\mathbf p}_1(x),\mathbf
q_1(x))_{\mathbb R^n}dx\Bigr]- \mathbb E\Bigl[\int_{D}\hat p_2(x)\mbox{\rm
div\,}\mathbf q_1(x)\,dx\Bigr]\\ =\mathbb E\Bigl[\int_{D}(C_1^tJ_{H_1}\tilde
Q_1(y_1-C_1\hat{\mathbf
j})(x),\mathbf q_1(x))_{\mathbb R^n}dx\Bigr]
\,\,\forall\mathbf q_1\in L^2(\Omega,H(\mbox{\rm div},D)),
\end{multline*}
\begin{multline*}
-\mathbb E\Bigl[\int_{D}v_1(x)\mbox{\rm div\,}\hat{\mathbf p}_1(x)\,dx\Bigr]
-\mathbb E\Bigl[\int_{D}c(x)\hat p_2(x)v_1(x)\,dx\Bigr]
\\=\mathbb E\Bigl[\int_DC_2^tJ_{H_2}\tilde
Q_2(y_2-C_2\hat
\varphi)(x)v_1(x)\,dx\Bigr]\quad\forall v_1\in L^2(\Omega,L^2(D)),
\end{multline*}
\begin{multline*}
\mathbb E\Bigl[\int_{D}((\mathbf A(x))^{-1}\hat{\mathbf j}(x),\mathbf
q_2(x))_{\mathbb R^n}\,dx\Bigr]\\- \mathbb E\Bigl[\int_{D}\hat\varphi(x)\mbox{\rm
div\,}\mathbf q_2(x)\,dx\Bigr]=0\quad
\forall \mathbf q_2\in L^2(\Omega,H(\mbox{\rm div},D)),
\end{multline*}
\begin{multline*}
-\mathbb E\Bigl[\int_{D}v_2(x)\mbox{\rm div\,}\hat{\mathbf j}(x)\,dx\Bigr]
-\mathbb E\Bigl[\int_{D}c(x)\hat\varphi(x)v_2(x)\,dx\Bigr]
\\=\mathbb E\Bigl[\int_{D}v_2(x)(Q^{-1}\hat p_2(x)-f_0(x))\,dx\Bigr] \quad
\forall v_2\in L^2(\Omega,L^2(D)).
\end{multline*}
From here following the argument of paper \cite{Ernst}, we conclude that problem
\eqref{13aaag}--\eqref{18aaagq} is uniquely solvable.

Now let us prove the representation \eqref{Altgxx}.
By virtue of (\ref{clas}) and (\ref{hat}),
$$
\widehat{\widehat {l(\mathbf j,\varphi)}}=(y_1,\hat u_1)_{H_1}+(y_2,\hat
u_2)_{H_2}+\hat c
$$
\begin{equation}\label{lptg}
=(y_1,\tilde Q_1C_1\mathbf
p_1)_{H_1}+(y_2,\tilde Q_2C_2
p_2)_{H_2}-
(\hat z_2,f_0)_{L^2(D)}.
\end{equation}
Putting in (\ref{13aaag}) and (\ref{p14aaag})
$\mathbf q_1=\mathbf p_1$ and $v_1=p_2,$ we obtain
\begin{multline} \label{13aaatg}
\int_{D}(((\mathbf A(x))^{-1})^T\hat{\mathbf p}_1(x),\mathbf
p_1(x))_{\mathbb R^n}dx- \int_{D}\hat p_2(x)\mbox{\rm
div\,}\mathbf p_1(x)\,dx\\ =\int_{D}(C_1^tJ_{H_1}\tilde
Q_1(y_1-C_1\hat{\mathbf
j})(x),\mathbf p_1(x))_{\mathbb R^n}dx,
\end{multline}
\begin{multline} \label{p14aaatg}
-\int_{D}p_2(x)\mbox{\rm div\,}\hat{\mathbf p}_1(x)\,dx-\int_{D}c(x)\hat p_2(x)p_2(x)\,dx
\\=\int_D(C_2^tJ_{H_2}\tilde
Q_2(y_2-C_2\hat
\varphi)(x)p_2(x)\,dx.
\end{multline}
Putting in \eqref{r16at1'} and \eqref{r15at1} $\mathbf q_2=\hat{\mathbf p}_1$ and $v_2=\hat p_2,$
we find
\begin{equation}\label{r16at1ztg}
\int_{D}((\mathbf A(x))^{-1}\mathbf {p}_1(x),\hat{\mathbf p}_1(x))_{\mathbb
R^n}dx- \int_{D}p_2(x)\mbox{\rm div\,}\hat{\mathbf p}_1(x)\,dx=0,
\end{equation}
\begin{equation} \label{r15aat1ztg}
-\int_{D}\hat p_2(x)\mbox{\rm div\,}\mathbf p_1(x)\,dx-\int_{D}c(x)p_2(x)\hat p_2(x)\,dx
=\int_{D}\hat p_2(x)Q^{-1}\hat z_2(x)\,dx.
\end{equation}
Since the sum of the left-hand sides of equalities \eqref{13aaatg} and \eqref{p14aaatg} is equal to
the sum of the left-hand sides of \eqref{r16at1ztg} and \eqref{r15aat1ztg}, we find from \eqref{lptg}
%тобто
%$$
%\int_{D}(((\mathbf A(x))^{-1})^T\hat{\mathbf p}_1(x),\mathbf
%p_1(x))_{\mathbb R^n}dx- \int_{D}\hat p_2(x)\mbox{\rm
%div\,}\mathbf p_1(x)\,dx
%$$
%$$
%-\int_{D}p_2(x)\mbox{\rm div\,}\hat{\mathbf p}_1(x)\,dx-\int_{D}c(x)\hat p_2(x)p_2(x)\,dx
%$$
%$$
%=\int_{D}(\mathbf A^{-1}(x)\,\mathbf {p}_1(x),\hat{\mathbf p}_1(x))_{\mathbb
%R^n}dx- \int_{D}p_2(x)\mbox{\rm div\,}\hat{\mathbf p}_1(x)\,dx
%$$
%$$
%-\int_{D}\hat p_2(x)\mbox{\rm div\,}\mathbf p_1(x)\,dx-\int_{D}c(x)p_2(x)\hat p_2(x)\,dx,
%$$
%а також, враховуючи співвідношення \eqref{lptg}, знаходимо
%$$
%\widehat{\widehat
%{l(\mathbf j,\varphi)}}+(\hat z_2,f_0)_{L^2(D)}-(C_1\hat{\mathbf
%j},\tilde Q_1C_1\mathbf p_1)_{H_1}
%-(C_2\hat\varphi,\tilde Q_2C_2p_2)_{H_2}=\int_{D}\hat p_2(x)Q^{-1}\hat z_2(x)\,dx,
%$$
%звідки
\begin{equation}\label{g2hat2tg}
\widehat{\widehat
{l(\mathbf j,\varphi)}}=(C_1\hat{\mathbf
j},\tilde Q_1C_1\mathbf p_1)_{H_1}
+(C_2\hat\varphi,\tilde Q_2C_2p_2)_{H_2}
\\+(Q^{-1}\hat p_2-f_0,\hat z_2)_{L^2(D)}.
\end{equation}
Next, putting in \eqref{pr16at1g'}, \eqref{18aaagq} $\mathbf q_2=\hat{\mathbf z}_1,$ $v_2=\hat z_2$ and in \eqref{r13at1}, \eqref{r14at1}
$\mathbf q_1=\hat{\mathbf j},$
$v_1=\hat \varphi,$ we obtain
\begin{equation}\label{pr16at1ytg}
\int_{D}((\mathbf A(x))^{-1}\hat{\mathbf j}(x),\hat{\mathbf
z}_1(x))_{\mathbb R^n}\,dx- \int_{D}\hat\varphi(x)\mbox{\rm
div\,}\hat{\mathbf z}_1(x)\,dx=0,
\end{equation}
\begin{multline} \label{18aaaytg}
-\int_{D}\hat z_2(x)\mbox{\rm div\,}\hat{\mathbf j}(x)\,dx-\int_{D}c(x)\hat\varphi(x)\hat z_2(x)\,dx
\\=\int_{D}\hat z_2(x)(Q^{-1}\hat p_2(x)-f_0(x))\,dx,
\end{multline}
and
\begin{multline} \label{r13at1ytg}
\int_{D}(((\mathbf A(x))^{-1})^T\hat{\mathbf z}_1(x),\hat{\mathbf j}(x))_{\mathbb R^n}dx-\int_{D}\hat z_2(x)\mbox{\rm
div\,}\hat{\mathbf j}(x)\,dx
\\=\int_{D}(\mathbf l_1(x)-C_1^tJ_{H_1}\tilde Q_1C_1\mathbf
{p}_1(x),\hat{\mathbf j}(x))_{\mathbb R^n}\,dx,
\end{multline}
\begin{multline} \label{r14at1ytg}
-\int_{D}\hat \varphi(x)\mbox{\rm div\,}\hat{\mathbf z}_1(x)\,dx-\int_{D}c(x)z_2(x)\hat \varphi(x)\,dx
\\=\int_{D}(l_2(x)-C_2^tJ_{H_2}\tilde Q_2C_2
p_2(x))\hat \varphi(x)\,dx.
\end{multline}
Relations \eqref{pr16at1ytg}--\eqref{r14at1ytg} imply
$$
\int_{D}\hat z_2(x)(Q^{-1}\hat p_2(x)-f_0(x))\,dx
=(\mathbf l_1-\tilde Q_1C_1\hat{\mathbf p}_1,C_1\hat{\mathbf j}(x))_{H_1}
+(l_2-\tilde Q_2C_2\hat p_2,C_2\hat \varphi(x))_{H_2}.
$$
By virtue of \eqref{g2hat2tg}, it follows from here representation \eqref{Altgxx}.
\end{proof}
{\bf Remark 1.}
Notice that
in  representation $l(\hat{\mathbf j},
\hat \varphi)$
for minimax estimate
$\widehat{\widehat {l(\mathbf j,
\varphi)}}$ the functions $\hat{\mathbf j},
\hat \varphi$ which are defined from equations \eqref{13aaag}--\eqref{18aaagq} do not depend on specific form of functional
$l$ and hence can be taken as a good estimate
for unknown solution $\mathbf j,\varphi$ of Dirichlet problem (\ref{eqsystem1}), (\ref{eqsystem2}).

\section{Approximate Guaranteed Estimates: The Theorems on Convergence}
%{\bf 6 Approximate Guaranteed Estimates: The Theorems on Convergence}

In this section we introduce the notion of approximate guaranteed estimates of $l(\mathbf j,\varphi)$ and prove their convergence to
$\widehat{\widehat {l(\mathbf j,\varphi)}}$.
To do this,
we use
%the finite element
the mixed finite element method for solving the
aforementioned problems (\ref{r13at1})-(\ref{r15at1}) and \eqref{13aaag}--\eqref{18aaagq} and obtain approximate
estimates via solutions of linear algebraic equations. We show
their convergence to the optimal estimates.

In this section $D$ is supposed to be bounded and connected domain of $\mathbb R^n$
with
polyhedral boundary $\Gamma.$
First, we note that according to the mixed finite element method, an approximation $(\mathbf j^h,\varphi^h)$ to the solution $(\mathbf j,\varphi)$
of the problem \eqref{eqn2}, \eqref{eqn2'}
%consists of the following.
%The mixed fixed finite element discretization of (2.1) is based on its mixed formulation (2.3). We
%Consider a simplicial triangulation $\mathcal T_h$ of $D$. The mixed method seeks
%an approximation
%$(\mathbf j_h,\varphi_h)$
%to the solution $(\mathbf j,\varphi)$
%of the problem \eqref{eqn2}, \eqref{eqn2'}
is sought in the finite element space
$V_1^h\times V_2^h$ given by
$$
V_1^h=\{\mathbf q^h\in H(\mbox{div};\,D): \mathbf q^h|_K\in (P^k(K))^n+\mathbf xP^k(K)\quad\forall K\in\mathcal T_h\},
$$
$$
V_2^h=\{v^h\in L^2(D): v^h|_K\in P^k(K)\quad\forall K\in\mathcal T_h\},
$$
where\label{page25} $\mathcal T_h$ is a simplicial triangulation  of $D$, $P^k(K)$ denotes the space of polynomials on $K$ of degree at most $k$, $k\geq 0,$ $\mathbf x:=(x_1,\dots,x_n)$,
and is defined by requiring that
\begin{equation} \label{eqnn2}
a(\mathbf j^h,\mathbf q^h)+b(\mathbf q^h,\varphi^h)=0\quad \forall \mathbf q^h \in V_1^h,
\end{equation}
\begin{equation} \label{eqnn2'}
b(\mathbf j^h,v^h)- c(\varphi^h,v^h)=(f,v^h)_{L^2(D)}\quad\forall v^h\in V_2^h
\end{equation}
%or what is the same
Here the bilinear forms $a(\cdot,\cdot),$ $b(\cdot,\cdot),$ and $c(\cdot,\cdot)$
%and the linear form $l(\cdot)$
are defined by \eqref{a}--\eqref{c}.
Hence system \eqref{eqnn2}, \eqref{eqnn2'} can be rewritten in the form
\begin{equation}\label{generalizedll}
\int_D\!((\mathbf A(x))^{-1}\mathbf j^h(x),\mathbf q^h(x))_{\mathbb
R^n}dx-\int_D\varphi^h(x)\mbox{div}\, \mathbf
q^h(x)dx=0\quad \forall \mathbf q^h\in V_1^h
\end{equation}
\begin{equation}\label{generalized1ll}
\int_Dv(x)\mbox{div}\,\mathbf j^h(x)dx+\int_Dc(x)\varphi^h (x) v^h(x)dx=\int_Df(x)v^h(x)\,dx\,\,\forall v^h\in V_2^h.
\end{equation}

It can be easily verified that the bilinear form $a|_{V_1^h\times V_1^h}$
is uniformly coercive on $\mbox{Ker}\, B|_{V_1^h}$ and that the bilinear form $b|_{V_1^h\times V_2^h}$
satisfies the inf-sup condition (Babuska-Brezzi condition). Moreover, we have $\mbox{Ker}\, B^t|_{V_2^h}=\emptyset$ and
therefore, the mixed discretization \eqref{eqnn2}, \eqref{eqnn2'} (or what is the same \eqref{generalizedll}, \eqref{generalized1ll}) is uniquely solvable and the following estimates
are valid
\begin{multline}\label{m7nn}
\|\mathbf j-\mathbf j^h\|_{H(\mbox{\rm\small div},D)}+\|\varphi-\varphi^h\|_{L^2(D)}\\ \leq \tilde c\left(\inf_{\mathbf q^h\in V_1^h}
\|\mathbf j-\mathbf q^h\|_{H(\mbox{\rm\small div},D)}+\inf_{v^h\in V_2^h}\|\varphi-v^h\|_{L^2(D)}\right),
\end{multline}
\begin{equation}\label{m7nnl}
\|\mathbf j^h\|_{H(\mbox{\rm\small div},D)}+\|\varphi^h\|_{L^2(D)}\\ \leq \tilde{\tilde c}\|f\|_{L^2(D)},
\end{equation}
where
$\tilde c$ and $\tilde{\tilde c}$ are constant not depending on $h$
%i $(\mathbf j,\varphi)\in H(\mbox{\rm div},D)\times L^2(D)$ розв'язок задачі \eqref{eqn2}, \eqref{eqn2'}
(cf. e.g. \cite{Brezzi}; \S II, Prop. 2.11])
and \cite{Gatica}, page 102).
%and $(\mathbf z_1^h(\cdot;u),z_2^h(\cdot;u))$ is a solution of system of variational equations
%\eqref{113h}--\eqref{115h} at $h_n=h$.

Note that since $\mbox{div}\,\mathbf q_h|_K\in P^k(K),$ $K\in \mathcal T_h,$ then a natural choice for the approximation of the
variable $\varphi$ is to use piecewise polynomials of degree at most $k$ leading to the
%ansatz
space $V_2^h$
defined above. Due to Proposition 3.9 of \cite{Brezzi}, p. 132, it follows that the sequences of the subspaces $\{V_1^h\}$ and $\{V_2^h\}$
%defined by an infinite set of parameters $h_1,h_2,\dots$ with $\lim_{k\to \infty}h_k=0$
are complete in
$H(\rm div;D)$ and $L^2(D),$ respectively, in the following sense.
\begin{predll}
{\it Let} V {\it be a Hilbert space.
Introduce a sequence of finite-dimensional subspaces} $V^h$ {\it in} $V$,
{\it defined by an infinite set of parameters} $h_1,h_2,\dots$ {\it with} $\lim_{k\to \infty}h_k=0$.

{\it We say that sequence} $\{V^h\}$ {\it is complete in} $V$,
{\it if for any} $v\in V$ {\it and} $\epsilon>0$ {\it there exists an }$\hat h=\hat h(v,\epsilon)>0$ {\it such that}
 $\inf_{w\in V^h}\|v-w\|_H<\varepsilon$ {\it for any} $h<\hat h$.
{\it In other words, the completeness of sequence} $\{V^h\}$ {\it means that any element} $v\in V$
{\it may be approximated with any degree of accuracy by elements of}
 $\{V^h\}$.
\end{predll}
Completeness of $\{V_1^h\}$ and $\{V_2^h\}$ in $H(\rm div;D)$ and $L^2(D)$ together with estimate \eqref{m7nn} imply that
\begin{equation}\label{m7nnn}
\lim_{h\to 0}\bigl (\|\mathbf j-\mathbf j^h\|_{H(\mbox{\rm\small div},D)}+\|\varphi-\varphi^h\|_{L^2(D)}\bigr)=0.
\end{equation}

Now we are in a position to give the following definition.

Take an approximate guaranteed estimate of $l(\mathbf j,\varphi)$ as
\begin{equation}\label{apprest}
\widehat {l^h(\mathbf j,\varphi)}=(u_1^{h},y_1)_{H_1}+(u_2^{h},y_2)_{H_1}+c^{h},
\end{equation}
where $u_1^{h}=\tilde Q_1C_1\mathbf
{p}_1^{h},$ $u_2^{h}=\tilde Q_2C_2 p_2^{h},$ $c^{h}=\int_{D}\hat
z_2^{h}(x)f_0(x)\,dx,$
and functions $\mathbf {\hat z}_1^{h},\mathbf {p}_1^{h}\in V_1^{h}$ and $\hat z_2^{h} ,p_2^{h}\in V_2^{h}$
%\footnote{однозначно розв'язність цієї системи доводиться цілком аналогічно, як і для системи варіаційних рівнянь \eqref{13at1v}-- \eqref{14aat1v} (див. доведення теореми \ref{})}
are determined from the following uniquely solvable system of variational equalities
\begin{multline} \label{r13at1hgg}
\int_{D}(\mathbf A^{-1}(x)\,\hat{\mathbf z}_1^{h}(x),\mathbf
q_1^{h}(x))_{\mathbb R^n}dx+ \int_{D}\hat z_2^{h}(x)\mbox{\rm
div\,}\mathbf q_1^{h}(x)\,dx
\\=\int_{D}(\mathbf l_1(x)-C_1^tJ_{H_1}\tilde Q_1C_1\mathbf
{p}_1^{h}(x),\mathbf q_1^{h}(x))_{\mathbb R^n}\,dx \quad\forall\mathbf
q_1^{h}\in V_1^{h},
\end{multline}
\begin{equation} \label{r14at1hgg}
\int_{D}v_1^{h}(x)\mbox{\rm div\,}\hat{\mathbf z}_1^{h}(x)\,dx
=\int_{D}(l_2(x)-C_2^tJ_{H_2}\tilde Q_2C_2
p_2^{h}(x))v_1^{h}(x)\,dx\quad\forall v_1^{h}\in V_2^{h},
\end{equation}
\begin{equation}\label{r16at1hgg}
\int_{D}(\mathbf A^{-1}(x)\,\mathbf {p}_1^{h}(x),\mathbf q_2^{h}(x))_{\mathbb
R^n}dx+ \int_{D}p_2^{h}(x)\mbox{\rm div\,}\mathbf q_2^{h}(x)\,dx=0
 \quad
\forall \mathbf q_2^{h}\in V_1^{h},
\end{equation}
\begin{equation} \label{r15aat1hgg}
\int_{D}v_2^{h}(x)\mbox{\rm div\,}\mathbf p_1^{h}(x)\,dx
=\int_{D}v_2^{h}(x)Q^{-1}\hat z_2^{h}(x)\,dx \quad \forall v_2^{h}\in
V_2^{h}.
\end{equation}

%Let sequences of finite-dimensional subspaces $V_1^{h}$ and $V_2^{h}$ be complete in the spaces $H(\mbox{\rm div},D)$ and $L^2(D),$ respectively, i.e.
%$$
%\forall \boldsymbol \varphi\in H(\mbox{\rm div},D), \forall \psi\in L^2(D) \,\,\exists\,\, \mbox{sequences}\,\,\{\boldsymbol \varphi^h\},
%\{\psi^h\},\boldsymbol \varphi^{h}\in V_1^{h}, \psi^h\in V_2^{h}\,\,\mbox{such that}
%$$
%$$
%\|\boldsymbol \varphi-\boldsymbol \varphi^h\|_{H(\mbox{\rm \small div},D)}\to 0,\quad \|\psi-\psi^h\|_{L^2(D)}\to 0\quad\mbox{as}
%\quad h\to 0,
%$$

%означает, что всякий элемент
%$E\in H_0(\mbox{\rm rot},D)$ может с любой степенью точности аппроксимирован элементами подпространств $V^h$.

The unique solvability of system \eqref{r13at1hgg}--\eqref{r15aat1hgg} follows from the same reasoning of the previous
sections which led to the proof of Theorem 1 with $H(\mbox{\rm div},D)$ and $L^2(D)$ being replaced by $V_1^h$ and $V_2^h$,
respectively.

%Застосуємо далі цю нерівність для доведення теорем про збіжність наближених мінімакних оцінок, знаходження яких зводитемется до розв'язання деяких систем лінійних алгебраїчних %рівнянь, до
%їх точних значень, що фігурують в теоремах 2.1 i 2.2. \label{Vh1}

\begin{pred}
Let $\hat{\mathbf z}_1, \mathbf p_1\in H(\mbox{\rm div},D),$ $\hat{z}_2, p_2\in L^2(D)$ and
$\hat{\mathbf z}_1^{h}, \mathbf p_1^{h}\in V_1^{h},$ $\hat{z}_2^{h}, p_2^{h}\in V_2^{h}$ be solutions
of problems \eqref{r13at1}--\eqref{r15at1} and \eqref{r13at1hgg}--\eqref{r15aat1hgg}, respectively.

Then the following hold:
\begin{equation} \label{r15aat1hggk}
\|\hat{\mathbf z}_1-
\hat{\mathbf z}_1^{h}\|_{H(\mbox{\rm \small div},D)}
+\|\hat z_2-\hat{z}_2^{h}\|_{L^2(D)}\to 0 \quad\mbox{\rm as}\quad h\to 0,
\end{equation}

i)
\begin{equation} \label{r15aat1hggk'}
\|\mathbf p_1-
\mathbf p_1^{h}\|_{H(\mbox{\rm \small div},D)}
+\|p_2-p_2^{h}\|_{L^2(D)}\to 0 \quad\mbox{\rm as}\quad h\to 0.
\end{equation}

ii) Approximate guaranteed estimate $\widehat {l^{h}(\mathbf j,\varphi)}$ of $l(\mathbf j,\varphi)$ tends to a guaranteed estimate
$\widehat{\widehat {l(\mathbf j,\varphi)}}$ of this expression as $h\to 0$ in the sense that
$$
\lim_{h\to 0}\mathbb E|\widehat {l^{h}(\mathbf j,\varphi)}-\widehat{\widehat {l(\mathbf j,\varphi)}}|^2=0.
$$

Moreover,
\begin{equation} \label{r15aat1hggk''}
\lim_{h\to 0}\sup_{\tilde f\in
G_0, (\tilde \eta_1,\tilde \eta_2)\in G_1}\mathbb E|\widehat {l^{h}(\tilde{\mathbf j},\tilde\varphi)}-\widehat{\widehat {l(\tilde{\mathbf j},\tilde\varphi)}}|^2=0,
\end{equation}
and
\begin{equation} \label{r15aat1hggk'''}
\lim_{h\to 0}\sup_{\tilde f\in
G_0, (\tilde \eta_1,\tilde \eta_2)\in G_1}\mathbb E|\widehat {l^{h}(\tilde{\mathbf j},\tilde\varphi)}-l(\tilde{\mathbf j},\tilde\varphi)|^2=
\sup_{\tilde f\in
G_0, (\tilde \eta_1,\tilde \eta_2)\in G_1}\mathbb E|\widehat{\widehat {l(\tilde{\mathbf j},\tilde\varphi)}}-l(\tilde{\mathbf j},\tilde\varphi))|^2,
\end{equation}
where $\tilde f,$
$\tilde {\mathbf j},$ and $\tilde \varphi$  have the same sence as in the definition 1,
$\widehat {l^h(\tilde {\mathbf j},\tilde \varphi)}=(u_1^{h},\tilde y_1)_{H_1}+(u_2^{h},\tilde y_2)_{H_1}+c^{h},$ $\tilde y_1=C_1\tilde {\mathbf j}+\tilde \eta_1,$ $\tilde y_2=C_2\tilde \varphi+\tilde \eta_2$.
\end{pred}
\begin{proof}
Denote by $\{h_n\}$ any sequence of positive numbers such that $h_n\to 0$ when $n\to\infty$.
Let $\mathbf z_1^{h_n}(\cdot;u)\in V_1^{h_n},$ i $z_2^{h_n}(\cdot;u)\in V_2^{h_n}$ be a solution of the problem
\begin{multline}\label{113h}
\int_{D}(((\mathbf A(x))^{-1})^T\mathbf z_1^{h_n}(x;u),\mathbf
q^{h_n}(x))_{\mathbb R^n}dx- \int_{D}z_2^{h_n}(x;u)\mbox{div\,}\mathbf
q^{h_n}(x)\,dx=\\ \int_{D}(\mathbf l_1(x)-(C_1^tJ_{H_1}u_1)(x),\mathbf q^{h_n}(x))_{\mathbb
R^n}dx\quad \forall \mathbf q^{h_n}\in V_1^{h_n},
\end{multline}
\begin{multline}\label{115h}
-\int_{D}v^{h_n}(x)\mbox{div\,}\mathbf
z_1^{h_n}(x;u)\,dx-\int_{D}c(x)z_2^{h_n}(x;u)v^{h_n}(x)\,dx\\=\int_{D}(l_2(x)-(C_2^tJ_{H_2}u_2)(x))v^{h_n}(x)\,dx\quad
\forall v^{h_n}\in V_2^{h_n}.
\end{multline}
Then
\begin{equation}\label{hatzz}
\mathbf {\hat z}_1^{h_n}(x)=\mathbf {z}_1^{h_n}(x;u^{h_n}),\quad \hat z_2^{h_n}(x)=z_2^{h_n}(x;u^{h_n}).
\end{equation}
Problem \eqref{113h}, \eqref{115h} can be rewritten as
\begin{equation*} \label{eqnn2l}
a^*(\mathbf j^h,\mathbf q^h)+b(\mathbf q^h,\varphi^h)=0\quad \forall \mathbf q^h \in V_1^h,
\end{equation*}
\begin{equation*} \label{eqnn2'l}
b(\mathbf j^h,v^h)- c(\varphi^h,v^h)=(f,v^h)_{L^2(D)}\quad\forall v^h\in V_2^h,
\end{equation*}
where
$$a^*(\mathbf j^h,\mathbf q^h)=a(\mathbf q^h,\mathbf j^h)=\int_{D}(((\mathbf A(x))^{-1})^T\mathbf j^{h}(x),\mathbf
q^{h}(x))_{\mathbb R^n}dx$$ and the bilinear forms $a(\cdot,\cdot),$ $b(\cdot,\cdot),$ and $c(\cdot,\cdot)$ are defined by \eqref{a}, \eqref{b}, and \eqref{c}
respectively.

Since the bilinear form $a(\mathbf j^h,\mathbf q^h)$ is uniformly coercive on $\mbox{Ker}\, B|_{V_1^h}$ then the form
is also uniformly coercive on $\mbox{Ker}\, B|_{V_1^h}$ with the same constant  and hence
system \eqref{113h}, \eqref{115h} is uniquely solvable.
Theorem 1.2, Prop. 2.11 in \S 2 from \cite{Brezzi} (see also \cite{Gatica}, page 102),
and uniform coerciveness of the form $a^*(\mathbf j^h,\mathbf q^h)$ on $\mbox{Ker}\, B|_{V_1^h}$
imply that the following estimates are valid
\begin{multline}\label{m7nn'}
\|\mathbf z_1(\cdot;u)-\mathbf z_1^{h_n}(\cdot;u)\|_{H(\mbox{\rm\small div},D)}+\|z_2(\cdot;u)-z_2^{h_n}(\cdot;u)\|_{L^2(D)}\\ \leq \tilde c\left(\inf_{\mathbf q^{h_n}\in V_1^{h_n}}
\|\mathbf z_1(\cdot;u)-\mathbf q^{h_n}\|_{H(\mbox{\rm\small div},D)}+\inf_{v^{h_n}\in V_2^h}\|z_2(\cdot;u)-v^{h_n}\|_{L^2(D)}\right),
\end{multline}

\begin{multline}\label{m7nn2}
\|\mathbf z_1^{h_n}(\cdot;u)\|_{H(\mbox{\rm\small div},D)}+\|z_2^{h_n}(\cdot;u)\|_{L^2(D)}\\ \leq \tilde {\tilde c}\left(
\|\mathbf l_1-C_1^tJ_{H_1}u_1\|_{L^2(D)^n}+\|l_2-C_2^tJ_{H_2}u_2\|_{L^2(D)}\right),
\end{multline}
where
$\tilde c,\tilde {\tilde c}$ are constants not depending on $h$
%i $(\mathbf j,\varphi)\in H(\mbox{\rm div},D)\times L^2(D)$ розв'язок задачі \eqref{eqn2}, \eqref{eqn2'}
%(cf. e.g. \cite{Brezzi}; \S II, Prop. 2.11 and \cite{Gatica}, page 102)
and $(\mathbf z_1(\cdot;u),z_2(\cdot;u))$ is a solution of system of variational equations
\eqref{113}, \eqref{115}.

From estimate \eqref{m7nn'} and completeness of $\{V_1^h\}$ and $\{V_2^h\}$ in $H(\rm div;D)$ and $L^2(D),$  it follows that
\begin{equation}\label{115hh}
\|\mathbf z_1(\cdot;u)-\mathbf z_1^{h_n}(\cdot;u)\|_{H(\mbox{\rm\small div},D)}+\|z_2(\cdot;u)-z_2^{h_n}(\cdot;u)\|_{L^2(D)}\to 0
\end{equation}
as $n\to \infty.$

Prove now that
$$
\lim_{n\to\infty}\|u^{h_n}-\hat u\|_H=\lim_{n\to\infty}\left(\|u_1^{h_n}-\hat u_1\|_{H_1}^2+\|u_2^{h_n}-\hat u_2\|_{H_2}^2\right)^{1/2}=0,
$$
where $u^{h_n}=(u_1^{h_n},u_2^{h_n}),$ $\hat u=(\hat u_1,\hat u_2),$ $H=H_1\times H_2.$

Set
\begin{equation*}\label{m20h}
I_n(u)=(Q^{-1}z_2^{h_n}(\cdot;u),z_2^{h_n}(\cdot;u))_{L^2(D)}
+ (\tilde Q_1^{-1}u_1,u_1)_{H_1}+(\tilde Q_2^{-1}u_2,u_2)_{H_2}.
\end{equation*}
It is clear that
$$
\inf_{u\in H}I_n(u)=I_n(u^{h_n})
$$
and
$$
I_n(u^{h_n})\leq I_n(\hat u).
$$
From strong convergence of the sequence $\{(\mathbf z_1^{h_n}(\cdot;\hat u),z_2^{h_n}(\cdot;\hat u))\}$
to $(\mathbf z_1(\hat u),z_2(\hat u))$ in the space $H(\mbox{\rm div},D)\times L^2(D),$
which follows from \eqref{115hh}, we have
%
%Since, on an account of \eqref{115hh}, the sequence $(\mathbf z_1^{h_n}(\cdot;\hat u),z_2^{h_n}(\cdot;\hat u))$
%strongly converges to $(\mathbf z_1(\hat u),z_2(\hat u))$ in the space $H(\mbox{\rm div},D)\times L^2(D),$ we have
$$
\lim_{n\to\infty}I_n(\hat u)=I(\hat u),
$$
and, hence $\overline{\lim}_{n\to\infty}I_n(u^{h_n})\leq I(\hat u).$
Since
$$
I_n(u^{h_n})\geq (\tilde Q_1^{-1}u_1^{h_n},u_1^{h_n})_{H_1}+(\tilde Q_2^{-1}u_2^{h_n},u_2^{h_n})_{H_2}
\geq \alpha\|u^{h_n}\|_H^2,
$$
where $\alpha >0$ is the constant from \eqref{posdefme}, then $\|u^{h_n}\|_H\leq C\,\,(C={\rm const)}$ and we can extract from the sequence $\{u^{h_n}\}$ a subsequence $\{u^{h_{n_k}}\}$ such that $u^{h_{n_k}}\to\tilde u$
weakly in $H$ (see \cite{Yosida}, Theorem 1, p. 180).

Prove that the sequence $\{(\mathbf z_1^{h_{n_k}}(\cdot;u^{h_{n_k}}),z_2^{h_{n_k}}(\cdot;u^{h_{n_k}}))\}$ weakly converges to $(\mathbf z_1(\tilde  u),z_2(\tilde u))$ in $H(\mbox{\rm div},D)\times L^2(D).$

%From boundedness of the sequence $u^{h_{n_k}}$ in $H$ and estimate \eqref{m7nn2} it follows boundedness of the sequence $(\mathbf z_1^{h_n}(\cdot;u^{h_n}),z_2^{h_n}(\cdot;u^{h_n}))$ %in  $H(\mbox{\rm div},D)\times L^2(D)$.
%that
%from the sequence $(\mathbf z_1^{h_n}(\cdot;u^{h_n}),z_2^{h_n}(\cdot;u^{h_n}))$
%Therefore, we can extract from this sequence a subsequence $(\mathbf z_1^{h_{n_k}}(\cdot;u^{h_{n_k}}),z_2^{h_{n_k}}(\cdot;u^{h_{n_k}}))$ weakly convergent to a certain element %$(\tilde{\mathbf z}_1,\tilde z_2)$
%converges to $(\mathbf z_1(\tilde  u),z_2(\tilde u))$
%in the space $H(\mbox{\rm div},D)\times L^2(D).$
%Show that
%the  $(\mathbf z_1^{h_{n_k}}(\cdot;u^{h_{n_k}}),z_2^{h_{n_k}}(\cdot;u^{h_{n_k}}))$ weakly converges to
%$(\tilde{\mathbf z}_1,\tilde z_2)=(\mathbf z_1(\tilde  u),z_2(\tilde u))$.
%in the space $H(\mbox{\rm div},D)\times L^2(D).$
In fact,
take a subsequence
$\{(\mathbf z_1^{h_{n_{k_i}}}(\cdot;u^{h_{n_{k_i}}}),z_2^{h_{n_{k_i}}}(\cdot;u^{h_{n_{k_i}}}))\}$ of the sequence
$\{(\mathbf z_1^{h_{n_k}}(\cdot;u^{h_{n_k}}),z_2^{h_{n_k}}(\cdot;u^{h_{n_k}}))\}$
which weakly converges to some $(\tilde {\mathbf z}_1,\tilde z_2)$ in $H(\mbox{\rm div},D)\times L^2(D)$
and
for an arbitrary $(\mathbf q,v)$ from $H(\mbox{\rm div},D)\times L^2(D)$
%$v$ an arbitrary function from $L^2(D)$
take a sequence $\{(\mathbf q^{h_{n_{k_i}}},v^{h_{n_{k_i}}})\}, (\mathbf q^{h_{n_{k_i}}},v^{h_{n_{k_i}}})\in V_1^{h_{n_{k_i}}}\times V_2^{h_{n_{k_i}}}$
%of $\{(\mathbf q^{h_n},v^{h_n})\}, (\mathbf q^{h_n},v^{h_n})\in V_1^{h_n}\times V_2^{h_n}$
which strongly converges to $(\mathbf q,v)$ in $H(\mbox{\rm div},D)\times L^2(D)$
\footnote{Such sequences exist due to the boundedness of the sequence $\{(\mathbf z_1^{h_{n_k}}(\cdot;u^{h_{n_k}}),z_2^{h_{n_k}}(\cdot;u^{h_{n_k}}))\}$
in the space $H(\rm div;D)\times L^2(D)$, which follows from inequality \eqref{m7nn2} and the boundedness of the sequence $\{u^{h_{n_k}}\}$ in the space $H$, and from completeness of the sequence of the subspaces $\{V_1^h\times V_2^h\}$
%defined by an infinite set of parameters $h_1,h_2,\dots$ with $\lim_{k\to \infty}h_k=0$
in
$H(\rm div;D)\times L^2(D).$}
and
%put
%$\mathbf z_1^{h_n}(\cdot;u^{h_n})=
%$n=n_{k_i}$ in equations \eqref{}, \eqref{}.
%Then
pass to the limit in both sides of equations
\begin{multline}\label{113h''}
\int_{D}(((\mathbf A(x))^{-1})^T\mathbf z_1^{h_{n_{k_i}}}(x;u^{h_{n_{k_i}}}),\mathbf
q^{h_{n_{k_i}}}(x))_{\mathbb R^n}dx- \int_{D}z_2^{h_{n_{k_i}}}(x;u^{h_{n_{k_i}}})\mbox{div\,}\mathbf
q^{h_{n_{k_i}}}(x)\,dx\\= \int_{D}(\mathbf l_1(x)-(C_1^tJ_{H_1}u_1^{h_{n_{k_i}}})(x),\mathbf q^{h_{n_{k_i}}}(x))_{\mathbb
R^n}dx,
\end{multline}
\begin{multline}\label{115h''}
-\int_{D}v^{h_{n_{k_i}}}(x)\mbox{div\,}\mathbf
z_1^{h_{n_{k_i}}}(x;u^{h_{n_{k_i}}})\,dx-\int_{D}c(x)z_2^{h_{n_{k_i}}}(x;u^{h_{n_{k_i}}})v^{h_{n_{k_i}}}(x)\,dx\\=\int_{D}(l_2(x)-(C_2^tJ_{H_2}u_2^{h_{n_{k_i}}})(x))
v^{h_{n_{k_i}}}(x)\,dx
\end{multline}
(which follows from \eqref{113h}, \eqref{115h}), when $i\to\infty$.
Taking into account that
%we get that $(\tilde {\mathbf z}_1,\tilde z_2)$ is a solution  of equations \eqref{}, \eqref{} at $u=\tilde u.$
%let $\mathbf q$ be an arbitrary function from $H(\mbox{\rm div},D)$. Then since the sequence of the subspaces $\{V_1^h\}$
%defined by an infinite set of parameters $h_1,h_2,\dots$ with $\lim_{k\to \infty}h_k=0$
%is complete in
%$H(\rm div;D)$ there exists a sequence $\{\mathbf q^{h_{n_k}}\}$
%of functions $\mathbf q^{h_{n_{k_i}}}\in V_1^{h_{n_{k_i}}}$ such that
%$\|\mathbf q-\mathbf q^{h_{n_{k_i}}}\|_{H(\mbox{\rm\small div},D)}\to 0
%$
%as $i\to \infty$ and we get
\footnote{Passage to the limit in \eqref{m200h}--\eqref{m203h} is justified by the following assertion (see, for example \cite{Badr}, page 12):

{\it Let a sequence} $\{v_n\}$ {\it weakly converge to} $v_0$ {\it in some linear normed space} $X$ {\it and a sequence} $\{F_n\}$ {\it strongly converge to}
$F_0$ {\it in the space} $X',$ {\it dual of}  $X.$ {\it Then}
$$
\lim_{n\to\infty}<F_n,u_n>_{X'\times X}=<F_0,u_0>_{X'\times X}.
$$  }
%$$
%\lim_{i\to\infty}\Bigl (a^*(\mathbf z_1^{h_{n_{k_i}}}(\cdot;u^{h_{n_{k_i}}}),\mathbf q^{h_{n_{k_i}}})+b(\mathbf q^{h_{n_{k_i}}},z_2^{h_{n_{k_i}}}(\cdot;u^{h_{n_{k_i}}}))\Bigr )
%$$
$$
\lim_{i\to\infty}\Bigl(\int_{D}(((\mathbf A(x))^{-1})^T\mathbf z_1^{h_{n_{k_i}}}(x;u^{h_{n_{k_i}}}),\mathbf
q^{h_{n_{k_i}}}(x))_{\mathbb R^n}dx- \int_{D}z_2^{h_{n_{k_i}}}(x;u^{h_{n_{k_i}}})\mbox{div\,}\mathbf
q^{h_{n_{k_i}}}(x)\,dx\Bigr)
$$
$$
=\lim_{i\to\infty}a(\mathbf q^{h_{n_{k_i}}},\mathbf z_1^{h_{n_{k_i}}}(\cdot;u^{h_{n_{k_i}}}))+\lim_{i\to\infty}b(\mathbf q^{h_{n_{k_i}}},z_2^{h_{n_{k_i}}}(\cdot;u^{h_{n_{k_i}}}))$$
$$
=\lim_{i\to\infty}<A\mathbf q_1^{h_{n_{k_i}}},\mathbf z_1^{h_{n_{k_i}}}(\cdot;u^{h_{n_{k_i}}})>_{H(\mbox{\rm\small div},D)'\times H(\mbox{\rm\small div},D)}
$$
$$
+\lim_{i\to\infty}<B\mathbf q^{h_{n_{k_i}}},z_2^{h_{n_{k_i}}}(\cdot;u^{h_{n_{k_i}}})>_{L^2(D)'\times L^2(D)}
$$
$$
=<A\mathbf q,\tilde{\mathbf z}_1>_{H(\mbox{\rm\small div},D)'\times H(\mbox{\rm\small div},D)}+<B\mathbf q,\tilde{z}_2>_{L^2(D)'\times L^2(D)}
=a(\mathbf q,\tilde{\mathbf z}_1)+b(\mathbf q,\tilde z_2)
$$
\begin{equation}\label{m200h}
%a^*(\tilde{\mathbf z}_1,\mathbf q)+b(\mathbf q,\tilde z_2),
=\int_{D}(((\mathbf A(x))^{-1})^T\tilde {\mathbf z}_1(x),\mathbf
q(x))_{\mathbb R^n}dx- \int_{D}\tilde z_2(x)\mbox{div\,}\mathbf
q(x)\,dx,
\end{equation}
where by $A:H(\mbox{\rm div},D)\to H(\mbox{\rm div},D)'$ we denote the bounded operator associated with the bilinear
form $a(\cdot,\cdot),$ defined by $a(u,v)=<Au,v> $ $\forall u,v \in H(\mbox{\rm div},D)$,
$$
-\lim_{i\to\infty}\int_{D}v^{h_{n_{k_i}}}(x)\mbox{div\,}\mathbf
z_1^{h_{n_{k_i}}}(x;u^{h_{n_{k_i}}})\,dx=\lim_{i\to\infty}b(\mathbf z_1^{h_{n_{k_i}}}(\cdot;u^{h_{n_{k_i}}}),v^{h_{n_{k_i}}})
$$
$$
=\lim_{i\to\infty}<B^tv^{h_{n_{k_i}}},\mathbf z_1^{h_{n_{k_i}}}(\cdot;u^{h_{n_{k_i}}}))>_{H(\mbox{\rm\small div},D)'\times H(\mbox{\rm\small div},D)}
$$
\begin{equation}\label{m201h}
=b(\tilde{\mathbf z}_1,v)=-\int_{D}v(x)\mbox{div\,}\tilde{\mathbf z}_1(x)\,dx,
\end{equation}
$$
\lim_{i\to\infty}\int_{D}c(x)z_2^{h_{n_{k_i}}}(x;u^{h_{n_{k_i}}})v^{h_{n_{k_i}}}(x)\,dx=\lim_{i\to\infty}(z_2^{h_{n_{k_i}}}(\cdot;u^{h_{n_{k_i}}}),cv^{h_{n_{k_i}}})_{L^2(D)}
$$
\begin{equation}\label{m2011h}
=(\tilde z_2,cv)_{L^2(D)}=\int_{D}c(x)\tilde z_2(x)v(x)\,dx,
\end{equation}
$$
\lim_{i\to\infty}\int_{D}(\mathbf l_1(x)-(C_1^tJ_{H_1}u_1^{h_{n_{k_i}}})(x),\mathbf q^{h_{n_{k_i}}}(x))_{\mathbb
R^n}dx
$$
$$
=\lim_{i\to\infty}\Bigl((\mathbf l_1,\mathbf q^{h_{n_{k_i}}})_{L^2(D)^n}
-<J_{H_1}C_1\mathbf q^{h_{n_{k_i}}},u_1^{h_{n_{k_i}}}>_{H_1'\times H_1} \Bigr)
$$
\begin{equation}\label{m202h}
=(\mathbf l_1,\mathbf q)_{L^2(D)^n}
-<J_{H_1}C_1\mathbf q,\tilde u_1>
_{H_1'\times H_1}
=\int_{D}(\mathbf l_1(x)-(C_1^tJ_{H_1}\tilde u_1)(x),\mathbf q(x))_{\mathbb
R^n}dx,
\end{equation}
%where $v$ is an arbitrary function from $L^2(D),$ $\{v^{h_{n_{k_i}}}\}$ is a sequence
%of functions $v^{h_{n_{k_i}}}\in V_2^{h_{n_{k_i}}}$ such that
%$\|v-v^{h_{n_{k_i}}}\|_{L^2(D)}\to 0
%$
%as $i\to \infty.$
%Similary,
\begin{multline}\label{m203h}
\lim_{i\to\infty}\int_{D}(l_2(x)-(C_2^tJ_{H_2}u_2^{h_{n_{k_i}}})(x))v^{h_{n_{k_i}}}(x)\,dx\\=\int_{D}(l_2(x)-(C_2^tJ_{H_2}\tilde u_2)(x))v(x)\,dx,
\end{multline}
we see, from \eqref{113h''} -- \eqref{m203h}, that $(\tilde{\mathbf z}_1,\tilde z_2)\in H(\mbox{div,\,}D)\times L^2(D)$
%$\hat{z}_2, p_2\in L^2(D)$ and
%$\hat{\mathbf z}_1^{h}, \mathbf p_1^{h}\in V_1^{h},$ $\hat{z}_2^{h}, p_2^{h}\in V_2^{h}$ be solutions
satisfy equations \eqref{113} and \eqref{115} at $u=\tilde u$. But problem \eqref{113}, \eqref{115} has a unique solution $(\mathbf z_1(\tilde  u),z_2(\tilde u))$ at
$u=\tilde u.$ Hence $(\tilde{\mathbf z}_1,\tilde z_2)=(\mathbf z_1(\tilde  u),z_2(\tilde u))$ and
%the sequence
\begin{equation*}\label{m203hrr}
(\mathbf z_1^{h_{n_k}}(\cdot;u^{h_{n_k}}),z_2^{h_{n_k}}(\cdot;u^{h_{n_k}}))\to (\mathbf z_1(\tilde  u),z_2(\tilde u))\quad \mbox{weakly in}\quad H(\mbox{\rm div},D)\times L^2(D).
\end{equation*}
%the  $(\mathbf z_1^{h_{n_k}}(\cdot;u^{h_{n_k}}),z_2^{h_{n_k}}(\cdot;u^{h_{n_k}}))$ weakly converges to
%Taking this into account as well as the fact, that
Then, since the functionals $F_1(z_2):=(Q^{-1}z_2,z_2)_{L^2(D)}$ and $F_2(u):=(\tilde Q^{-1}u,u)_{H}:=(\tilde Q_1^{-1}u_1,u_1)_{H_1}+(\tilde Q_2^{-1}u_2,u_2)_{H_2}$ are weakly lower
semicontinuous in the spaces $L^2(D)$ and $H$, respectively,\footnote{These assertions are the corollary of a more general statement
(that can be found, for example, in \cite{Badr}, p. 41):
{\it Let} $X$ {\it be a reflexive Banach space, and} $B:X \to X^*$ {\it a linear bounded nonnegative selfadjoint operator. Then the functional} $F(u):=<Bu,u>_{X^*\times X}$
 {\it is a weakly lower
semicontinuous on} $X$.}
%\eqref{m203hrr} implies that
%В силу напівнеперевності знизу функціонала $I(u)$ в слабкій топології простора $H$ випливає, що
we obtain
$$
I(\tilde u)=(Q^{-1}z_2(\cdot;\tilde u),z_2(\cdot;\tilde u))_{L^2(D)}
+ (\tilde Q_1\tilde u,\tilde u)_{H}
$$
$$
\leq \underline{\lim}_{k\to\infty}(Q^{-1}z_2^{h_{n_k}}(\cdot;u^{h_{n_k}}),z_2^{h_{n_k}}(\cdot;u^{h_{n_k}}))_{L^2(D)}
+\underline{\lim}_{k\to\infty}(\tilde Q^{-1}u^{h_{n_k}},u^{h_{n_k}})_H
$$
$$
\leq \underline{\lim}_{k\to\infty} \Bigl[(Q^{-1}z_2^{h_{n_k}}(\cdot;u^{h_{n_k}}),z_2^{h_{n_k}}(\cdot;u^{h_{n_k}}))_{L^2(D)}
+(\tilde Q^{-1}u^{h_{n_k}},u^{h_{n_k}})_H\Bigr]
$$
\begin{equation}\label{m203hrrrr}
=\underline{\lim}_{k\to\infty}I_{n_k}(u^{h_{n_k}})\leq \overline{\lim}_{k\to\infty}I_{n_k}(u^{h_{n_k}})
\leq I(\hat u).
\end{equation}
Here $\tilde Q^{-1}: H\to H$ is the bounded selfadjoint positive definite operator defined by
$$
\tilde Q^{-1}u=\tilde Q^{-1}_1u_1+\tilde Q^{-1}_2u_2,\quad u=(u_1,u_2)\in H=H_1\times H_2,
$$
satisfying the inequality
\begin{equation}\label{poz}
(\tilde Q^{-1}u,u)_H\geq\alpha\|u\|_H^2\quad\forall u\in H,
\end{equation}
where $\alpha$ is a constant from \eqref{posdefme}.
Taking into account the uniqueness of an element on which the minimum of functional $I(u)$ is attained,
we find from \eqref{m203hrrrr} that $\tilde u=\hat u.$
%то внаслідок єдиності елемента, на якому досягається мінімум функціонала $I(u),$ одержуємо що $\tilde u=\hat u.$
This implies that
\begin{equation}\label{m204h'}
\lim_{n\to\infty}I_n(u^{h_n})= I(\hat u)
\end{equation}
and
$
u^{h_n}\xrightarrow{\mbox{\scriptsize weakly}}\hat u$ in $H,$ $\hat z_2^{h_n}=z_2^{h_n}(\cdot;u^{h_n})\xrightarrow{\mbox{\scriptsize weakly}} z_2(\cdot;\hat u)=\hat z_2$ in  $L^2(D)$ as $n\to\infty.$
Hence,
\begin{equation}\label{m204h}
(Q^{-1}z_2(\cdot;\hat u),z_2(\cdot;\hat u))_{L^2(D)}\leq \underline{\lim}_{n\to\infty}(Q^{-1}z_2^{h_n}(\cdot;u^{h_n}),z_2^{h_n}(\cdot;u^{h_n}))_{L^2(D)},
\end{equation}
\begin{equation}\label{m205h}
 (\tilde Q_1\hat u,\hat u)_{H}\leq \underline{\lim}_{n\to\infty}
(\tilde Q^{-1}u^{h_n},u^{h_n})_H
\end{equation}
and from \eqref{m204h}, \eqref{m205h}, we have
$$
\underline{\lim}_{n\to\infty}(Q^{-1}z_2^{h_n}(\cdot;u^{h_n}),z_2^{h_n}(\cdot;u^{h_n}))_{L^2(D)}+\underline{\lim}_{n\to\infty}
(\tilde Q^{-1}u^{h_n},u^{h_n})_H
$$
$$
\geq (Q^{-1}z_2(\cdot;\hat u),z_2(\cdot;\hat u))_{L^2(D)}+(\tilde Q_1\hat u,\hat u)_{H}= I(\hat u)
$$
$$
=\lim_{n\to\infty}\Bigl[Q^{-1}z_2^{h_n}(\cdot;u^{h_n}),z_2^{h_n}(\cdot;u^{h_n}))_{L^2(D)}+(\tilde Q^{-1}u^{h_n},u^{h_n})_H \Bigr]
$$
$$
=\overline{\lim}_{n\to\infty}\Bigl[Q^{-1}z_2^{h_n}(\cdot;u^{h_n}),z_2^{h_n}(\cdot;u^{h_n}))_{L^2(D)}+(\tilde Q^{-1}u^{h_n},u^{h_n})_H \Bigr]
$$
$$
\geq\underline{\lim}_{n\to\infty}Q^{-1}z_2^{h_n}(\cdot;u^{h_n}),z_2^{h_n}(\cdot;u^{h_n}))_{L^2(D)}+\overline{\lim}_{n\to\infty}(\tilde Q^{-1}u^{h_n},u^{h_n})_H.
$$
Whence
$$
\underline{\lim}_{n\to\infty}
(\tilde Q^{-1}u^{h_n},u^{h_n})_H\geq \overline{\lim}_{n\to\infty}(\tilde Q^{-1}u^{h_n},u^{h_n})_H.
$$
The last inequality shows that the sequence $\{(\tilde Q^{-1}u^{h_n},u^{h_n})_H\}$ is convergent. This fact and \eqref{m204h'} also imply convergence of the sequence $\{(Q^{-1}z_2^{h_n}(\cdot;u^{h_n}),z_2^{h_n}(\cdot;u^{h_n}))_{L^2(D)}\}$
and equality
\begin{equation}\label{contr0}
I(\hat u)=\lim_{n\to\infty}(Q^{-1}z_2^{h_n}(\cdot;u^{h_n}),z_2^{h_n}(\cdot;u^{h_n}))_{L^2(D)}+\lim_{n\to\infty}(\tilde Q^{-1}u^{h_n},u^{h_n})_H.
\end{equation}
It is easy to see that
\begin{equation}\label{contr}
\lim_{n\to\infty}(\tilde Q^{-1}u^{h_n},u^{h_n})_H=(\tilde Q^{-1}\hat u,\hat u)_H.
\end{equation}
In fact, if we suppose that \eqref{contr} does not hold, i.e.
\begin{equation*}\label{contr1}
\lim_{n\to\infty}(\tilde Q^{-1}u^{h_n},u^{h_n})_H=(\tilde Q^{-1}\hat u,\hat u)_H+a,
\end{equation*}
where $a$ is a certain positive number, then (due to \eqref{contr0})
there must be valid
%the following
\begin{multline}\label{contr2}
\underline{\lim}_{n\to\infty}(Q^{-1}z_2^{h_n}(\cdot;u^{h_n}),z_2^{h_n}(\cdot;u^{h_n}))_{L^2(D)}\\=\lim_{n\to\infty}(Q^{-1}z_2^{h_n}(\cdot;u^{h_n}),z_2^{h_n}(\cdot;u^{h_n}))_{L^2(D)}
=(Q^{-1}z_2(\cdot;\hat u,z_2(\cdot;\hat u))_{L^2(D)}-a.
\end{multline}
But this is impossible since \eqref{contr2} leads to the contradictory inequality
$$
\underline{\lim}_{n\to\infty}(Q^{-1}z_2^{h_n}(\cdot;u^{h_n}),z_2^{h_n}(\cdot;u^{h_n}))_{L^2(D)}<(Q^{-1}z_2(\cdot;\hat u,z_2(\cdot;\hat u))_{L^2(D)}.
$$
Hence, \eqref{contr} is proved.

Now let us show that $u^{h_n} \to \hat u$ strongly in $H.$  To this end introduce
Hilbert space $\tilde H$ consisting of elements of $H$ endowed with norm
$$
\|v\|_{\tilde H}:=(\tilde Q^{-1}v,v)_H^{1/2}.
$$
Then from weak convergence of the sequence $\{u^{h_n}\}$ to $\hat u$ as $n\to \infty$, it follows, obviously, that
\begin{equation} \label{tildd}
u^{h_n} \to \hat u\,\,\mbox{weakly in}\,\, \tilde H\,\, \mbox{as}\,\, n\to \infty.
\end{equation}
Since \eqref{contr} means that
\begin{equation} \label{contrr}
\|u^{h_n}\|_{\tilde H}\to \|\hat u\|_{\tilde H}\,\, \mbox{as}\,\, n\to \infty,
\end{equation}
we obtain from \eqref{tildd} and \eqref{contrr} that
%which when combined with the fact that $u^{h_n} \to \hat u$ weakly in $H$ implies
$u^{h_n} \to \hat u$ strongly in $\tilde H$ i.e.,\footnote{{Here we use the following statement (see, for example \cite{Yosida}, p. 124).
\it Let} $\{f_n\}$ {\it be a sequence in Hilbert space} $X$. {\it If} $f_n \to f$ {\it weakly in} $X$ {\it and}
$\|f_n\|_X\to \|f\|_X$ {\it as} $n\to \infty$ {\it then} $f_n \to f$ {\it strongly in} $X.$ }
\begin{equation*} \label{tild}
\lim_{n\to\infty}\|u^{h_n}-\hat u\|_{\tilde H}=\lim_{n\to\infty}(\tilde Q^{-1}(u^{h_n}-\hat u),u^{h_n}-\hat u)_H^{1/2}=0.
\end{equation*}
From here, due to the inequality
$$
\|u^{h_n}-\hat u\|_{H}\leq \frac{1}{\alpha}(\tilde Q^{-1}(u^{h_n}-\hat u),u^{h_n}-\hat u)_H^{1/2},
$$
following from \eqref{poz}, we find
%then from \eqref{tild} it follows
that
$$
\lim_{n\to\infty}\|u^{h_n}-\hat u\|_{H}=0,
$$
i.e. the sequence $\{u^{h_n}\}$ strongly converges to $\hat u$ in $H.$

In order to get estimate \eqref{r15aat1hggk}, we note that
$$
(\mathbf z_1^{h_n}(\cdot;\hat u)-\mathbf z_1^{h_n}(\cdot;u^{h_n}),z_2^{h_n}(\cdot;\hat u)-z_2^{h_n}(\cdot; u^{h_n}))
$$
is a solution of the following problem
\begin{multline}\label{1133hhh}
\int_{D}(((\mathbf A(x))^{-1})^T(\mathbf z_1^{h_n}(x;\hat u)-\mathbf z_1^{h_n}(x;u^{h_n})),\mathbf
q^{h_n}(x))_{\mathbb R^n}dx \\- \int_{D}(z_2^{h_n}(x;\hat u)-z_2^{h_n}(x; u^{h_n}))\mbox{div\,}\mathbf
q^{h_n}(x)\,dx\\=\int_{D}((C_1^tJ_{H_1}(u_1^{h_n}-\hat u_1))(x),\mathbf q^{h_n}(x))_{\mathbb
R^n}dx\quad \forall \mathbf q^{h_n}\in V_1^{h_n},
\end{multline}
\begin{multline}\label{1153hhh}
-\int_{D}v^{h_n}(x)\mbox{div\,}(\mathbf
z_1^{h_n}(x;\hat u)-\mathbf z_1^{h_n}(x;u^{h_n}))\,dx\\-\int_{D}c(x)(z_2^{h_n}(x;\hat u)-z_2^{h_n}(x; u^{h_n}))v^{h_n}(x)\,dx\\=\int_{D}(C_2^tJ_{H_2}(u_2^{h_n}-\hat u_2))(x)v^{h_n}(x)\,dx\quad
\forall v^{h_n}\in V_2^{h_n}.
\end{multline}
Applying estimate \eqref{m7nn2} to the solution of problem \eqref{1133hhh}, \eqref{1153hhh}, we obtain
\begin{multline} \label{r15aat1hggkk}
\|\mathbf z_1^{h_n}(\cdot;\hat u)-\mathbf z_1^{h_n}(\cdot;u^{h_n})\|_{H(\mbox{\rm \small div},D)}
\\+\|z_2^{h_n}(\cdot;\hat u)-z_2^{h_n}(\cdot; u^{h_n})\|_{L^2(D)}\leq C \|u_1^{h_n}-\hat u_1\|_H.
\end{multline}
From triangle inequality, \eqref{hatzz}, \eqref{r15aat1hggkk}, and the fact that the sequence $\{(\mathbf z_1^{h_n}(\cdot;\hat u),z_2^{h_n}(\cdot;\hat u))\}$
strongly converges to $(\mathbf z_1(\hat u),z_2(\hat u))$ in the space $H(\mbox{\rm div},D)\times L^2(D),$ we have
$$
\|\hat{\mathbf z}_1-
\hat{\mathbf z}_1^{h_n}\|_{H(\mbox{\rm \small div},D)}
+\|\hat z_2-\hat{z}_2^{h_n}\|_{L^2(D)}
$$
$$
\leq\|\mathbf z_1(\cdot;\hat u)-\mathbf z_1^{h_n}(\cdot;\hat u)\|_{H(\mbox{\rm \small div},D)}
+\|z_2(\cdot;\hat u)-z_2^{h_n}(\cdot; \hat u)\|_{L^2(D)}
$$
\begin{multline} \label{r15aat1hggkkr}
+\|\mathbf z_1^{h_n}(\cdot;\hat u)-\mathbf z_1^{h_n}(\cdot;u^{h_n})\|_{H(\mbox{\rm \small div},D)}
\\+\|z_2^{h_n}(\cdot;\hat u)-z_2^{h_n}(\cdot; u^{h_n})\|_{L^2(D)}\to 0 \quad\mbox{\rm as}\quad n\to \infty.
\end{multline}
Analogously, in order to obtain estimate \eqref{r15aat1hggk'}, we note that
$$
\|\mathbf p_1-
\mathbf p_1^{h_n}\|_{H(\mbox{\rm \small div},D)}
+\|p_2-p_2^{h_n}\|_{L^2(D)}
$$
$$
\leq\|\mathbf p_1-\mathbf p_1^{h_n}(\cdot;\hat u)\|_{H(\mbox{\rm \small div},D)}
+\|p_2-p_2^{h_n}(\cdot; \hat u)\|_{L^2(D)}
$$
\begin{equation} \label{jjhjj}
+\|\mathbf p_1^{h_n}(\cdot;\hat u)-\mathbf p_1^{h_n}\|_{H(\mbox{\rm \small div},D)}
+\|p_2^{h_n}(\cdot;\hat u)-p_2^{h_n}\|_{L^2(D)},
\end{equation}
where
$(\mathbf p_1^{h_n}(\cdot;\hat u),p_2^{h_n}(\cdot; \hat u))$ is a solution of the problem
\begin{multline}\label{r16at1hggv}
\int_{D}((\mathbf A(x))^{-1}\mathbf {p}_1^{h_n}(x;\hat u),\mathbf q_2^{h_n}(x))_{\mathbb
R^n}dx\\- \int_{D}p_2^{h_n}(x;\hat u)\mbox{\rm div\,}\mathbf q_2^{h}(x)\,dx=0
 \quad
\forall \mathbf q_2^{h_n}\in V_1^{h_n},
\end{multline}
\begin{multline} \label{r15aat1hggv}
-\int_{D}v_2^{h_n}(x)\mbox{\rm div\,}\mathbf p_1^{h_n}(x;\hat u)\,dx
-\int_{D}c(x)p_2^{h_n}(x;\hat u)v_2^{h_n}(x)\,dx
\\=\int_{D}v_2^{h_n}(x)Q^{-1}z_2(x;\hat u)\,dx \quad v_2^{h_n}\in V_2^{h_n}.
\end{multline}
Taking into account that, due to \eqref{r16at1hgg},\eqref{r15aat1hgg} and \eqref{r16at1hggv}, \eqref{r15aat1hggv},
$$
(\mathbf p_1^{h_n}(\cdot;\hat u)-\mathbf p_1^{h_n},p_2^{h_n}(\cdot;\hat u)-p_2^{h_n})
$$
is a solution of the following problem
\begin{multline}\label{1133hhhv}
\int_{D}((\mathbf A(x))^{-1}(\mathbf p_1^{h_n}(x;\hat u)-\mathbf p_1^{h_n}(x)),\mathbf
q_2^{h_n}(x))_{\mathbb R^n}dx\\- \int_{D}(p_2^{h_n}(x;\hat u)-p_2^{h_n}(x))\mbox{div\,}\mathbf
q_2^{h_n}(x)\,dx=0\quad \forall \mathbf q_2^{h_n}\in V_1^{h_n},
\end{multline}
\begin{multline}\label{1153hhhv}
-\int_{D}v_2^{h_n}(x)\mbox{div\,}(\mathbf
p_1^{h_n}(x;\hat u)-\mathbf p_1^{h_n}(x))\,dx-\int_{D}c(x)(p_2^{h_n}(x;\hat u)-p_2^{h_n}(x))v_2^{h_n}(x)\,dx\\
=\int_{D}v_2^{h_n}(x)Q^{-1}(z_2(\cdot;\hat u)-z_2^{h_n}(\cdot;u^{h_n}))(x)\,dx \quad v_2^{h_n}\in V_2^{h_n},
\end{multline}
and applying relationship \eqref{m7nnn} to the solution of problem \eqref{r16at1hggv}, \eqref{r15aat1hggv}
and estimate \eqref{m7nnl} to the solution of problem \eqref{1133hhhv}, \eqref{1153hhhv}, respectively, we obtain, in view of \eqref{r15aat1hggkkr}, that
\begin{equation} \label{r15aat1hggkkb}
\|\mathbf p_1-\mathbf p_1^{h_n}(\cdot;\hat u)\|_{H(\mbox{\rm \small div},D)}
+\|p_2-p_2^{h_n}(\cdot; \hat u)\|_{L^2(D)}\to 0 \quad\mbox{\rm as}\quad n\to \infty,
\end{equation}
\begin{multline} \label{r15aat1hggkks}
\|\mathbf p_1^{h_n}(\cdot;\hat u)-\mathbf p_1^{h_n})\|_{H(\mbox{\rm \small div},D)}
+\|p_2^{h_n}(\cdot;\hat u)-p_2^{h_n})\|_{L^2(D)}\\ \leq C \|z_2(\cdot;\hat u)-\hat z_2^{h_n}\|_{L^2(D)}
\to 0 \quad\mbox{\rm as}\quad n\to \infty.
\end{multline}
From \eqref{r15aat1hggkks}, \eqref{r15aat1hggkkb}, and
\eqref{jjhjj},
we find
\begin{equation} \label{jjjf}
\|\mathbf p_1-
\mathbf p_1^{h_n}\|_{H(\mbox{\rm \small div},D)}
+\|p_2-p_2^{h_n}\|_{L^2(D)}\to 0 \quad\mbox{\rm as}\quad n\to \infty.
\end{equation}
Relationships \eqref{jjjf} and \eqref{r15aat1hggkkr} mean that \eqref{r15aat1hggk} and \eqref{r15aat1hggk'} are proved.

Now show the validity of \eqref{r15aat1hggk''} and \eqref{r15aat1hggk'''}.

Let $(\tilde{\mathbf j},\tilde\varphi)$ be a solution of problem \eqref{generalized},
\eqref{generalized1} at $f(x)=\tilde f(x).$ Then from \eqref{apprest} and \eqref{exactest},
we have
$$
\mathbb E|\widehat {l^{h_n}(\tilde {\mathbf j},\tilde \varphi)}-\widehat{\widehat {l(\tilde {\mathbf j},\tilde \varphi)}}|^2
=\mathbb E[(u_1^{h_n},\tilde y_1)_{H_1}+(u_2^{h_n},\tilde y_2)_{H_1}+c^{h_n}-(\hat u_1,\tilde y_1)_{H_1}-(\hat u_2,\tilde y_2)_{H_2}-\hat c]^2
$$
$$
=\mathbb E[(u_1^{h_n}-\hat u_1,\tilde y_1)_{H_1}+(u_2^{h_n}-\hat u_2,\tilde y_2)_{H_2}+c^{h_n}- \hat c]^2
$$
$$
=[(u_1^{h_n}-\hat u_1,C_1\tilde {\mathbf j})_{H_1}+(u_2^{h_n}-\hat u_2,C_2\tilde \varphi)_{H_2}+c^{h_n}- \hat c]^2
$$
\begin{equation}\label{lhm2}
+\mathbb E[(u_1^{h_n}-\hat u_1,\tilde \eta_1)_{H_1}+(u_2^{h_n}-\hat u_2,\tilde \eta_2)_{H_2}].
%:=J_1+J_2.
\end{equation}
Weak convergence of the sequence $\{\hat z_2^{h_n}\}$  to $\hat z_2$ in the space $L^2(D)$ implies that $c^{h_n}\to \hat c$ as $n\to \infty.$ Then from the fact that $\tilde f\in G_0$
and the inequality
$$
[(u_1^{h_n}-\hat u_1,C_1\tilde {\mathbf j})_{H_1}+(u_2^{h_n}-\hat u_2,C_2\tilde \varphi)_{H_2}+c^{h_n}- \hat c]^2
$$
$$
\leq C \left(\|u_1^{h_n}-\hat u_1\|_{H_1}^2+\|u_2^{h_n}-\hat u_2\|_{H_2}^2+(c^{h_n}- \hat c)^2 \right)\left(\|\tilde {\mathbf j}\|^2_{H(\mbox{\rm \small div},D)}
+\|\tilde\varphi\|^2_{L^2(D)}\right)
$$
$$
\leq\tilde  C\left(\|u^{h_n}-\hat u\|_{H}^2+(c^{h_n}- \hat c)^2 \right)\|\tilde f\|_{L^2(D)}^2
$$
$$
\leq\tilde{\tilde  C}\left(\|u^{h_n}-\hat u\|_{H}^2+(c^{h_n}- \hat c)^2 \right)\quad (C,\tilde C,\tilde{\tilde C}=\mbox{const}),
$$
we see that the fist term in the r.h.s of \eqref{lhm2} tends to $0$
as $n\to\infty.$
Analogously, we may show that for the last term
in the r.h.s of \eqref{lhm2} the following estimate is valid
$$
\mathbb E[(u_1^{h_n}-\hat u_1,\tilde \eta_1)_{H_1}+(u_2^{h_n}-\hat u_2,\tilde \eta_2)_{H_2}]
\leq C\|u^{h_n}-\hat u\|_{H}^2 \quad (C=\mbox{const})
$$
and therefore this term also tends to $0$ as $n\to\infty.$
From here and the inequality
$$
\mathbb E|\widehat {l^{h_n}(\tilde {\mathbf j},\tilde \varphi)}-l(\tilde {\mathbf j},\tilde \varphi)|^{1/2}
=\mathbb E[\widehat {l^{h_n}(\tilde {\mathbf j},\tilde \varphi)}-\widehat{\widehat {l(\tilde {\mathbf j},\tilde \varphi)}}
+\widehat{\widehat {l(\mathbf j,\varphi)}}
-l(\tilde {\mathbf j},\tilde \varphi)|^{1/2}
$$
$$
\leq \left\{\mathbb E|\widehat {l^{h_n}(\tilde {\mathbf j},\tilde \varphi)}-\widehat{\widehat {l(\tilde {\mathbf j},\tilde \varphi)}}|^2\right\}^{1/2}
+\left\{\mathbb E[
\widehat{\widehat {l(\tilde {\mathbf j},\tilde \varphi)}}
-l(\tilde {\mathbf j},\tilde \varphi)]^2|\right\}^{1/2},
$$
it follows the validity of the conclusion of the theorem.
\end{proof}

Let us formulate a similar result in the case when an estimate $(\hat{\mathbf j},\hat\varphi)$ of the state $(\mathbf j,\varphi)$
is directly determined from the solution to problem
\eqref{13aaag}--\eqref{18aaagq}.
\begin{pred}
Let $(\hat{\mathbf j}^{h},\hat\varphi^{h})\in V_1^{h}\times V_2^{h}$ be an approximate estimate of
%state of the system
$(\hat{\mathbf j},\hat\varphi)$
determined from the solution to the variational problem
\begin{multline} \label{13aaagg}
\int_{D}(((\mathbf A(x))^{-1})^T\hat{\mathbf p}_1^{h}(x),\mathbf
q_1^{h}(x))_{\mathbb R^n}dx- \int_{D}\hat p_2^{h}(x)\mbox{\rm
div\,}\mathbf q_1^{h}(x)\,dx\\ =\int_{D}(C_1^tJ_{H_1}\tilde
Q_1(y_1(x)-C_1\hat{\mathbf
j}^{h}),\mathbf q_1^{h}(x))_{\mathbb R^n}dx
\,\,\forall\mathbf q_1^{h}\in V_1^{h},
\end{multline}
\begin{multline} \label{p14aaagg}
-\int_{D}v_1^{h}(x)\mbox{\rm div\,}\hat{\mathbf p}_1^{h}(x)\,dx
-\int_{D}c(x)\hat p_2^{h}(x)v_1^{h}(x)\,dx
\\=\int_D(C_2^tJ_{H_2}\tilde
Q_2(y_2(x)-C_2\hat
\varphi^{h}(x))v_1^{h}(x)\,dx\quad\forall v_1^{h}\in V_2^{h},
\end{multline}
\begin{equation}\label{pr16at1gg}
\int_{D}((\mathbf A(x))^{-1}\hat{\mathbf j}^{h}(x),\mathbf
q_2^{h}(x))_{\mathbb R^n}\,dx-\int_{D}\hat\varphi^{h}(x)\mbox{\rm
div\,}\mathbf q_2^{h}(x)\,dx=0\quad
\forall \mathbf q_2^{h}\in V_1^{h},
\end{equation}
\begin{multline} \label{18aaagg}
-\int_{D}v_2^{h}(x)\mbox{\rm div\,}\hat{\mathbf j}^{h}(x)\,dx
-\int_{D}c(x)\hat\varphi^{h}(x)v_2^{h}(x)\,dx
\\=\int_{D}v_2(x)(Q^{-1}\hat p_2^{h}(x)+f_0(x))\,dx \quad
\forall v_2^{h}\in V_2^{h}.
\end{multline}
Then
$$
\|\hat{\mathbf j}-
\hat{\mathbf j}^{h}\|_{H(\mbox{\rm \small div},D)}
+\|\hat\varphi-\hat{\varphi}^{h}\|_{L^2(D)}\to 0 \quad\mbox{\rm as}\quad h\to 0
$$
and
$$
\|\hat{\mathbf p}_1-\hat{\mathbf p}_1^{h}\|_{H(\mbox{\rm \small div},D)}
+\|\hat p_2-\hat p_2^{h}\|_{L^2(D)}\to 0 \quad\mbox{\rm as}\quad h\to 0.
$$
\end{pred}

Introducing the bases in the spaces $V_1^h$ and $V_2^h$, problem \eqref{r13at1hgg}--\eqref{r15aat1hgg} can be rewritten as a system of
liner algebraic equations.  To do this, let us denote the elements of the bases of $V_1^h$ and $V_2^h$ by $\boldsymbol \xi_i$ $(i=1,\dots,n_1)$ and $\eta_i$ $(i=1,\dots,n_2),$
respectively,
where $n_1=\mbox{\rm dim\,}V_1^h,$ $n_2=\mbox{\rm dim\,}V_2^h.$ The fact that $\hat{\mathbf z}_1^{h},$ $\mathbf p_1^{h}$
and $\hat z_2^{h},$ $p_2^{h}$ belong to the spaces $V_1^h$ and $V_2^h$ means the existence of constants $\hat{z}_i^{(1)},$ $p_i^{(1)}$ and $\hat z_i^{(2)},$ $p_i^{(2)}$ such that
\begin{equation}\label{elr}
\hat{\mathbf z}_1^{h}=\sum_{j=1}^{n_1}\hat{z}_j^{(1)}\boldsymbol \xi_j,\quad \mathbf p_1^{h}=\sum_{j=1}^{n_1}p_j^{(1)}\boldsymbol \xi_j
\end{equation}
and
\begin{equation}\label{elr1}
\hat z_2^h=\sum_{j=1}^{n_2}\hat z_j^{(2)}\eta_j,\quad p_2^{h}=\sum_{j=1}^{n_2}p_j^{(2)}\eta_j.
\end{equation}
Setting in \eqref{r13at1hgg}, \eqref{r16at1hgg} $\mathbf q_1^{h}=\mathbf q_2^{h}=\boldsymbol \xi_i$
$(i=1,\dots,n_1)$ and in \eqref{r14at1hgg}, \eqref{r15aat1hgg}
$v_1^{h}=v_2^{h}=\eta_i$ $(i=1,\dots,n_2)$ respectively,
we obtain  that finding $\hat{\mathbf z}_1^{h},$ $\mathbf p_1^{h},$ $z_2^h$ and $p_2^{h}$ from \eqref{r13at1hgg}--\eqref{r15aat1hgg}
is equivalent to solving the following system of linear algebraic equations with respect to coefficients
%систему \eqref{r13at1hggtt}--\eqref{r15aat1hggtt} можна записати у вигляді
$\hat{z}_j^{(1)},$ $p_j^{(1)},$ $\hat z_j^{(2)},$ and $p_j^{(2)}$ of expansions \eqref{elr}, \eqref{elr1}:
\begin{equation} \label{r13at1hggtttt}
\sum_{j=1}^{n_1}\bar a^{(1)}_{ij}\hat{z}_j^{(1)}
+ \sum_{j=1}^{n_2}a^{(2)}_{ji}\hat z_j^{(2)}
+\sum_{j=1}^{n_1}a^{(3)}_{ij}p_j^{(1)}=b^{(1)}_{i},\quad i=1,\dots,n_1,
\end{equation}
\begin{equation} \label{r14at1hggtttt}
\sum_{j=1}^{n_1}a^{(2)}_{ij}\hat{z}_j^{(1)}+\sum_{j=1}^{n_2}a^{(6)}_{ij}\hat z_j^{(2)}
+\sum_{j=1}^{n_2}a^{(4)}_{ij}p_j^{(2)}
=b^{(1)}_{i},\quad i=1,\dots,n_2,
\end{equation}
\begin{equation}\label{r16at1hggtttt}
\sum_{j=1}^{n_1}a^{(1)}_{ij}p_j^{(1)}+ \sum_{j=1}^{n_2}a^{(2)}_{ji}p_j^{(2)}=0,\quad i=1,\dots,n_1,
\end{equation}
\begin{equation} \label{r15aat1hggtttt}
\sum_{j=1}^{n_1}a^{(2)}_{ij}p_j^{(1)}+\sum_{j=1}^{n_2}a^{(6)}_{ij} p_j^{(2)}
+\sum_{j=1}^{n_2}a^{(5)}_{ij}\hat z_j^{(2)}=0,\quad i=1,\dots,n_2,
\end{equation}
where
$$\bar a^{(1)}_{ij}=\int_{D}(((\mathbf A(x))^{-1})^T
\boldsymbol \xi_i(x),
\boldsymbol \xi_j(x))_{\mathbb R^n}dx,\quad i,j=1,\dots,n_1,$$
$$a^{(1)}_{ij}=\int_{D}((\mathbf A(x))^{-1}
\boldsymbol \xi_i(x),
\boldsymbol \xi_j(x))_{\mathbb R^n}dx,\quad i,j=1,\dots,n_1,$$
$$a^{(2)}_{ij}=-\int_{D}\eta_i(x)\mbox{\rm
div\,}\boldsymbol \xi_j(x)\,dx,\quad i=1,\dots,n_2,\quad j=1,\dots,n_1,$$
$$a^{(3)}_{ij}=\int_{D}(C_1^tJ_{H_1}\tilde Q_1C_1
\boldsymbol \xi_i(x),\boldsymbol \xi_j(x))_{\mathbb R^n}\,dx,\quad i,j=1,\dots,n_1,$$
%$$a^{(4)}_{ij}=\int_{D}\eta_j(x)\mbox{\rm div\,}\boldsymbol \xi_i(x)\,dx,\quad i=1,\dots,n_1,\quad j=1,\dots,n_2,$$
$$a^{(4)}_{ij}=\int_{D}C_2^tJ_{H_2}\tilde Q_2C_2
\eta_i(x)\eta_j(x)\,dx,\quad i,j=1,\dots,n_2,$$
$$a^{(5)}_{ij}=-\int_{D}\eta_j(x)Q^{-1}\eta_i(x)\,dx,\quad i,j=1,\dots,n_2,$$
$$a^{(6)}_{ij}=-\int_{D}c(x)\eta_i(x)\eta_j(x)\,dx,\quad i,j=1,\dots,n_2,$$
$$b^{(1)}_{i}=\int_{D}(\mathbf l_1(x),\boldsymbol \xi_i(x))_{\mathbb R^n}\,dx,\quad i=1,\dots,n_1,$$
$$b^{(2)}_{i}=\int_{D}l_2(x)\eta_i(x)\,dx,\quad i=1,\dots,n_2.$$

Analogous system of linear algebraic equations can be also obtained for problem
\eqref{13aaagg}--\eqref{18aaagg}.

\section{The case of integral observation operators}
%{\bf 7 The case of integral observation operators}

As an example, we consider the case when  $H_1=L^2\bigl(D_1\n\bigr)^{n}\times\dots \times L^2\bigl(D_{i_1}\n\bigr)^{n}\times\dots\times L^2\bigl(D_{n_1}\n\bigr)^{n}$, $H_2=L^2\bigl(D_1\m\bigr)\times\dots\times L^2\bigl(D_{i_2}\m\bigr)\times\dots\times L^2\bigl(D_{n_2}\m\bigr)$. Then $J_{H_1}=I_{H_1}$, $J_{H_2}=I_{H_2}$,
where $I_{H_1}$ and $I_{H_2}$  are the identity operators in $H_1$ and $H_2$, respectively,
$$
y_1(x)=\Bigl(\mathbf y_1\n(x),\dots,\mathbf y\n_{i_1}(x),\dots,\mathbf y\n_{n_1}(x)\Bigr),
$$
\begin{equation*}\label{eta_1}
\eta_1(x)=\Bigl(\boldsymbol{\eta}\n_1(x),\dots,\boldsymbol{\eta}\n_{i_1}(x),\dots,\boldsymbol{\eta}\n_{n_1}(x)\Bigr),
\end{equation*}
where $\mathbf y\n_{i_1}(x)=(y\n_{i_1,1}(x),\ldots,y\n_{i_1,n}(x))^T\in
L^2\bigl(D_{i_1}\n\bigr)^{n}$, $\boldsymbol \eta\n_{i_1}(x)=(\eta\n_{i_1,1}(x),\ldots,\eta\n_{i_1,n}(x))^T
$  is a stochastic vector process with components
$\eta_{i_1,j}^{(1)}(x)$ $(j=1,\dots,n, i_1=1,\dots,n_1)$ that are stochastic processes with zero expectations and finite second moments,
$$
y_2(x)=\Bigl(y_1^{1}(x),\dots,y\m_{i_2}(x),\dots,y\m_{n_2}(x)\Bigr),
$$
\begin{equation}\label{eta_2}
\eta_2(x)=\Bigl(\eta\m_1(x),\dots,\eta\m_{i_2}(x),\dots,\eta\m_{n_2}(x)\Bigr),
\end{equation}
where $y\m_{i_2}\in L^2(D)$, $\eta\m_{i_2}(x)$ $(i_2=1,\dots,n_2)$  is a stochastic process with zero expectation and finite second moment.

Let in observations
(\ref{observ}) the operators $C_1:$ $L^2(D)^n\to H_1$ and $C_2:$ $L^2(D)\to H_2$ be defined by
$$
C_1\mathbf j(x)=\left(C_1\n\mathbf j(x), \ldots,C_{i_1}\n\mathbf j(x)\ldots,C_{n_1}\n\mathbf j(x)\right),
$$
$$
C_2 \varphi(x)=\left(C_1\m\varphi(x), \ldots,C_{i_2}\m\varphi(x)\ldots,C_{n_2}\m\varphi(x)\right),
$$
where $C_{i_1}\n:L^2(D)^{n}\to L^2(D_{i_1}\n)^{n}$ and $C_{i_2}\m:L^2(D)\to L^2(D_{i_2}\m)$ are integral operators defined by
$$C_{i_1}\n\mathbf j(x):=\int_{D_{i_1}\n}\mathbf K_{i_1}\n(x,\xi)\mathbf j(\xi)\,d\xi,$$
and
$$C_{i_2}\m\varphi(x):=\int_{D_{i_2}\m} K_{i_2}\m(x,\xi)\varphi(\xi)\,d\xi,
%\quad i_2=1\dots, n_2,
$$
correspondingly,
%$$
%C_1\mathbf j(x)=\left(\int_D\mathbf K_{1}\n(x,\xi)\mathbf j(\xi)\,d\xi, \ldots,\int_D\mathbf K_{i_1}\n(x,\xi)\mathbf j(\xi)\,d\xi, \ldots,\int_D\mathbf K_{n_1}\n(x,\xi)\mathbf %j(\xi)\,d\xi \right)^T\!\!\!\!,
%$$
$\mathbf K_{i_1}\n(x,\xi)=\{k^{(i_1)}_{is}(x,\xi)\}_{i,j=1}^{n}$ is a matrix with entries $k^{(i_1)}_{is}\in L^2(D_{i_1}\n)\times L^2(D_{i_1}\n),$
$i_1=1,\dots,n_1,$ $K_{i_2}\m(x,\xi)\in L^2(D_{i_2}\m)\times L^2(D_{i_2}\m)$ is a given function, $i_2=1,\dots,n_2$.
%a оператор  -- рівністю
%де $C_{i_2}\m:L^2(D)\to L^2(D_{i_2}\m)$ -- інтегральний оператор, заданий виразом
%$$C_{i_2}\m\varphi(x):=\int_{D_{i_2}\m} K_{i_2}\m(x,\xi)\varphi(\xi)\,d\xi,\quad i_2=1\dots, n_2,$$

As a result,  observations  $y_1$ and $y_2$ in (\ref{observ}) take the form
$$y_1=(\mathbf y_{1}\n(x),\ldots,\mathbf y_{i_1}\n(x),\ldots,\mathbf y_{n_1}\n(x)),$$
$$y_2=(y_{1}\m(x),\ldots,y_{i_2}\m(x),\ldots,\mathbf y_{n_2}\m(x)),$$
where
\begin{equation}\label{obint}
\mathbf y_{i_1}\n(x)=\int_{D_{i_1}\n}\mathbf K_{i_1}\n(x,\xi)
\mathbf j(\xi)\,d\xi+\boldsymbol{\eta}_{i_1}\n(x) ,\quad i_1=\overline{1,n_1},
\end{equation}
\begin{equation}\label{obint1}
y_{i_2}\m(x)=\int_{D_{i_2}\m}K_{i_2}\m(x,\xi)
\varphi(\xi)\,d\xi+\eta_{i_2}\m(x) ,\quad i_2=\overline{1,n_2},
\end{equation}
and  the operators
\begin{multline*}
\tilde Q_1\in \mathcal L\Bigl(L^2\bigl(D_1\n\bigr)^{n}\times\dots\times L^2\bigl(D_{i_1}\n\bigr)^{n}\times\dots\times L^2\bigl(D_{n_1}\n\bigr)^{n},\\
L^2\bigl(D_1\n\bigr)^{n}\times\dots\times L^2\bigl(D_{i_1}\n\bigr)^{n}\times\dots\times L^2\bigl(D_{n_1}\n\bigr)^{n}\Bigr)
\end{multline*}
and
\begin{multline*}
\tilde Q_2\in \mathcal L\Bigl(L^2\bigl(D_1\m\bigr)\times\dots\times L^2\bigl(D_{i_2}\m\bigr)\times\dots\times L^2\bigl(D_{n_2}\m\bigr),\\
L^2\bigl(D_1\m\bigr)\times\dots\times L^2\bigl(D_{i_2}\m\bigr)\times\dots\times L^2\bigl(D_{n_2}\m)\Bigr)
\end{multline*}
in (\ref{restr2}),
%і (\ref{sk11}),
which is contained in the definition of set $G_1$,
%і $G_0$
are given by
$$
\tilde Q_1\tilde \eta_1=(\tilde {\mathbf Q}\n_{1}\tilde {\boldsymbol {\eta}}\n_{1},\dots,\tilde
{\mathbf Q}\n_{r_1}\tilde {\boldsymbol {\eta}}\n_{r_1},\dots,\tilde
{\mathbf Q}\n_{n_1}\tilde {\boldsymbol {\eta}}\n_{n_1})
$$
and
$$
\tilde Q_2\tilde \eta_2=(\tilde {Q}\m_{2}\tilde {\eta}\m_{2},\dots,\tilde
{Q}\m_{r_2}\tilde {\eta}\m_{r_2},\dots,\tilde
{Q}\m_{n_2}\tilde {\eta}\m_{n_2}),
$$
%$$
%Q(\tilde{\mathbf f})=(\sum_{k=1}^nq_{1k}(x)\tilde f_k(t)(x)),\ldots,
%\sum_{k=1}^nq_{nk}(x)\tilde f _k(x))^T,
%$$
where $\tilde{\mathbf Q}\n_{r_1}(x)$
is a symmetric positive definite $n\times n$-matrix
%і $n$ відповідно
with entries $\tilde q_{ij}^{(r_1)}\in C(\bar D_{r_1}\n)$,\footnote{Here and below we denote by $C(\bar D)$ a class
of functions continuous in the domain
$\bar D.$} $i,j=1,\ldots,n$, $\tilde {\boldsymbol {\eta}}\n_{r_1}\in L^2(\Omega, L^2(D_{r_1}\n)^{n}),$
 $r_1=1,\ldots,n_1,$ $\tilde
{Q}\m_{r_2}(x)$ is a continuous positive function
defined in the domain $\bar D_{r_2}\m$,
$\tilde {\eta}\m_{r_2}\in L^2(\Omega, L^2(D_{r_2}\m)),$ $r_2=1,\ldots,n_2.$

In this case condition (\ref{restr2}) takes the form\footnote{By
$$\mbox{Sp}(\tilde{\mathbf Q}_{r_1}\n(x)\tilde
{\mathbf R}_{r_1}\n(x,x))$$ we denote the trace of matrix
$\tilde{\mathbf Q}_{r_1}\n(x)\tilde
{\mathbf R}_{r_1}\n(x,x)$, i.e. the sum of diagonal elements of this matrix.}
\begin{equation*}\label{8ns}
\sum_{r_1=1}^{n_1}\int_{D_{r_1}\n}\mbox{Sp}(\tilde{\mathbf Q}_{r_1}\n(x)\tilde
{\mathbf R}_{r_1}\n(x,x))\,dx\leq 1,\quad
%$$
%a умова (\ref{}) -- в умову
\sum_{r_2=1}^{n_2}\int_{D_{r_2}\m}\tilde Q_{r_2}\m(x)
\tilde R_{r_2}\m(x,x)\,dx\leq 1,
\end{equation*}
where by $\tilde{\mathbf R}_{r_1}\n(x,y)=[\tilde
b^{(r_1)}_{i,j}(x,y)]_{i,j=1}^{n}$ we denote the correlation matrix of vector process
$\tilde{\boldsymbol \eta}\n_{r_1}(x)=(\tilde\eta\n_{r_1,1}(x),\ldots,\tilde\eta\n_{r_1,n}(x))$
with components $$\tilde
b^{(r_1)}_{i,j}(x,y)=\mathbb
E\Bigl(\tilde\eta\n_{r_1,i}(x)\tilde\eta\n_{r_1,j}(y)\Bigr),\quad (x,y)\in D_{i_1}\n\times
D_{i_1}\n,$$
and by
$\tilde R_{r_2}\m(x,y)=\mathbb E\tilde
\eta_{r_2}\m(x) \tilde\eta_{r_2}\m(y)$ we denote the correlation function of process
$\tilde\eta_{r_2}\m(x),$
%$r_2=\overline{1,n_2},$
$(x,y)\in D_{r_2}\m\times
D_{r_2}\m.$

In fact,
$$
\mathbb E(\tilde Q_1\tilde \eta_1,\tilde \eta_1)_{H_1}= \sum_{r_1=1}^{n_1}\mathbb E(\tilde
{\mathbf Q}\n_{r_1}(x)\boldsymbol {\eta}\n_{r_1}(x),\boldsymbol {\eta}\n_{r_1}(x))_{L^2\bigl(D_{r_1}\n\bigr)^n}
$$
$$
=\sum_{r_1=1}^{n_1}\mathbb E\Bigl(\int_{D_{r_1}\n}\tilde
{\mathbf Q}\n_{r_1}(x)\boldsymbol {\eta}\n_{r_1}(x),\boldsymbol {\eta}\n_{r_1}(x))_{\mathbb R^n}dx \Bigr)
$$
$$
=\sum_{r_1=1}^{n_1}\sum_{i=1}^{n}\int_{D_{r_1}\n}\sum_{j=1}^{n}\mathbb E(\tilde q_{ij}^{(r_1)}(x)\eta\n_{j,r_1}(x)\eta\n_{i,r_1}(x))\,dx
$$
$$
=\sum_{r_1=1}^{n_1}\sum_{i=1}^{n}\int_{D_{r_1}\n}\sum_{j=1}^{n}\tilde q_{ij}^{(r_1)}(x)\mathbb E(\eta\n_{j,r_1}(x)\eta\n_{i,r_1}(x))\,dx
$$
$$
=\sum_{r_1=1}^{n_1}\int_{D_{r_1}\n}\sum_{i=1}^{n}\sum_{j=1}^{n}\tilde q_{ij}^{(r_1)}(x)\tilde
b^{(r_1)}_{j,i}(x,x)\,dx=
\sum_{r_1=1}^{n_1}\int_{D_{r_1}\n}\mbox{Sp}(\tilde{\mathbf Q}_{r_1}\n(x)\tilde
{\mathbf R}_{r_1}\n(x,x))\,dx.
$$
Analogously,
$$
\mathbb E(\tilde Q_2\tilde \eta_2,\tilde \eta_2)_{H_2}=\sum_{r_2=1}^{n_2}\mathbb E(\tilde
{Q}\m_{r_2}(x){\eta}\m_{r_2}(x),{\eta}\m_{r_2}(x))_{L^2\bigl(D_{r_2}\m\bigr)}
$$
$$
=\sum_{r_2=1}^{n_2}\int_{D_{r_2}\m}\tilde
{Q}\m_{r_2}(x)\mathbb E(\eta\m_{r_2}(x)\eta\m_{r_2}(x))\,dx=\sum_{r_2=1}^{n_2}\int_{D_{r_2}\m}\tilde Q_{r_2}\m(x)
\tilde R_{r_2}\m(x,x)\,dx.
$$
%тже, у наших припущеннях, множина $G_1$ буде складатися з елементів $\tilde\eta=(\tilde\eta_1,\tilde\eta_2)$, де
%$\tilde\eta_1=(\tilde{\boldsymbol{\eta}}\n_1,\ldots,\tilde{\boldsymbol{\eta}}\n_{r_1,1},\dots
%\tilde{\boldsymbol{\eta}}_{n_1}\n),$ $\tilde{\boldsymbol \eta}\n_{r_1}=(\tilde\eta\n_{r_1,1},\ldots,\tilde\eta\n_{r_1,n})\in L^2(\Omega,L^2(D)^n),$ $r_1=\overline{1,n_1},$ i %$\tilde\eta_2=(\tilde\eta\m_1,\ldots,\tilde\eta_{r_2}\m,\ldots,\tilde\eta_{n_2}\m),$ $\tilde\eta\m_{r_2}\in L^2(\Omega,L^2(D)),$
%$r_2=\overline{1,n_2},$
%що задовольняють наступним умовам:

%(i) $\mathbb E\tilde{\boldsymbol \eta}\n_{r_1}=0,$ $r_1=\overline{1,n_1},$
%$\mathbb E\tilde{\eta}\m_{r_2}=0,$ $r_2=\overline{1,n_2},$
%$\mathbb E(\tilde\eta_{i_1,j}\n(x)\tilde\eta_{i_2}\m(x))=0,$ тобто
%компоненти $\tilde\eta_{i_1,j}\n(x,t)$ векторних
%випадкових полів $\tilde{\boldsymbol{\eta}}_{i_1}\n(x,t),$
%$i_1=\overline{1,n_1},$ $j_1=\overline{1,n},$ некорельовані з
%випадковими полями $\tilde\eta_{i_2}\m(x,t),$
%$i_2=\overline{1,n_2};$

%(ii) кореляційні матриці $\tilde{\mathbf
%R}_{i_1}\n(x,y)$ полів
%$\tilde{\boldsymbol{\eta}}_{i_1}\n(x),$ $i_1=\overline{1,n_1},$ та
%кореляційні функції $\tilde R_{i_2}\m(x,y)$ полів
%$\tilde\eta_{i_2}\m(x),$ $i_2=\overline{1,n_2},$ задовольняють нерівностям:
%$$
%\sum_{i_1=1}^{n_1}\int_{D_{i_1}\n}\mbox{Sp\,}[\tilde{\mathbf
%Q}_{i_1}\n(x)\tilde{\mathbf R}_{i_1}\n(x,x)]\,dx\leq 1, \quad i_1=\overline{1,n_1},
%$$
%\begin{equation}\label{8ns}
% \sum_{i_1=1}^{n_1}\int_{D_{i_2}\m}\tilde Q_{i_2}\m(x)
%\tilde R_{i_2}\m(x,x)\,dx\leq 1,\quad i_2=\overline{1,n_2}.
%\end{equation}

Uncorrelatedness  of random variables $\tilde\eta_1$ and $\tilde\eta_2$ reduces in this case to the condition
of uncorrelatedness  of the componets $\tilde\eta_{i_1,j}\n$ of random vector fields
$\tilde{\boldsymbol{\eta}}_{i_1}\n,$
$i_1=\overline{1,n_1},$ $j=\overline{1,n},$ with
random fields $\tilde\eta_{i_2}\m,$
$i_2=\overline{1,n_2}$,
%тобто до умови
%$\mathbb E(\tilde\eta_{i_1,j}\n\tilde\eta_{i_2}\m)=0$,
%$x\in D,$
and hence the set
$G_1$ is described by the formula
$$
G_1=\Bigl\{\tilde\eta=(\tilde\eta_1,\tilde\eta_2):
\tilde\eta_1=(\tilde{\boldsymbol{\eta}}\n_1,\ldots,\tilde{\boldsymbol{\eta}}\n_{i_1,1},\dots
\tilde{\boldsymbol{\eta}}_{n_1}\n), \tilde{\boldsymbol \eta}\n_{i_1}=(\tilde\eta\n_{i_1,1},\ldots,\tilde\eta\n_{i_1,n})
$$
%\\[-40pt]
$$
\in L^2(\Omega,L^2(D_{i_1}\n)^n), \tilde\eta_2=(\tilde\eta\m_1,\ldots,\tilde\eta_{i_2}\m,\ldots,\tilde\eta_{n_2}\m),\tilde\eta\m_{i_2}\in L^2(\Omega,L^2(D_{i_2}\m)),
$$
$$
\mathbb E\tilde{\boldsymbol \eta}\n_{i_1}=0,
\mathbb E\tilde{\eta}\m_{i_2}=0,
%\mathbb E(\tilde\eta_{i_1,j}\n(x)\tilde\eta_{i_2}\m(x))=0,x\in D,
%\mbox{компоненти}
\tilde\eta_{i_1,j}\n \,\, \mbox{and} \,\,
\tilde\eta_{i_2}\m\,\,\mbox{are uncorrelated},\,\,
%i_2=\overline{1,n_2},
j=\overline{1,n},
i_1=\overline{1,n_1}, i_2=\overline{1,n_2};
$$
%\\[-40pt]
\begin{equation}\label{G_1ap}
\sum_{i_1=1}^{n_1}\int_{D_{i_1}\n}\mbox{Sp\,}[\tilde{\mathbf
Q}_{i_1}\n(x)\tilde{\mathbf R}_{i_1}\n(x,x)]\,dx\leq 1,
\,\,
\sum_{i_1=1}^{n_1}\int_{D_{i_2}\m}\tilde Q_{i_2}\m(x)
\tilde R_{i_2}\m(x,x)\,dx\leq 1
\Bigr\}.
\end{equation}

It is easily verified that the operator $C_1^t:L^2\bigl(D_1\n\bigr)^{n}\times\dots \times L^2\bigl(D_{i_1}\n\bigr)^{n}\times\dots\times L^2\bigl(D_{n_1}\n\bigr)^{n}
 \to L^2(D)^n,$ transpose of
$C_1,$ is defined by
$C_{1}^t\psi_1(x)=\sum_{l_1=1}^{n_1}\chi_{D_{l_1}\n}(x)\int_{D_{l_1}\n}
[\mathbf K_{l_1}\n(\xi,x)]^T\boldsymbol\psi_{l_1}\n(\xi)\,d\xi,$
where
$$
\psi_1=(\boldsymbol\psi_1\n,\dots, \boldsymbol\psi_{l_1}\n,\dots,\boldsymbol\psi_{n_1}\n),\quad \boldsymbol\psi_{l_1}\in L^2\bigl(D_{l_1}\n\bigr)^n,\quad  l_1=1,\dots, n_1,
$$
and the operator $C_2^t:L^2\bigl(D_1\m\bigr)\times\dots\times L^2\bigl(D_{i_2}\m\bigr)\times\dots\times L^2\bigl(D_{n_2}\m\bigr) \to L^2(D),$ transpose of
$C_2,$ is defined by
$C_{2}^t\psi_2(x)=\sum_{l_2=1}^{n_2}\chi_{D_{l_2}\m}(x)\int_{D_{l_2}\m}
K_{l_2}\m(\xi,x)\psi_{l_2}\m(\xi)\,d\xi,$
where
$$
\psi_2=(\psi_1\m,\dots, \psi_{l_2}\m,\dots,\psi_{n_2}\m),\quad \psi_2\in L^2(D_{l_2}\m),\quad  l_2=1,\dots, n_2,
$$
and $\chi(M)$ is a characteristic function of the set
$M\subset \mathbb R^n$.

Since
%Taking into account that
$$
\hat u_1=\tilde Q_1C_1\mathbf
p_1=(\hat{\mathbf u}_{1}^1,\dots, \hat{\mathbf u}_{l_1}^1,\dots,\hat{\mathbf u}_{n_1}^1),\,\, \hat{\mathbf u}_{l_1}^1
\in L^2\bigl(D_{l_1}\n\bigr)^n,\,\, l_1=1,\dots, n_1,
$$
$$
\hat u_2=\tilde Q_2C_2
p_2=(\hat{u}_{1}^2,\dots, \hat{u}_{l_2}^2,\dots,\hat{u}_{n_2}^2),\,\, \hat{u}_{l_2}^2
\in L^2(D_{l_2}\m),\,\, l_2=1,\dots, n_2,
$$
where
\begin{equation}\label{uuns1apx}
\hat{\mathbf
u}_{i_1}\n(\cdot) =\tilde{\mathbf Q}_{i_1}\n(\cdot)
\int_{D_{l_1}\n}\mathbf K_{i_1}\n(\cdot,\eta)\mathbf
p_1(\eta)\,d\eta,\quad i_1=1,\dots, n_1,
\end{equation}
\begin{equation}\label{uuns2ap}
\hat
u_{i_2}\m(\cdot)=\tilde{ Q}_{i_2}\m(\cdot)
\int_{D_{l_2}\m}K_{i_2}\m(\cdot,\eta)
p_2(\eta)\,d\eta,\quad i_2=1,\dots, n_2,
\end{equation}
we find
$$
C_1^tJ_{H_1}\tilde Q_1C_1\mathbf
{p}_1(\cdot)=C_1^t\hat u_1(\cdot)=\sum_{l_1=1}^{n_1}\chi_{D_{l_1}\n}(\cdot)\int_{D_{l_1}\n}
[\mathbf K_{l_1}\n(\xi,\cdot)]^T\hat{\mathbf u}_{l_1}(\xi)\,d\xi
$$
$$
=\sum_{l_1=1}^{n_1}\chi_{D_{l_1}\n}(\cdot)\int_{D_{l_1}\n}
[\mathbf K_{l_1}\n(\xi,\cdot)]^T\tilde{\mathbf Q}_{l_1}\n(\xi)
\int_{D_{l_1}\n}\mathbf K_{l_1}\n(\xi,\xi_1)\mathbf
p_1(\xi_1)\,d\xi_1\,d\xi=
$$
\begin{equation}\label{vir}
=\sum_{l_1=1}^{n_1}\chi_{D_{l_1}\n}(\cdot)\int_{D_{l_1}\n}
\tilde{\mathbf K}_{l_1}\n(\cdot,\xi_1)\mathbf p_1(\xi_1)
\,d\xi_1,
\end{equation}
\begin{equation}\label{virr}
C_2^tJ_{H_2}\tilde Q_2C_2
p_2(\cdot)
=\sum_{l_2=1}^{n_2}\chi_{D_{l_2}\m}(\cdot)\int_{D_{l_2}\m}
\tilde{K}_{l_2}\m(\cdot,\xi_1)p_2(\xi_1)
\,d\xi_1,
\end{equation}
where\footnote{We use the following notation: if $\mathbf A(\xi)=[a_{ij}(\xi)]_{i,j=1}^N$
is a matrix dependig on variable $\xi$ that varies on measurable set $\Omega,$ then we define $\int_\Omega \mathbf A(\xi)\,d\xi$ by the equality
$$
\int_\Omega \mathbf A(\xi)\,d\xi=\left[\int_\Omega a_{ij}(\xi)\,d\xi\right]_{i,j=1}^N.
$$}
$$
\tilde{\mathbf K}_{l_1}\n(\cdot,\xi_1)=\int_{D_{l_1}\n}(\mathbf K_{l_1}\n(\xi,\cdot))^T
\tilde{\mathbf Q}_{l_1}\n(\xi)\mathbf K_{l_1}\n(\xi,\xi_1)d\xi,
$$
$$
\tilde{K}_{l_2}\m(\cdot,\xi_1)=\int_{D_{l_1}\n}K_{l_2}\m(\xi,\cdot))
\tilde{Q}_{l_2}\n(\xi)K_{l_2}\m(\xi,\xi_1)d\xi.
$$

A class of linear with respect of observations (\ref{obint}) and (\ref{obint1}) estimates $\widehat{l(\mathbf j,
\varphi)}$ will take the form
\begin{equation}\label{uunsap}
\widehat{l(\mathbf j,\varphi)}=
\sum_{i_1=1}^{n_1}\int_{D_{i_1}\n}({\mathbf
u}_{i_1}\n(x),\mathbf y_{i_1}\n(x))_{\mathbb R^n}\,dx +
\sum_{i_2=1}^{n_2}\int_{D_{i_2}\m}
u_{i_2}\m(x)y_{i_2}\m(x)\,dx+ c.
\end{equation}
Thus,
%з проведеного вище аналізу та
taking into account  \eqref{uuns1apx}--\eqref{uunsap}, we obtain that,
under assumptions (\ref{h11h}), (\ref{G_1ap}), and (\ref{etaG_1}),
the following result is valid for integral observation operators as a corollary from Theorems 1 and 2.

\begin{pred}
The guaranteed estimate $\widehat{\widehat{l(\mathbf j,
\varphi)}}$ of $l(\mathbf {j},\varphi)$ is determined by the formula
\begin{equation*}
\widehat{\widehat{l(\mathbf j,
\varphi)}}=\sum_{i_1=1}^{n_1}\int_{D_{i_1}\n}(\hat{\mathbf
u}_{i_1}\n(x),\mathbf y_{i_1}\n(x))_{\mathbb R^n}\,dx +
\sum_{i_2=1}^{n_2}\int_{D_{i_2}\m}\hat
u_{i_2}\m(x)y_{i_2}\m(x)\,dx+\hat c=l(\hat{\mathbf j},
\hat\varphi),
\end{equation*}
where
\begin{equation*}\label{oooons}
\hat c=-\int_{D}\hat {z}_2(x)
f_0(x)\,dx,
\end{equation*}
\begin{equation*}\label{uuns1ap}
\hat{\mathbf
u}_{i_1}\n(x) =\tilde{\mathbf Q}_{i_1}\n(x)
\int_{D_{i_1}\n}\mathbf K_{i_1}\n(x,\eta)\mathbf
p_1(\eta)\,d\eta,\quad i_1= \overline{1,n_1},
\end{equation*}
\begin{equation*}\label{1uuns2ap}
\hat
u_{i_2}\m(x)=\tilde{ Q}_{i_2}\m(x)
\int_{D_{i_2}\m}K_{i_2}\m(x,\eta)
p_2(\eta)\,d\eta,\quad i_2= \overline{1,n_2},
\end{equation*}
and functions $\mathbf p_1\in H(\mbox{\rm div},D),$ $\hat z_2,p_2\in L^2(D)$ and $\hat{\mathbf j}\in H(\mbox{\rm div},D),$
$\hat\varphi\in L^2(D)$ are found from solution to systems of variational equations
\begin{multline*}
\int_{D}(((\mathbf A(x))^{-1})^T\hat{\mathbf z}_1(x),\mathbf
q_1(x))_{\mathbb R^n}dx- \int_{D}\hat z_2(x)\mbox{\rm
div\,}\mathbf q_1(x)\,dx
\\=\int_{D}\Bigl(\mathbf l_1(x)-\sum_{i_1=1}^{n_1}\chi_{D_{i_1}\n}(x)\int_{D_{i_1}\n}
\tilde{\mathbf K}_{i_1}\n(x,\xi_1)\mathbf p_1(\xi_1)
\,d\xi_1,\mathbf q_1(x)\Bigr)_{\mathbb R^n}
\,dx \,\,\forall\mathbf
q_1\in H(\mbox{\rm div},D),
\end{multline*}
\begin{multline*} \label{r14at1}
-\int_{D}v_1(x)\mbox{\rm div\,}\hat{\mathbf z}_1(x)\,dx-\int_{D}c(x)\hat z_2(x)v_1(x)\,dx
\\=\int_{D}\Bigl(l_2(x)-\sum_{i_2=1}^{n_2}\chi_{D_{i_2}\m}(x)\int_{D_{i_2}\m}
\tilde{K}_{i_2}\m(x,\xi_1)p_2(\xi_1)
\,d\xi_1\Bigr)v_1(x)\,dx\,\,\,\forall v_1\in L^2(D),
\end{multline*}
\begin{equation*}\label{r16at1}
\int_{D}((\mathbf A(x))^{-1}\mathbf {p}_1(x),\mathbf q_2(x))_{\mathbb
R^n}dx\!-\!\int_{D}p_2(x)\mbox{\rm div\,}\mathbf q_2(x)\,dx=0
\,\,\,\,
\forall \mathbf q_2\in H(\mbox{\rm div},D),
\end{equation*}
\begin{multline*}
-\int_{D}v_2(x)\mbox{\rm div\,}\hat{\mathbf p}_1(x)\,dx
-\int_{D}c(x)\hat p_2(x)v_2(x)\,dx\\
=\int_{D}v_2(x)Q^{-1}\hat z_2(x)\,dx \quad \forall v_2\in
L^2(D).
\end{multline*}
and
\begin{multline*}
\int_{D}\!\!\!(((\mathbf A(x))^{-1})^T\hat{\mathbf p}_1(x),\mathbf
q_1(x))_{\mathbb R^n}dx- \int_{D}\hat p_2(x)\mbox{\rm
div\,}\mathbf q_1(x)\,dx\\=\int_{D}\Bigl(\mathbf d_1(x)-\sum_{i_1=1}^{n_1}\chi_{D_{i_1}\n}(x)\!\!\int_{D_{i_1}\n}\!\!\!
\tilde{\mathbf K}_{i_1}\n(x,\xi_1)\hat{\mathbf j}_1(\xi_1)
\,d\xi_1,\mathbf q_1(x)\Bigr)_{\mathbb R^n}
\!dx \,\,\forall\mathbf
q_1\in H(\mbox{\rm div},D),
\end{multline*}\\[-40pt]
\begin{multline*} \label{p14aaag}
-\int_{D}v_1(x)\mbox{\rm div\,}\hat{\mathbf p}_1(x)\,dx
-\int_{D}c(x)\hat p_2(x)v_1(x)\,dx
\\=\int_{D}\Bigl(d_2(x)-\sum_{i_2=1}^{n_2}\chi_{D_{i_2}\m}(x)\int_{D_{i_2}\m}
\tilde{K}_{i_2}\m(x,\xi_1)\hat\varphi(\xi_1)
\,d\xi_1\Bigr)v_1(x)\,dx\,\,\,\forall v_1\in L^2(D),
\end{multline*}
\begin{equation*}\label{pr16at1g}
\int_{D}((\mathbf A(x))^{-1}\hat{\mathbf j}(x),\mathbf
q_2(x))_{\mathbb R^n}\,dx- \int_{D}\hat\varphi(x)\mbox{\rm
div\,}\mathbf q_2(x)\,dx=0\quad
\forall \mathbf q_2\in H(\mbox{\rm div},D),
\end{equation*}
\begin{multline*} \label{18aaag}
-\int_{D}v_2(x)\mbox{\rm div\,}\hat{\mathbf j}(x)\,dx
-\int_{D}c(x)\hat\varphi(x)v_2(x)\,dx
\\=\int_{D}v_2(x)(Q^{-1}\hat p_2(x)+f_0(x))\,dx \quad
\forall v_2\in L^2(D),
\end{multline*}
respectively. Here $\hat{\mathbf z}_1, \hat{\mathbf p}_1\in H(\mbox{\rm div},D),$ $\hat p_2\in L^2(D)$ and
$$
\mathbf d_1(x)=\sum_{i_1=1}^{n_1}\chi_{D_{i_1}\n}(x)\int_{D_{i_1}\n}(\mathbf K_{i_1}\n(\xi,x))^T \tilde{
\mathbf Q}_{i_1}\n(\xi)\mathbf y_{i_1}\n(\xi)\,d\xi,
$$
$$
d_2(x)=\sum_{i_2=1}^{n_2}\chi_{D_{i_2}\m}(x)\int_{D_{i_2}\m}K_{i_2}\m(\xi,x) \tilde{
Q}_{i_2}\m(\xi)y_{i_2}\m(\xi)\,d\xi.
$$

The estimation error $\sigma$ is given by the expression
\begin{equation*}\label{rllns}
\sigma=l(\mathbf {p}_1,p_2)^{1/2}.
\end{equation*}
\end{pred}
\section{Minimax estimation of linear functionals from right-hand sides of elliptic equations: Representations for guaranteed estimates and estimation errors}
%{\bf 8 Minimax estimation of linear functionals from right-hand sides of elliptic equations: Representations for guaranteed estimates and estimation errors}

The problem is to determine a minimax estimate of the value of the functional
\begin{equation}\label{linff}
l(f):=  \int_{D}l_0(x)f(x)\, dx
\end{equation}
from observations \eqref{observ}
%випадкових елементів вигляду
%\begin{equation}\label{observ1}
%y_1=C_1\mathbf j+\eta_1, \quad
%y_2=C_2\varphi+\eta_2,
%\end{equation}
%що належать сепарабельним гільбертовим просторам $H_1$ і $H_2$ над $\mathbb R$ відповідно,
%при умовах
%, що функція $f(x)$  у рівняннях
%\eqref{eqsystem2} i \eqref{generalized1} -- невідома і належить множині
%$G_0$, яка визначается спiввiдношенням
%\eqref{h11h} and \eqref{etaG_1},
in the class of estimates, linear with respect to observations,
\begin{equation}\label{clasf}
\widehat{l(f)}:=(y_1,u_1)_{H_1}
+(y_2,u_2)_{H_2}+c,
\end{equation}
where $u_1$ and $u_2$ are elements from Hilbert spaces $H_1$ and $H_2,$ respectively, $c\in \mathbb R,$
%$(\mathbf j,\varphi)$ -- невідомий розв'язок
%задачі \eqref{eqsystem1}--\eqref{eqsystem2},
$l_0\in L^2(D)$ is a given
function, under the assumption
that $f\in G_0$ and $\eta\in G_1,$ where sets $G_0$ and $G_1$ are defined on page \pageref{s5}.
% \eqref{h11h} and \eqref{etaG_1},
%respectively.
%$u_1\in H_1,$ $u_2\in
%H_2,$ $c\in \mathbb R,$ $C_1\in\mathcal L(L^2(D)^n,H_1)$ і
%$C_2\in\mathcal L(L^2(D),H_2)$ -- лінійні неперервні оператори,
%$(\eta_1,\eta_2)\in G_1,$ a через $G_1$ позначено множину
%випадкових елементів $\tilde \eta_1=\tilde \eta_1(\omega)\in L^2(\Omega,H_1)$ і $\tilde
%\eta_2=\tilde \eta_2(\omega)\in L^2(\Omega,H_2)$ з нульовими середніми, що задовольняють
%умову \eqref{restr2}.
\begin{predll}
The  estimate of the form
\begin{equation}\label{defmef}
\widehat{\widehat {l(f)}}=(y_1,\hat u_1)_{H_1}+(y_2,\hat
u_2)_{H_2}+\hat c
\end{equation}
will be called the guaranteed estimate of $l(f)$ if the elements
$\hat u_1\in H_1,$ $\hat u_2\in H_2$ and a number $\hat c$
are determined from the condition
\begin{equation*}\label{m16f}
\inf_{u\in H,\,c\in \mathbb R}\sigma(u,c)=\sigma(\hat u,\hat c),
\end{equation*}
where $u=(u_1,u_2)\in H=H_1\times H_2,$ $\hat u=(\hat u_1,\hat u_2)\in H,$
$$\sigma(u,c):=\sup_{\tilde f \in
G_0, (\tilde \eta_1,\tilde \eta_2)\in G_1}\mathbb E|l(\tilde f)-\widehat {l(\tilde f)}|^2,$$
\begin{equation}\label{mef}
\widehat{l(\tilde f)}:=(\tilde y_1,u_1)_{H_1}
+(\tilde y_2,u_2)_{H_2}+c,
\end{equation}
$\tilde y_1=C_1\tilde {\mathbf j}+\tilde \eta_1, \quad
\tilde y_2=C_2\tilde \varphi+\tilde \eta_2,$ and $(\mathbf {\tilde j},\tilde
\varphi)$ is a solution to problem
\eqref{eqsystem1}--\eqref{eqsystem2} when $f(x)=\tilde f(x)$.

The quantity
\begin{equation*}\label{mm16f}
\sigma:=[\sigma(\hat u,\hat c)]^{1/2}
% \varrho:=\{\mathbb
%E|l(\mathbf {j},\varphi)-\widehat{\widehat {l(\mathbf
%{j},\varphi)}}|^2\}^{1/2}
\end{equation*} is called the error of the guaranteed estimation of $l(f).$

\end{predll}

%\section{Другий розділ}

%\subsection
%{\bf Представлення для мінімаксних оцінок і похибок
%оцінювання.}

For any fixed $u:=(u_1,u_2)\in H,$ introduce a pair of functions $
(\mathbf z_1(\cdot;u),z_2(\cdot;u))\in
H(\mbox{div};D)\times L^2(D)
$
as a unique solution of the following problem:
\begin{multline}\label{113f}
\int_{D}(((\mathbf A(x))^{-1})^T\mathbf z_1(x;u),\mathbf
q(x))_{\mathbb R^n}dx- \int_{D}z_2(x;u)\mbox{div\,}\mathbf
q(x)\,dx=\\ -\int_{D}((C_1^tJ_{H_1}u_1)(x),\mathbf q(x))_{\mathbb
R^n}dx\quad \forall \mathbf q\in H(\mbox{div},D),
\end{multline}
\begin{multline}\label{115f}
\int_{D}v(x)\mbox{div\,}\mathbf
z_1(x;u)\,dx+\int_{D}c(x)z_2(x;u)v(x)\,dx\\=\int_{D}(C_2^tJ_{H_2}u_2)(x)v(x)\,dx\quad
\forall v\in L^2(D).
\end{multline}
 %З теорії варіаційних задач, які допускають змішанe формулювання
%(див., наприклад, \cite{Brezzi}), випливає, що
%из которой функції $\mathbf
%z_1(x;u),z_2(x;u)$ визначаються
%із рівнянь (\ref{113f})-- (\ref{115f})
%єдиним чином.
\begin{predl}
Finding the guaranteed estimate
of $l(f)$ is equivalent to the problem of optimal control of a system described by the problem
(\ref{113f}), (\ref{115f}) with cost function
\begin{multline}\label{m20f}
I(u)=\left(Q^{-1}(l_0-z_2(\cdot;u)),
l_0-z_2(\cdot;u)\right)_{L^2(D)}\\+ (\tilde
Q_1^{-1}u_1,u_1)_{H_1}+(\tilde Q_2^{-1}u_2,u_2)_{H_2}\!\to
\inf_{u\in H}.
\end{multline}
\end{predl}
\begin{proof}
Taking into account \eqref{linff} at$f=\tilde f$ and \eqref{mef}, we have
$$
l(\tilde f)-\widehat{l(\tilde f)}=
(l_0,\tilde f)_{L^2(D)}
-(\tilde y_1,u_1)_{H_1}
-(\tilde y_2,u_2)_{H_2}-c
$$
$$
=(l_0,\tilde f)_{L^2(D)}
-(u_1,C_1\tilde{\mathbf j}+\tilde\eta_1)_{H_1}-(u_2,C_2\tilde
\varphi+\tilde\eta_2)_{H_2}-c
$$
$$
=(l_0,\tilde f)_{L^2(D)}
-<J_{H_1}u_1,C_1\tilde{\mathbf j}>_{H_1'\times H_1}-<
J_{H_2}u_2,C_2\tilde \varphi>_{H_2'\times H_2}
$$
$$
 -(u_1,\tilde\eta_1)_{H_1}-(u_2,\tilde\eta_2)_{H_2}-c
$$
$$
=-(C_1^tJ_{H_1}u_1,\tilde{\mathbf j})_{L^2(D)^n}
-(C_2^t J_{H_2}u_2,\tilde{\varphi})_{L^2(D)}
$$
\begin{equation}\label{gh1f}
+(l_0,\tilde f)_{L^2(D)}
-(u_1,\tilde\eta_1)_{H_1}-(u_2,\tilde\eta_2)_{H_2}-c.
\end{equation}

Using a similar argument as in the proof of Lemma 1 in which the solution of problem \eqref{113}, \eqref{115} is substituted by the solution
of problem (\ref{113f}), (\ref{115f}), we obtain from \eqref{gh1f} the following representation
$$
l(\tilde f)-\widehat{l(\tilde f)}
=(\tilde
f,l_0-z_2(\cdot;u)_{L^2(D)}
-(u_1,\tilde\eta_1)_{H_1}-(u_2,\tilde\eta_2)_{H_2}-c
$$
$$
=(\tilde f-f_0,l_0-z_2(\cdot;u))_{L^2(D)}+(f_0,l_0-z_2(\cdot;u))_{L^2(D)}
$$
$$
-(u_1,\tilde\eta_1)_{H_1}-(u_2,\tilde\eta_2)_{H_2}-c.
$$
By virtue of \eqref{dispe}, we find from here
$$
\mathbb E\left|l(\tilde f)-\widehat
{(l(\tilde f)}\right|^2
 =\left|(\tilde
f_2-f_0,l_0-z_2(\cdot;u))_{L^2(D)}
+(f_0,l_0-z_2(\cdot;u))_{L^2(D)}-c\right|^2
$$
\begin{equation*}
+\mathbb E[(u_1,\tilde\eta_1)_{H_1}+(u_2,\tilde\eta_2)_{H_2}]^2.
\end{equation*}
From the latter equality, we obtain
$$
 \inf_{c \in
\mathbb R}\sup_{\tilde f\in G_0, (\tilde
\eta_1,\tilde \eta)\in G_1} \mathbb E|l(\tilde f
)-\widehat {l(\tilde
f)}|^2=
$$
$$
=\inf_{c \in \mathbb R}\sup_{\tilde f\in
G_0}\left[(\tilde f-f_0,l_0-z_2(\cdot;u))_{L^2(D)}+(f_0,l_0-z_2(\cdot;u))_{L^2(D)}-c\right]^2
$$
$$
+ \sup_{ (\tilde \eta_1,\tilde \eta_2)\in G_1}\mathbb
E[(\tilde\eta_1,u_1)_{H_1}+(\tilde\eta_2,u_2)_{H_2}]^2
$$
\begin{equation}\label{exhf}
=\sup_{\tilde f\in G_0}\left[
 (\tilde f-f^{(0)},l_0-z_2(\cdot;u))_{L^2(D)}\right]^2
+ \sup_{ (\tilde \eta_1,\tilde \eta_2)\in G_1}\mathbb
E[(\tilde\eta_1,u_1)_{H_1}+(\tilde\eta_2,u_2)_{H_2}]^2,
\end{equation}
where infimum over $c$ is attained at
$
c=(f_0,l_0-z_2(\cdot;u))_{L^2(D)}.
$
Cauchy$-$Bunyakovsky inequality and \eqref{h11h}
imply
$$
|(\tilde f-f_0,l_0-z_2(\cdot;u))_{L^2(D)}|^2
$$
$$
\leq
(Q^{-1}(l_0-z_2(\cdot;u)),l_0-z_2(\cdot;u))_{L^2(D)}
(Q(\tilde f-f_0),\tilde f-f_0)_{L^2(D)}
$$
$$
\leq
(Q^{-1}(l_0-z_2(\cdot;u)),l_0-z_2(\cdot;u))_{L^2(D)},
$$
where inequality  becomes an equality at
$$
\tilde f=f_0+\frac {Q^{-1}(l_0-z_2(\cdot;u))}{(Q^{-1}(l_0-z_2(\cdot;u)),l_0-z(\cdot;u))_{L^2(D)}^{1/2}}.
$$
Hence
$$
\sup_{\tilde f\in G_0}\left[
 (\tilde f_2-f_2^{(0)},l_0-z_2(\cdot;u))_{L^2(D)}\right]^2
 =(Q^{-1}(l_0-z_2(\cdot;u)),l_0-z_2(\cdot;u))_{L^2(D)}.
$$
Analoguosly, due to \eqref{restr2}, \eqref{rand4}, and \eqref{nekor}, we have
$$
 \sup_{ (\tilde \eta_1,\tilde \eta_2)\in G_1}\mathbb
E[(\tilde\eta_1,u_1)_{H_1}+(\tilde\eta_2,u_2)_{H_2}]^2
=(\tilde
Q_1^{-1}u_1,u_1)_{H_1}+(\tilde Q_2^{-1}u_2,u_2)_{H_2}.
$$
From two latter relations and (\ref{exhf}), we get
$$
 \inf_{c \in
\mathbb R}\sup_{\tilde f\in G_0, (\tilde
\eta_1,\tilde \eta)\in G_1} \mathbb E|l(\tilde f)-\widehat {l(
\tilde f)}|^2
=I(u),
$$
at $c=(
l_0-z_2(\cdot;u),f_0)_{L^2(D)},$ where $I(u)$
is determined by (\ref{m20f}).
\end{proof}
As a result of solving of optimal control problem
(\ref{113f}) -- (\ref{m20f}), we come to the following assertion.
\begin{pred} \label{t4}
There exists a unique estimate of $l(f)$
which has the form
\begin{equation}\label{haturr}
\widehat{\widehat {l(f)}}=(y_1,\hat u_1)_{H_1}+(y_2,\hat
u_2)_{H_2}+\hat c,
\end{equation}
where
\begin{equation}\label{haturr1}
\hat c=\int_{D}
(l_0(x)-\hat z_2(x))f_0(x)\,dx,\quad\hat u_1=\tilde Q_1C_1\mathbf
{p}_1,\quad\hat u_2= \tilde Q_2C_2 p_2,
\end{equation}
and the functions $\mathbf {p}_1\in H(\mbox{\rm
div},D)$ and $\hat z_2 ,p_2\in L^2(D)$ are found from solution
of the following variational problem:
\begin{multline} \label{r13at1f}
\int_{D}(((\mathbf A(x))^{-1})^T\hat{\mathbf z}_1(x),\mathbf
q_1(x))_{\mathbb R^n}dx- \int_{D}\hat z_2(x)\mbox{\rm
div\,}\mathbf q_1(x)\,dx
\\=-\int_{D}(C_1^tJ_{H_1}\tilde Q_1C_1\mathbf
{p}_1(x),\mathbf q_1(x))_{\mathbb R^n}\,dx \quad\forall\mathbf
q_1\in H(\mbox{\rm div},D),
\end{multline}
\begin{multline} \label{r14at1f}
\int_{D}v_1(x)\mbox{\rm div\,}\hat{\mathbf z}_1(x)\,dx+\int_{D}c(x)\hat z_2(x)v_1(x)\,dx
\\=\int_{D}(C_2^tJ_{H_2}\tilde Q_2C_2
p_2)(x)v_1(x)\,dx\quad\forall v_1\in L^2(D),
\end{multline}
\begin{multline}\label{r16at1f}
\int_{D}((\mathbf A(x))^{-1}\mathbf {p}_1(x),\mathbf q_2(x))_{\mathbb
R^n}dx\\- \int_{D}p_2(x)\mbox{\rm div\,}\mathbf q_2(x)\,dx=0
 \quad
\forall \mathbf q_2\in H(\mbox{\rm div},D),
\end{multline}
\begin{multline} \label{r15aat1f}
\int_{D}v_2(x)\mbox{\rm div\,}\mathbf p_1(x)\,dx
+\int_{D}c(x)p_2(x)v_2(x)\,dx\\=\int_{D}v_2(x)Q^{-1}(l_0-\hat z_2(\cdot))(x)\,dx \quad \forall v_2\in
L^2(D),
\end{multline}
where $\hat{\mathbf z}_1\in H(\mbox{\rm
div},D).$ Problem
\eqref{r13at1f}--\eqref{r15aat1f} is uniquely solvable.
The error of estimation $\sigma$ is given by the expression
\begin{equation}\label{rllf}
\sigma=\left(l(Q^{-1}(
l_0-\hat z_2))\right)^{1/2}.
\end{equation}
\end{pred}
\begin{proof}
Show that the solution to the optimal control problem
(\ref{113f})--(\ref{m20f}) can be reduced to the solution
of system (\ref{r13at1f})-(\ref{r15aat1f}).

First, we note that fuctional
$I(u),$ defined by \eqref{m20f}, can be represented in the form
\begin{equation}\label{qfunkf}
I(u)=\tilde I(u)-L(u) +\int_{D}Q^{-1}l_0(x)l_0(x) \,dx,
\end{equation}
where
\begin{equation*} \tilde I(u)=\int_{D}Q^{-1}z_2(x;u)z_2(x;u) \,dx
+ (\tilde
Q_1^{-1}u_1,u_1)_{H_1}+(\tilde Q_2^{-1}u_2,u_2)_{H_2}
\end{equation*}
is a quadratic form corresponding
to a symmetric continuous bilinear form

$$
\pi(u,v):= \int_{D}Q^{-1}z_2(x;u)z_2(x;v) \,dx+ (\tilde
Q_1^{-1}u_1,v_1)_{H_1}+(\tilde Q_2^{-1}u_2,v_2)_{H_2},
$$
defined on $H\times H$ and
\begin{equation*}
L(u)= 2\int_{D} Q^{-1}z_2(x;u)l_0(x) \, dx
\end{equation*}
is a linear continuous functional defined on  $H$.

The representation of in the form \eqref{qfunkf} follows from the reasoning similar to that in the proof of Theorem 1
(replacing $\tilde z_2(x;u)$ by $z_2(x;u)$ and $z_2^{(0)}(x)$ by $l_0(x),$ correspondingly).

Since
$$
\tilde I(u)=\pi(u,u)\geq
(Q_1^{-1}u_1,u_1)_{H_1}+(\tilde Q_2^{-1}u_2,u_2)_{H_2}
\geq \alpha\|u\|_H^2 \quad \forall
u\in H,
$$
where $\alpha$ is a constant from \eqref{posdefme},
then the bilinear form  $\pi(u,v)$ and the linear functional
$L(u)$ satisfy the condition of Theorem 1.1 from \cite{BIBLlio}.
Therefore, by this theorem, there exists a unique element $
\hat{u}:=(\hat u_1,\hat u_2)
 \in H$ on which the minimum of the functional $I(u)$ is attained, i.e.
$I(\hat{u})=\inf_{u
\in H}I(u).$
This implies that for any fixed $
w=(w_1,w_2)\in H$ and $\tau \in \mathbb R$ the function $s(\tau):=I(\hat{
u}+\tau w)$ reaches its minimum at the point $\tau
=0,$ so that
\begin{equation} \label{z43f}
\frac {d}{d\tau} I(\hat {u}+\tau w)\left.
\right|_{\tau=0}=0.
\end{equation}
Taking into account that
$$z_2(x; \hat {u}+\tau w)=z_2(x;\hat u)+\tau z_2(x;w),$$
we obtain from \eqref{z43f}
$$
0=\frac{1}{2}\frac{d}{dt}I(\hat u+\tau
w)\Bigl.\Bigr|_{\tau=0}
$$
\begin{equation}\label{z43aaf}
=-(Q^{-1}(l_0-z_2(\cdot;\hat u)),z_2(\cdot;w ))_{L^2(D)}
+(\tilde Q_1^{-1}\hat u_1,w_1)_{H_1}+(\tilde Q_2^{-1}\hat u_2,w_2)_{H_2}.
\end{equation}

Further, introducing a pair of functions $(\mathbf {p}_1,p_2)\in H(\mbox{\rm
div},D)\times L^2(D)$ as a unique solution of the problem
\begin{multline}\label{r16at1rf}
\int_{D}((\mathbf A(x))^{-1}\mathbf {p}_1(x),\mathbf q_2(x))_{\mathbb
R^n}dx\\- \int_{D}p_2(x)\mbox{\rm div\,}\mathbf q_2(x)\,dx=0
 \quad
\forall \mathbf q_2\in H(\mbox{\rm div},D),
\end{multline}
\begin{multline} \label{r15aat1rf}
\int_{D}v_2(x)\mbox{\rm div\,}\mathbf p_1(x)\,dx+\int_{D}c(x)p_2(x)v_2(x)\,dx
\\=\int_{D}v_2(x)Q^{-1}(l_0-z_2(\cdot;\hat u))(x)\,dx \quad \forall v_2\in
L^2(D)
\end{multline}
%Тоді перший доданок в лівій частині (\ref{z43a}) можна перетворити
%наступним чином:
and reasoning analogously as in the proof of Theorem 1, we arrive at the following relation
$$
-(Q^{-1}(l_0-z_2(\cdot;\hat u)),z_2(\cdot;w ))_{L^2(D)}
=-(w_1,C_1\mathbf {p}_1)_{H_1}-(w_2,C_2p_2)_{H_2}.
$$
By using (\ref{z43a}), we find from the latter equality
$$
(w_1,C_1\mathbf {p}_1)_{H_1}+(w_2,C_2p_2)_{H_2}=
(\tilde Q_1^{-1}\hat u_1,w_1)_{H_1}+(\tilde Q_2^{-1}\hat u_2,w_2)_{H_2},
$$
Whence, it follows that
$
\hat u_1=\tilde Q_1C_1\mathbf {p}_1,\quad\hat u_2=\tilde Q_2C_2p_2.
$
Substituting these expressions into
\eqref{113f} and \eqref{115f}
%$u_1$ and $u_2$ відповідно на знайдені вирази $\hat u_1$ i $\hat u_2$ і,
and setting $\mathbf z_1(x;\hat u)=:\hat {\mathbf z}_1(x),$
$z_2(x;\hat u)=:\hat z_2(x),$ we establish
that functions
$\hat {\mathbf z}_1,$ $\hat z_2$ and $\mathbf p_1,$ $p_2$ satisfy system of variational equations
\eqref{r13at1f}--\eqref{r15aat1f} and the validity of equalities \eqref{haturr}, \eqref{haturr1}. The unique solvability of this system follows from the existence of the unique minimum point
$\hat u$ of functional $I(u)$.

Now let us find the error of estimation. From (\ref{m20f}) at $u=\hat u$ and
(\ref{haturr1}), it follows
$$
\sigma^2=I(\hat u)
=(Q^{-1}(l_0-z_2(\cdot;\hat u)),l_0-
z_2(\cdot;\hat u))_{L^2(D)}
$$
$$
+ (\tilde
Q_1^{-1}\hat u_1,\hat u_1)_{H_1}+(\tilde Q_2^{-1}\hat u_2,\hat u_2)_{H_2}
$$
\begin{equation}\label{z43afff}
=(Q^{-1}(l_0-\hat z_2),l_0-
\hat z_2)_{L^2(D)}+(C_1\mathbf {p}_1,\tilde Q_1C_1\mathbf {p}_1)_{H_1}
+(C_2p_2,\tilde Q_2C_2p_2)_{H_2}.
\end{equation}
Setting in \eqref{r16at1f} and \eqref{r15aat1f}
$\mathbf q_2=\mathbf {\hat z}_1$ and $v_2=\hat z_2,$
we find
$$
\int_{D}((\mathbf A(x))^{-1}\mathbf {p}_1(x),\mathbf {\hat z}_1(x))_{\mathbb
R^n}dx- \int_{D}p_2(x)\mbox{\rm div\,}\mathbf {\hat z}_1(x)\,dx=0,
$$
$$
\int_{D}\hat z_2(x)\mbox{\rm div\,}\mathbf p_1(x)\,dx+\int_{D}c(x)p_2(x)\hat z_2(x)\,dx
=\int_{D}\hat z_2(x)Q^{-1}(l_0-\hat z_2)(x)\,dx.
$$
Setting in equations \eqref{r13at1f} and \eqref{r14at1f} $\mathbf q_1=\mathbf p_1$
and $v_1=p_2,$ we derive from two latter relations
$$
(Q^{-1}(l_0-\hat z_2),
l_0-\hat z_2)_{L^2(D)}=(Q^{-1}l_0,
l_0-\hat z_2)_{L^2(D)}-\int_{D}\hat z_2(x)\mbox{\rm div\,}\hat{\mathbf p}_1(x)\,dx
$$
$$
-\int_{D}c(x)p_2(x)\hat z_2(x)\,dx+\int_{D}((\mathbf A(x))^{-1}\mathbf {p}_1(x),\mathbf {\hat z}_1(x))_{\mathbb
R^n}dx- \int_{D}p_2(x)\mbox{\rm div\,}\mathbf {\hat z}_1(x)\,dx
$$
$$
=(l_0,Q^{-1}(
l_0-\hat z_2))_{L^2(D)}+\int_{D}(((\mathbf A(x))^{-1})^T\mathbf {\hat z}_1(x),\mathbf {p}_1(x))_{\mathbb
R^n}dx-\int_{D}z_2(x)\mbox{\rm div\,}\mathbf p_1(x)\,dx
$$
$$
- \int_{D}p_2(x)\mbox{\rm div\,}\mathbf {\hat z}_1(x)\,dx-\int_{D}c(x)p_2(x)\hat z_2(x)\,dx
=(l_0,Q^{-1}(
l_0-\hat z_2))_{L^2(D)}
$$
\begin{equation*}
-\int_{D}(C_1^tJ_{H_1}\tilde Q_1C_1\mathbf
{p}_1(x),\mathbf p_1(x))_{\mathbb R^n}\,dx-\int_{D}C_2^tJ_{H_2}\tilde Q_2C_2
p_2(x)p_2(x)\,dx
\end{equation*}
\begin{equation*}
=l(Q^{-1}(
l_0-\hat z_2))-(C_1\mathbf {p}_1,\tilde Q_1C_1\mathbf {p}_1)_{H_1}
-(C_2p_2,\tilde Q_2C_2p_2)_{H_2}.
\end{equation*}
From here and
\eqref{z43afff}, it follows representation \eqref{rllf} for the estimation error.
\end{proof}
In the following theorem we obtain another representation
for the guaranteed estimate $\widehat{\widehat {l(f)}}$ of
quantity $l(f)$ similar to \eqref{Altgxx}.
\begin{pred}\label{t6}
The guaranteed estimate of $l(f)$ has the form
\begin{equation}\label{Altf}
\widehat{\widehat {l(f)}}=l(\hat f),
\end{equation}
where
$\hat f(x)=f_0(x)-Q^{-1}\hat p_2(x)$ and
%випадкові поля $\hat{\mathbf j} \in L^2(\Omega,H(\mbox{\rm
%div},D))$ і $\hat\varphi\in L^2(\Omega,L^2(D)),$ a також допоміжні
%and random field
%$\hat{\mathbf p}_1\in L^2(\Omega,H(\mbox{\rm
%div},D))$ і
$\hat p_2\in L^2(\Omega,L^2(D))$ is determined from solution
of problem \eqref{13aaag}--\eqref{18aaagq}.
%\begin{equation*}
%\int_{D}(((\mathbf A^{-1}(x))^{-1})^T\hat{\mathbf p}_1(x),\mathbf
%q_1(x))_{\mathbb R^n}dx- \int_{D}\hat p_2(x)\mbox{\rm
%div\,}\mathbf q_1(x)\,dx
%\end{equation*}
%\begin{equation} \label{13aaaf}
%=\int_{D}(C_1^tJ_{H_1}\tilde
%Q_1(y_1-C_1\hat{\mathbf
%j})(x),\mathbf q_1(x))_{\mathbb R^n}dx
%\,\,\forall\mathbf q_1\in H(\mbox{\rm div},D),
%\end{equation}
%\begin{multline} \label{p14aaaf}
%-\int_{D}v_1(x)\mbox{\rm div\,}\hat{\mathbf p}_1(x)\,dx
%-\int_{D}c(x)\hat p_2(x)v_1(x)\,dx\\=\int_D(C_2^tJ_{H_2}\tilde
%Q_2(y_2-C_2\hat
%\varphi)(x)v_1(x)\,dx\quad\forall v_1\in L^2(D),
%\end{multline}
%\begin{multline}\label{pr16at1f}
%\int_{D}((\mathbf A(x))^{-1}\,\hat{\mathbf j}(x),\mathbf
%q_2(x))_{\mathbb R^n}\,dx\\- \int_{D}\hat\varphi(x)\mbox{\rm
%div\,}\mathbf q_2(x)\,dx=0\quad
%\forall \mathbf q_2\in H(\mbox{\rm div},D),
%\end{multline}
%\begin{multline} \label{18aaaf}
%-\int_{D}v_2(x)\mbox{\rm div\,}\hat{\mathbf j}(x)\,dx
%-\int_{D}c(x)\hat\varphi(x)v_2(x)
%\\=\int_{D}v_2(x)(Q^{-1}\hat p_2(x)-f_0(x))\,dx \quad
%\forall v_2\in L^2(D),
%\end{multline}
%where equalities \eqref{13aaaf} -- \eqref{18aaaf} are fulfilled with probability
%$1.$ Problem \eqref{13aaaf} -- \eqref{18aaaf} is uniquely
%solvable.

%The random fields
%$\hat{\mathbf j}$, $\hat{\mathbf p}_1$ and $\hat\varphi$,  $\hat p_2,$ whose realizations
%$\hat{\mathbf p}(x,t)$ and $\hat{\boldsymbol{\varphi}}(x,t)$
%satisfy problem \eqref{13aaaf}--\eqref{18aaaf},
%belong to the spaces
%$L^2(\Omega,H(\mbox{\rm
%div},D))$ and $L^2(\Omega,L^2(D)),$ respectively.
\end{pred}
%Доведення цієї теореми є аналогічним доведенню теореми 2.2.
\begin{proof}
%Доведемо справедливість представлення  $\widehat{\widehat
%{l(f)}}=l(\hat f).$
From (\ref{haturr}) and (\ref{haturr1}), we have
$$
\widehat{\widehat {l(f)}}=(y_1,\hat u_1)_{H_1}+(y_2,\hat
u_2)_{H_2}+\hat c
$$
\begin{equation}\label{lpt}
=(y_1,\tilde Q_1C_1\mathbf
p_1)_{H_1}+(y_2,\tilde Q_2C_2
p_2)_{H_2}+
(l_0-\hat z_2,f_0)_{L^2(D)}.
\end{equation}
Putting in
%Oскільки функції $\hat{\mathbf p}_1\in H(\mbox{\rm div},D)$ i $\hat p_2\in L^2(D)$ задовольняють тотожностям
(\ref{13aaag}) and (\ref{p14aaag}),
$\mathbf q_1=\mathbf p_1$ and $v_1=p_2,$ respectively, we come to the relations
\begin{multline} \label{13aaat}
\int_{D}(((\mathbf A(x))^{-1})^T\hat{\mathbf p}_1(x),\mathbf
p_1(x))_{\mathbb R^n}dx- \int_{D}\hat p_2(x)\mbox{\rm
div\,}\mathbf p_1(x)\,dx\\ =\int_{D}(C_1^tJ_{H_1}\tilde
Q_1(y_1-C_1\hat{\mathbf
j})(x),\mathbf p_1(x))_{\mathbb R^n}dx,
\end{multline}
\begin{multline} \label{p14aaat}
-\int_{D}p_2(x)\mbox{\rm div\,}\hat{\mathbf p}_1(x)\,dx
-\int_{D}c(x)\hat p_2(x)p_2(x)\,dx
\\=\int_D(C_2^tJ_{H_2}\tilde
Q_2(y_2-C_2\hat
\varphi)(x)p_2(x)\,dx,
\end{multline}
Putting $\mathbf q_2=\hat{\mathbf p}_1$ and $v_2=\hat p_2$ in
%З іншого боку, завдяки тому, що функції $\mathbf {p}_1$ i $p_2$ задовольняють тотожностям
\eqref{r16at1f} and \eqref{r15aat1f},
we have
\begin{equation}\label{r16at1zt}
\int_{D}((\mathbf A(x))^{-1}\,\mathbf {p}_1(x),\hat{\mathbf p}_1(x))_{\mathbb
R^n}dx- \int_{D}p_2(x)\mbox{\rm div\,}\hat{\mathbf p}_1(x)\,dx=0,
\end{equation}
\begin{multline} \label{r15aat1zt}
-\int_{D}\hat p_2(x)\mbox{\rm div\,}\mathbf p_1(x)\,dx
-\int_{D}c(x)p_2(x)\hat p_2(x)\,dx
\\=-\int_{D}\hat p_2(x)Q^{-1}(l_0-\hat z_2)(x)\,dx.
\end{multline}
Relations \eqref{13aaat}--\eqref{r15aat1zt},
and \eqref{lpt} imply
%$$
%\widehat{\widehat
%{l(f)}}-(l_0-\hat z_2,f_0)_{L^2(D)}-(C_1\hat{\mathbf
%j},\tilde Q_1C_1\mathbf p_1)_{H_1}
%-(C_2\hat\varphi,\tilde Q_2C_2 p_2)_{H_2}
%$$
%$$
%=\int_{D}\hat p_2(x)Q^{-1}(l_0-\hat z_2)(x)\,dx,
%$$
%звідки
\begin{equation}\label{g2hat2t}
\widehat{\widehat
{l(f)}}=(C_1\hat{\mathbf
j},\tilde Q_1C_1\mathbf p_1)_{H_1}
+(C_2\hat\varphi,\tilde Q_2C_2p_2)_{H_2}
\\-(Q^{-1}\hat p_2-f_0,(l_0-\hat z_2)_{L^2(D)}.
\end{equation}
Setting $\mathbf q_2=\hat{\mathbf z}_1,$ $v_2=\hat z_2$ and $\mathbf q_1=\hat{\mathbf j},$
$v_1=\hat \varphi$ in equations \eqref{pr16at1g'}, \eqref{18aaagq} and \eqref{r13at1f}, \eqref{r14at1f}, respectively, we obtain
\begin{equation}\label{pr16at1yt}
\int_{D}((\mathbf A(x))^{-1}\hat{\mathbf j}(x),\hat{\mathbf
z}_1(x))_{\mathbb R^n}\,dx- \int_{D}\hat\varphi(x)\mbox{\rm
div\,}\hat{\mathbf z}_1(x)\,dx=0,
\end{equation}
\begin{equation} \label{18aaayt}
-\int_{D}\hat z_2(x)\mbox{\rm div\,}\hat{\mathbf j}(x)\,dx
-\int_{D}c(x)\hat\varphi(x)\hat z_2(x)\,dx
=\int_{D}\hat z_2(x)(Q^{-1}\hat p_2(x)-f_0(x))\,dx,
\end{equation}
and
\begin{multline} \label{r13at1yt}
\int_{D}(((\mathbf A(x))^{-1})^T\hat{\mathbf z}_1(x),\hat{\mathbf j}(x))_{\mathbb R^n}dx- \int_{D}\hat z_2(x)\mbox{\rm
div\,}\hat{\mathbf j}(x)\,dx
\\=-\int_{D}(C_1^tJ_{H_1}\tilde Q_1C_1\mathbf
{p}_1(x),\hat{\mathbf j}(x))_{\mathbb R^n}\,dx,
\end{multline}
\begin{equation} \label{r14at1yt}
-\int_{D}\hat \varphi(x)\mbox{\rm div\,}\hat{\mathbf z}_1(x)\,dx
-\int_{D}c(x)\hat z_2(x)\hat \varphi(x)\,dx
=-\int_{D}(C_2^tJ_{H_2}\tilde Q_2C_2
p_2(x)\hat \varphi(x)\,dx.
\end{equation}
From \eqref{pr16at1yt} and \eqref{r14at1yt}, we deduce
$$
\int_{D}\hat z_2(x)(Q^{-1}\hat p_2(x)-f_0(x))\,dx
=-(\tilde Q_1C_1\hat{\mathbf p}_1,C_1\hat{\mathbf j}(x))_{H_1}
-(\tilde Q_2C_2\hat p_2,C_2\hat \varphi(x))_{H_2},
$$
whence,
by virtue of \eqref{g2hat2t},
it follows represetation \eqref{Altf}.
\end{proof}
{\bf Remark 2}.
Notice that in representation $l(\hat{f})$
for minimax estimate $\widehat{\widehat {l(f)}}$ the function $\hat{f}(x)=f_0(x)-Q^{-1}\hat p_2(x),
$ where $\hat p_2$ is defined from equations \eqref{13aaag}--\eqref{18aaagq}, can be taken as a good estimate
for unknown function $f$ entering the right-hand side of equation (\ref{eqsystem2}) (for explanations, see Remark 1).

%На останок зауважимо, що користуючись запропонованими в \cite{Brezzi} змішаними методами скінченних елементів,

\section{Approximate guaranteed estimates of linear functionals from right-sides of elliptic equations}
%{\bf 9 Approximate guaranteed estimates of linear functionals from right-sides of elliptic equations}

%In this section we elaborate algorithms of approximate solving of problems \eqref{r13at1f}--\eqref{r15aat1f} and \eqref{13aaaf}--\eqref{18aaaf}
%через розв'язки яких визначається мінімаксна оцінка.

In this section we introduce the notion of approximate guaranteed estimates of $l(\mathbf j,\varphi)$ and prove their convergence to
$\widehat{\widehat {l(\mathbf j,\varphi)}}$.

Futher, as in section 6, the domain $D$ is supposed to be bounded and connected domain of $\mathbb R^n$
with
polyhedral boundary $\Gamma.$

Take an approximate minimax estimate of $l(f)$ as
\begin{equation}\label{ttrrh}
\widehat {l^{h}(f)}=(u_1^{h},y_1)_{H_1}+(u_2^{h},y_2)_{H_1}+c^{h},
\end{equation}
where $u_1^{h}=\tilde Q_1C_1\mathbf
{p}_1^{h},$ $u_2^{h}=\tilde Q_2C_2 p_2^{h},$ $c^{h}=\int_{D}
(l_0(x)-\hat z_2^{h}(x))f_0(x)\,dx,$
 and functions $\mathbf {\hat z}_1^{h},\mathbf {p}_1^{h}\in V_1^{h}$ and $\hat z_2^{h} ,p_2^{h}\in V_2^{h}$
are determined\footnote{The spaces $V_1^{h}$ and $V_2^{h}$ are described on page \pageref{page25}.} from the following uniquely solvable system of variational equalities:
\begin{multline} \label{r13at1hh}
\int_{D}(((\mathbf A(x))^{-1})^T\,\hat{\mathbf z}^{h}_1(x),\mathbf
q^{h}_1(x))_{\mathbb R^n}dx- \int_{D}\hat z^{h}_2(x)\mbox{\rm
div\,}\mathbf q^{h}_1(x)\,dx
\\=-\int_{D}(C_1^tJ_{H_1}\tilde Q_1C_1\mathbf
{p}^{h}_1(x),\mathbf q^{h}_1(x))_{\mathbb R^n}\,dx \quad\forall\mathbf
q^{h}_1\in V_1^{h},
%H(\mbox{\rm div},D),
\end{multline}
\begin{multline} \label{r14at1hh}
\int_{D}v^{h}_1(x)\mbox{\rm div\,}\hat{\mathbf z}^{h}_1(x)\,dx
+\int_{D}c(x)\hat z^{h}_2(x)v^{h}_1(x)\,dx
\\=\int_{D}(C_2^tJ_{H_2}\tilde Q_2C_2
p^{h}_2(x)v^{h}_1(x)\,dx\quad\forall v^{h}_1\in V_2^{h},
%L^2(D),
\end{multline}
\begin{equation}\label{r16at1hh}
\int_{D}((\mathbf A(x))^{-1}\mathbf {p}^{h}_1(x),\mathbf q^{h}_2(x))_{\mathbb
R^n}dx- \int_{D}p^{h}_2(x)\mbox{\rm div\,}\mathbf q^{h}_2(x)\,dx=0
 \quad
\forall \mathbf q^{h}_2\in V_1^{h},
%H(\mbox{\rm div},D),
\end{equation}
\begin{multline} \label{r15aat1hh}
\int_{D}v^{h}_2(x)\mbox{\rm div\,}\mathbf p^{h}_1(x)\,dx
+\int_{D}c(x)p^{h}_2(x)v^{h}_2(x)\,dx
\\=\int_{D}v^{h}_2(x)Q^{-1}(l_0-\hat z^{h}_2(\cdot))(x)\,dx \quad \forall v^{h}_2\in V_2^{h}.
%L^2(D).
\end{multline}
The quantity $\sigma^{h}=(I(u^{h}))^{1/2},$
where
$$
I(u^{h})
=(Q^{-1}(l_0-\hat z_2^{h}),l_0-
\hat z_2^{h})_{L^2(D)}+ (\tilde
Q_1^{-1}u^{h}_1,u^{h}_1)_{H_1}+(\tilde Q_2^{-1}u^{h}_2,u^{h}_2)_{H_2},
$$
is called the approximate error of the guaranteed estimation of $l(f)$.

\begin{pred}
Approximate guaranteed estimate $\widehat {l^{h}(f)}$ of $l(f)$
which is defined by \eqref{ttrrh}
can be represented in the form
 $
 \widehat {l^{h}(f)}=l(\hat f^{h})$,
where $\hat f^{h}=f_0(x)-Q^{-1}\hat p_2^{h}(x),$ and function $\hat p_2^{h}\in Q^{h}$ is determined from solution of problem
 \eqref{13aaagg}--\eqref{18aaagg}.
Approximate error of estimation has the form
$$
\sigma^{h}=\left(l(Q^{-1}(
l_0-\hat z_2^{h}))\right)^{1/2}.
$$
In addition,
$$
\lim_{h\to 0}\mathbb E|\widehat {l^{h}(f)}-\widehat{\widehat {l(f)}}|^2=0,
\quad
\lim_{h\to \infty}\sigma^{h}=\sigma,
$$
and
$$
\|\hat{\mathbf z}_1-
\hat{\mathbf z}_1^{h}\|_{H(\mbox{\rm \small div},D)}
+\|\hat z_2-\hat z_2^{h}\|_{L^2(D)}\to 0 \quad\mbox{\rm as}\quad h\to 0,
$$
$$
\|\mathbf p_1-\mathbf p_1^{h}\|_{H(\mbox{\rm \small div},D)}
+\|p_2-p_2^{h}\|_{L^2(D)}\to 0 \quad\mbox{\rm as}\quad h\to 0,
$$
$$
\|\hat{\mathbf j}-
\hat{\mathbf j}^{h}\|_{H(\mbox{\rm \small div},D)}
+\|\hat\varphi-\hat{\varphi}^{h}\|_{L^2(D)}\to 0 \quad\mbox{\rm as}\quad h\to 0,
$$
$$
\|\hat{\mathbf p}_1-\hat{\mathbf p}_1^{h}\|_{H(\mbox{\rm \small div},D)}
+\|\hat p_2-\hat p_2^{h}\|_{L^2(D)}\to 0 \quad\mbox{\rm as}\quad h\to 0.
$$
\end{pred}
\begin{proof}
The proof of this theorem is similar to the proofs of Theorems 4 and 5.
\end{proof}

 System of linear algebraic equations with respect to coefficients of expansions \eqref{elr}, \eqref{elr1} of functions $\hat{\mathbf z}_1^{h}$, $\hat z_2^{h}$, $\mathbf p_1^{h},$ and $p_2^{h}$, analogous to \eqref{r13at1hggtttt}--\eqref{r15aat1hggtttt}, can be also obtained for problem \eqref{r13at1hh}--\eqref{r15aat1hh}.
\section{Corollary from the obtained results}
%{\bf 10 Corollary from the obtained results}
Note in conclusion that the above results generalize, for the class of estimation
problems for systems described by boundary value problems considered in this
work, the results by A. G. Nakonechnyi
\cite{BIBLnak1}, \cite{nak22}.

To do this, suppose, as in these papers, that from observations of random variable of the form
\begin{equation}\label{observnak}
y_2=C_2\varphi+\eta_2,
\end{equation}
it is necessary to estimate the expression
\begin{equation}\label{linfnak}
l(\varphi):=  \int_{D}l_2(x)\varphi(x)\, dx
\end{equation}
in the class of estimates of the form
\begin{equation}\label{clasnak}
\widehat{l(\varphi)}:=
(y_2,u_2)_{H_2}+c,
\end{equation}
where $\varphi$ is a solution to the problem \eqref{1}, \eqref{3a}, $l_2$
is a given function from $L^2(D),$ $u_2\in
H_2,$ $c\in \mathbb R,$
$C_2\in\mathcal L(L^2(D),H_2)$ is a linear operator.
%\begin{equation}\label{etaG_1}
%\eta_2\in G_1.
%\end{equation}
%a через $G_1$ позначено множину $\tilde\eta_2$
%випадкових елементів
%  $\tilde
%\eta_2\in L^2(\Omega,H_2)$ з нульовими середніми, що задовольняють
%умову
%\begin{equation}\label{restr2nak}
%E(\tilde Q_2\tilde \eta_2,\tilde \eta_2)_{H_2} \leq 1,
%\end{equation} в якій
%$\tilde Q_2$ -- заданный в $H_2$ ограниченный самосопряженный положительно определенный оператор, который имеет ограниченный обратный.

The case considered here corresponds to setting $C_1=0,$ $\eta_1=0,$ $\mathbf l_1=0,$ $u_1=0,$ respectively in
\eqref{observ}, \eqref{linf}, \eqref{clas}, \eqref{restr2},
and Lemma 1 can be stated as follows.
\begin{predl}
Finding the minimax estimate of
$l(\varphi)$ is equivalent to the problem of optimal control of the system described by
the equations
\begin{multline}\label{113c}
\int_{D}(((\mathbf A(x))^{-1})^T\mathbf z_1(x;u),\mathbf
q(x))_{\mathbb R^n}dx- \int_{D}z_2(x;u)\mbox{\rm div\,}\mathbf
q(x)\,dx=0\\
%\int_{D}(\mathbf l_1(x)-(C_1^tJ_{H_1}u_1)(x),\mathbf q(x))_{\mathbb
%R^n}dx
\quad \forall \mathbf q\in H(\mbox{\rm div},D),
\end{multline}
\begin{multline}\label{115c}
-\int_{D}v(x)\mbox{\rm div\,}\mathbf
z_1(x;u)\,dx-\int_{D}c(x)z_2(\cdot;u)v(x)\,dx\\=\int_{D}(l_2(x)-(C_2^tJ_{H_2}u_2)(x))v(x)\,dx\quad
\forall v\in L^2(D)
\end{multline}
with the cost function
\begin{equation}\label{m20c}
I(u)=(Q^{-1}z_2(\cdot;u),
z_2(\cdot;u))_{L^2(D)}
+(\tilde Q_2^{-1}u_2,u_2)_{H_2}\!\to
\inf_{u\in H_2}.
\end{equation}
\end{predl}

It is easy to see that the second component $z_2(\cdot;u)$ of the solution $(\mathbf z_1(\cdot;u),z_2(\cdot;u))$ to this problem
belongs to the space $H_0^1(D)$ and is a weak solution to problem \eqref{1}--\eqref{3a}, i.e. it satisfies the integral identity
\begin{multline}\label{m20cd}
-(\mathbf A^T\,\mbox{\bf grad\,}z_2,\mbox{\bf grad\,}v)_{L^2(D)^n}-(cz_2,v)_{L^2(D)}\\=((l_2-(C_2^tJ_{H_2}u_2),v)_{L^2(D)}\quad \forall v\in H_0^1(D).
\end{multline}
Therefore, Lemma 3 takes the form:
\begin{predl}
Finding the minimax estimate of
$l(\varphi)$ is equivalent to the problem of optimal control of the system described by
equation \eqref{m20cd}
with the cost function
\eqref{m20c}.
\end{predl}
Theorems 1 and 2 are transformed into the following assertions.
\begin{pred}
There exists a unique minimax estimate of $l(\mathbf {j},\varphi)$
which has the form
\begin{equation}\label{exactest1}
\widehat{\widehat {l(\mathbf {j},\varphi)}}=(y_2,\hat u_2)_{H_2}+\hat c,
\end{equation}
where
\begin{equation}\label{hat1}
\hat c=-\int_{D}\hat
z_2(x)f_0(x)\,dx,\quad \hat u_2= \tilde Q_2C_2 p_2,
\end{equation}
and the functions $\hat z_2$ and $p_2\in H^1_0(D)$ are determined from solution
of the following problem:
\begin{multline} \label{r13at11}
\!\!-\int_{D}(\mathbf A^T(x)\,\mbox{\bf grad\,}\hat z_2(x),\mbox{\bf grad\,}v_1(x))_{\mathbb R^n}\,dx-\int_{D}c(x)\hat z_2(x)v_1(x)\,dx\\
\!=\!\int_{D}(l_2(x)-C_2^tJ_{H_2}\tilde Q_2C_2
p_2(x))v_1(x)\,dx\,\,\,\forall v_1\in H_0^1(D),
\end{multline}
\begin{multline} \label{r15at11}
\!\!-\int_{D}(\mathbf A(x)\,\mbox{\bf grad\,}p_2(x),\mbox{\bf grad\,}v_2(x))_{\mathbb R^n}\,dx-\int_{D}c(x)p_2(x)v_2(x)\,dx
\\=\int_{D}v_2(x)Q^{-1}\hat z_2(x)\,dx \quad \forall v_2\in
H_0^1(D).
\end{multline}
Problem
\eqref{r13at11}--\eqref{r15at11} is uniquely solvable. The error of estimation $\sigma$ is given by the expression
\begin{equation}\label{rll1}
\sigma=l(p_2)^{1/2}.
\end{equation}
\end{pred}
\begin{pred}\label{t80}
The minimax estimate of $l({\mathbf j},\varphi)$ has the form
\begin{equation*}\label{Altg}
\widehat{\widehat {l({\mathbf j},\varphi)}}=l(\hat{\mathbf j},\hat \varphi)),
\end{equation*}
where the function
$\hat \varphi\in H_0^1(D)$ is determined from solution
of the following problem:
\begin{multline} \label{13aaag1}
\!\!-\int_{D}(\mathbf A^T(x)\,\mbox{\bf grad\,}\hat p_2(x),\mbox{\bf grad\,}v_1(x))_{\mathbb R^n}\,dx-\int_{D}c(x)\hat p_2(x)v_1(x)\,dx\\
=\int_DC_2^tJ_{H_2}\tilde
Q_2(y_2-C_2\hat
\varphi)(x)v_1(x)\,dx\quad\forall v_1\in H_0^1(D),
\end{multline}
\begin{multline} \label{18aaagq1}
-(\mathbf A^T(x)\,\mbox{\bf grad\,}\hat \varphi(x),\mbox{\bf grad\,}v_2(x))_{\mathbb R^n}\,dx-\int_{D}c(x)\hat\varphi(x)v_2(x)\,dx
\\=\int_{D}v_2(x)(Q^{-1}\hat p_2(x)-f_0(x))\,dx \quad
\forall v_2\in H_0^1(D),
\end{multline}
where equalities \eqref{13aaag1} and \eqref{18aaagq1} are fulfilled with probability
$1.$ Problem \eqref{13aaag1}, \eqref{18aaagq1} is uniquely
solvable.

The random fields
$\hat\varphi$ and $\hat p_2,$ whose realizations
%$\hat{\mathbf p}(x,t)$ and $\hat{\boldsymbol{\varphi}}(x,t)$
satisfy equations \eqref{13aaag1} and \eqref{18aaagq1},
belong to the space
$L^2(\Omega,H_0^1(D))$.
\end{pred}

\addcontentsline{toc}{section}{REFERENCES}
\renewcommand{\refname}
{\bf \begin{center} REFERENCES
\end{center}}

\end{document}